\documentclass[a4paper,12pt]{amsart}
\usepackage{mathrsfs}

\usepackage{amssymb}
\usepackage{latexsym}
\usepackage{amsfonts}
\usepackage{amsmath}
\usepackage{eucal}
\usepackage{bm}
\usepackage{bbm}
\usepackage{graphicx}
\usepackage[english]{varioref}
\usepackage[nice]{nicefrac}
\usepackage[all]{xy}
\usepackage{amsthm}

\newcommand{\Id}{\textrm{id}}

\def\i{^{-1}}
\def\ge{\geqslant}
\def\le{\leqslant}
\def\<{\langle}
\def\>{\rangle}

\def\stab{\text{stab}}
\def\dom{\text{dom}}

\def\a{\alpha}
\def\b{\beta}
\def\g{\gamma}
\def\G{\Gamma}
\def\d{\delta}
\def\D{\Delta}
\def\L{\Lambda}
\def\e{\epsilon}

\def\o{\omega}

\def\s{\sigma}
\def\t{\tau}
\def\th{\theta}
\def\k{\kappa}
\def\l{\lambda}
\def\z{\zeta}

\def\tPhi{\tilde \Phi}

\def\ZZ{\mathbb Z}
\def\AA{\mathbb A}
\def\NN{\mathbb N}
\def\QQ{\mathbb Q}
\def\FF{\mathbb F}
\def\RR{\mathbb R}

\def\PP{\mathbb P}

\def\ca{\mathcal A}

\def\cl{\mathcal L}

\def\co{\mathcal O}
\def\cp{\mathcal P}

\def\car{\mathcal R}

\def\cu{\mathcal U}

\def\car{\mathcal R}

\def\tta{{\tilde \alpha}}
\def\tu{\tilde u}

\def\tW{\tilde W}
\def\tw{\tilde w}

\theoremstyle{plain}
\newtheorem{thm}{Theorem}[section]
\newtheorem*{thm*}{Theorem}
 \newtheorem{prop}[thm]{Proposition}
 \newtheorem{lem}[thm]{Lemma}
 \newtheorem{cor}[thm]{Corollary}

\theoremstyle{definition}

\theoremstyle{remark}

\newtheorem*{rmk}{Remark}
\newtheorem*{claim*}{Claim}

\begin{document}

\title[]{Connected components of closed affine Deligne-Lusztig varieties in affine Grassmannians}

\author{Sian Nie}
\address{Institute of Mathematics, Academy of Mathematics and Systems Science, Chinese Academy of Sciences, 100190, Beijing, China}
\email{niesian@amss.ac.cn}

\thanks{This work is supported in part by the National Natural Science Foundation of China (No. 11321101)}

\begin{abstract}
We determine the set of connected components of closed affine Deligne-Lusztig varieties for hyperspecial maximal parahoric subgroups of unramified connected reductive groups. This extends the work by Viehmann for split reductive groups, and the work by Chen-Kisin-Viehmann on minuscule affine Deligne-Lusztig varieties.

\end{abstract}
\maketitle

\section{Introduction}
\subsection{} Let $\FF_q$ be a finite field with $q$ elements and let $\bold k$ be an algebraic closure of $\FF_q$. Let $\bold k [[t]]$ (resp. $\FF_q[[t]]$) be the ring of formal power series over $\bold k$ (resp. $\FF_q$), whose fractional field is denoted by $L$ (resp. $F$). Let $\s$ be the Frobenius automorphism of $L / F$.

Let $G$ be a (connected) reductive group over $\FF_q[[t]]$. We also denote by $\s$ the induced automorphism on $G(L)$. Since $G$ is unramified and quasi-split over $\FF_q[[t]]$, there exist a Borel subgroup $B \subseteq G$ and a maximal torus $T \subseteq B$ (over $\FF_q[[t]]$), which split over $\bold k[[t]]$. Denote by $Y$ the (absolute) cocharacter group of $T$ and by $Y^+ \subseteq Y$ the set of $G$-dominant cocharacters defined with respect to $B$. We have Cartan decomposition $G(L)=\sqcup_{\mu \in Y^+} K t^\mu K$, where $K=G(\bold k[[t]])$ is hyperspecial maximal parahoric subgroup of $G(L)$.

For $\l \in Y^+$ and $b \in G(L)$, the affine Deligne-Lusztig variety $X_{\l}^G(b)=X_\mu(b)$ is defined as $$X_\l(b)=\{g \in G(L) / K; g\i b \s(g) \in K t^\l K\}.$$ The closed affine Deligne-Lusztig variety $X_{\preceq \l}(b)$ is defined as $$X_{\preceq \l}(b)=\cup_{\l' \preceq \l} X_{\l'}(b).$$ Here $\l' \preceq \l$ means $K t^{\l'} K / K$ is contained in the closure of $K t^\l K / K$ in $G(L) / K$. Note that $X_\l(b)=X_{\preceq \l}(b)$ if $\mu$ is minuscule.

The main purpose of this paper is to determine the set $\pi_0(X_{\preceq \l}(b))$ of connected components of $X_{\preceq \l}(b)$. When $G$ is split over $\FF_q[[t]]$ or $\l$ is minuscule, this problem has been solved by Viehmann \cite{Vie} and by Chen-Kisin-Viehmann \cite{CKV} respectively. Their description of $\pi_0(X_{\preceq \l}(b))$ has essential applications on the connected components of unramified Rapoport-Zink spaces (see \cite{Ch} and \cite{CKV}) and on Langlands-Rapoport conjecture for mod $p$ points on Shimura varieties (see \cite{Ki}). In this paper, we finish the computation of $\pi_0(X_{\preceq \l}(b))$ in general case.

\subsection{} To state our main result, we need more notations. Let $$[b]=\{g\i b \s(g); g \in G(L)\}$$ be the $G(L)$-conjugacy class of $b$. Due to Kottwitz, $[b]$ is determined by two invariants. One is the Kottwitz point $\k_{[b]}^G$, the image of $b$ under the natural projection $\k_G: G(L) \to \pi_1(G)_\s:=\pi_1(G)/ (1-\s)(\pi_1(G))$. The other is the Newton point $\nu_{[b]}^G$, the $G$-dominant cochracter conjugate to the Newton cocharater of $b$. By \cite{KR} and \cite{Ga}, $X_{\preceq \l}(b) \neq \emptyset$ if and only if $\k_G(b)=\k_G(t^\l)$ and $\l^\diamond-\nu_{[b]}^G$ is a linear combination of coroots occurring in $B$ with nonnegative rational coefficients. Here $\l^\diamond$ is the average of $\s$-conjugates of $\l$.

From now on we assume $X_{\preceq \l}(b) \neq \emptyset$. Following \cite{CKV}, we say the pair $(\l, b)$ is irreducible with respect to the Hodge-Newton decomposition (HN irreducible for short) if the coefficient of any simple coroot in $\l^\diamond - \nu_{[b]}^G$ is strictly positive. Thanks to \cite[\S 2.4 \& \S 2.5]{CKV}, to study $\pi_0(X_{\preceq \l}(b))$, it suffices to consider the case that $G$ is simple and adjoint, and $(b, \l)$ is HN-irreducible. Thus the following main result of the paper, which is conjectured by Chen-Kisin-Viehmann, completes the computation of connected components of closed affine Deligne-Lusztig varieties.

\begin{thm} \label {main'}
Assume $G$ is adjoint and simple. If $(\l, b)$ is HN-irreducible, then the natural projection $\eta_G: G(L) \to \pi_1(G)$ induces a bijection $$\pi_0(X_{\preceq \l}(b)) \cong (\s-1)\i(\eta_G(t^\l)-\eta_G(b)) \subseteq \pi_1(G).$$
\end{thm}

The following is a byproduct of Theorem \ref{main'}, which holds without any restrictions on $G$, $\l$ and $b$.
\begin{thm}\label{main''}
Let $J_b(F)=\{g \in G(L); g\i b \s(g)=b\}$ be the $\s$-centralizer of $b$. Then $J_b(F)$ acts transitively on $\pi_0(X_{\preceq \l}(b))$ by left multiplication.
\end{thm}

\begin{rmk}
Theorem \ref{main'} and Theorem \ref{main''} are proved if $G$ splits over $\bold k[[t]]$ (see \cite{Vie}) or $\l$ is minuscule (see \cite{CKV}).
\end{rmk}

The proofs of the main results follow the strategy of \cite{CKV}. Firstly, we prove Theorem \ref{main'} when $b$ is superbasic, that is, $b$ is not $\s$-conjugate to an element of any proper Levi subgroup of $G$. As a consequence, $J_b(F)$ acts on $\pi_0(X_{\preceq \l}(b))$ transitively. Secondly, we find a suitable standard parabolic subgroup $P=MN \subseteq G$ and a superbasic $\s$-conjugacy class $[b']_M \subseteq [b]$ of $M(L)$ with $\nu_{[b']_M}^M=\nu_{[b]}^G$ such that each connected component of $X_{\preceq \l}(b')$ contains a point in $(J_{b'}(F) \cap N(L))X_\mu^M(b')$, where $M \supseteq T$ (resp. $N$) is the Levi subgroup (resp. unipotent radical) of $P$ and $\mu$ ranges over the set $$\bar I_{\l, M, b'}=\{\mu \in Y; \mu \preceq \l, \text{ $\mu$ is $M$-minuscule and $\k_M(b')=\k_M(t^\mu)$}\}.$$ Thirdly, we show the image of the natural map $$\varphi_\mu: \pi_1(M)^\s \cong \pi_0(X_{\mu}^M(b')) \to \pi_0(X_{\preceq \l}(b'))$$ dose not depend on the choice of $\mu \in \bar I_{\l, M, b'}$. Here $\pi_1(M)^\s$ is the set of $\s$-fixed points on $\pi_1(M)$. From the first three steps, we see immediately $J_{b'}(F)$ acts on $\pi_0(X_{\preceq \l}(b'))$ transitively and Theorem \ref{main''} follows. To deduce Theorem \ref{main'} (i.e., $\pi_0(X_{\preceq \l}(b')) \cong \pi_1(G)^\s$), we need to show, under the assumption that $(\l, b)$ is HN-irreducible, that $\varphi_\mu$ is surjective, and moreover, any two points $P, P' \in X_\mu^M(b')$ are in the same connected component of $X_{\preceq \l}(b')$ if and only if $\eta_G(P)-\eta_G(P')$ lies in the kernel of the natural surjective homomorphism $\pi_1(M)^\s \twoheadrightarrow \pi_1(G)^\s$. This is the last step of the whole proof.

The first two steps are already established in \cite{CKV}. In this paper, we develop the techniques used in \cite{CKV} systematically and finish the last two steps. Based on Deligne-Lusztig reduction methods, we also provide a new conceptual proof for the first step, avoiding concrete computations in superbasic case.

Compared to the original proof for the case that $\l$ is minuscule, there are two main difficulties for the general case. The first lies in the third step. We need to construct affine lines in $X_{\preceq \l}(b')$ connecting $X_{\mu}^M(b')$ with $X_{\mu'}^M(b')$ for any $\mu, \mu' \in \bar I_{\l, M, b}$. When $\l$ is minuscule, the original construction of affine lines relies on the property that each element of $\bar I_{\l, M, b}$ is conjugate to $\l$ (under the Weyl group of $T$). In general case, $\bar I_{\l, M, b}$ becomes much more complicated. To overcome the difficulty, we introduce a new algorithm for the construction via defining the key set $\Theta(\mu, \mu', \l)$ (see \S 6 for definition). The other difficulty is to prove the sufficiency direction of the ``moreover" part in the last step. There is a gap in the original proof in \cite{CKV}. However, this gap can be filled in (for minuscule $\l$) by Miaofen Chen via a case-by-case analysis. For general case, we figure out a conceptual proof using weakly dominant cocharacters (see \S 4 for definition), extending Chen's arguments.

The paper is organized as follows. In \S 2, we collect basic notations and properties which are used frequently in the paper. In \S 3, we prove Theorem \ref{main'} for superbasic case. In \S 4, we show the existence of the Levi subgroup $M$ as above such that $\bar I_{\l, M, b'}$ contains a weakly dominant cocharacter. In \S 5, we outline the proof of Theorem \ref{main'}, where the surjectivity of $\varphi_\mu$ in the fourth step is redundant. In \S 6 and \S 7, we introduce the set $\Theta(\mu, \mu', \l)$ and establish the third step (see Proposition \ref{k1}). In \S 8 and \S 9, we finish the last step (see Proposition \ref{k2}).

\subsection*{Acknowledgement} We would like to thank Xuhua He for his preprint \cite{He1}. We are grateful to Miaofen Chen for sharing her essential ideas on the proof of Proposition \ref{k2}. Various definitions and arguments, especially in Section 5, 7, 8 and 9, are inspired or borrowed form the work of Chen-Kisin-Viehmann \cite{CKV}. Their work has a foundational influence on the present paper. Part of the work was done during the author's visit to the Institut Mittag-Leffler. We are grateful to the institute for the excellent working atmosphere.

\section{Preliminaries}
\subsection{}\label{setup1} Let $G \supseteq B \supseteq T$ be as above. We denote by $\car=(Y, \Phi_G^\vee, X, \Phi_G, \Pi_G)$ the root datum of $G$, where $X$ (resp. $Y$) is the absolute character (resp. cocharacter) group of $T$; $\Phi_G=\Phi \subseteq X$ (resp. $\Phi^\vee \subseteq Y$) is the set of roots (resp. coroots); $\Pi_G=\Pi_0 \subseteq \Phi$ is the set of simple roots occurring in $B$. Moreover, there exists a bijection $\a \mapsto \a^\vee$ between $\Phi$ and $\Phi^\vee$.

Let $N_T \subseteq G$ be the normalizer of $T$ in $G$. The quotients $W_0=N_T(L) / T(L)$ and $\tW=N_T(L) / T(\bold k[[t]])$ are called the Weyl group and the extended affine Weyl group of $G$ respectively. We have $$\tW_G=\tW=Y \rtimes W_0=\{t^\mu w; \mu \in Y, w \in W_0\}.$$ We can embed $\tW$ into the group of affine transformations of $Y_\RR$, where the action of $\tw=t^\mu w$ is given by $v \mapsto \mu+w(v)$. Here $W_0$ acts on $Y_\RR$ naturally by scalar extension.

Let $\tPhi_G=\tPhi=\Phi \times \ZZ$ be the set of (real) affine roots. Let $a=(\a, k) \in \tPhi$. We can view $a$ as an affine function such that $a(v)=-\<\a, v\>+k$ for $v \in Y_\RR$, where $\< , \>: X_\RR \times Y_\RR \to \ZZ$ is the scalar extension of the natural pairing between $X$ and $Y$. The induced action of $\tW$ on $\tPhi$ is given by $(\tw(a))(v)=a(\tw\i(v))$ for $\tw \in \tW$. Let $s_a=t^{k \a^\vee} s_\a \in \tW$ be the corresponding affine reflection. Then $\{s_a; a \in \tPhi\}$ generates the affine Weyl group $$W^a=\ZZ \Phi^\vee \rtimes W_0=\{t^\mu w; \mu \in \ZZ \Phi^\vee, w \in W_0\}.$$

Let $\Phi^+=\Phi \cap \NN \Pi_0$ be the set of positive roots and let $Y_\RR^+=\{v \in Y_\RR; \<\a, v\> \ge 0, \a \in \Phi^+\}$ be the set of dominant vectors. For $v \in Y_\RR$, we denote by $\bar v$ the unique dominant vector in the $W_0$-orbit of $v$. Let $$\D=\{v \in Y_\RR; 0 < \<\a, v\> < 1, \a \in \Phi^+\}$$ be the base alcove. Note that $\tW=W^a \rtimes \Omega$, where $\Omega=\{x \in \tW; x(\D)=\D\}$. Set $\tPhi^+=\{a \in \tPhi; a(\D) > 0\}$ and $\tPhi^-=-\tPhi^+$. Then $\tPhi=\tPhi^+ \sqcup \tPhi^-$. The associated length function $\ell: \tW \to \NN$ is defined by $\ell(\tw)=|\tPhi^- \cap \tw(\tPhi^+)|$. Let $S^a=\{s_a; a \in \tPhi, \ell(s_a)=1\}$ be the set of simple affine reflections and let $S_0=S^a \cap W_0$. Then $(W^a, S^a)$ and $(W_0, S_0)$ are Coxeter systems.

For $\tw, \tw' \in \tW$, we say $\tw \leq \tw'$ if there exists a sequence $\tw=\tw_1, \dots, \tw_r=\tw'$ such that $\ell(\tw_k) < \ell(\tw_{k+1})$ and $\tw_k \tw_{k+1}\i$ is an affine reflection for $k \in [1, r-1]$. We call this partial order $\leq$ the Bruhat order on $\tW$.

Let $J \subseteq S^a$. We denote by $W_J$ the parabolic subgroup of $W^a$ generated by $J$. Let $W \supseteq W_J$ be a subgroup of $\tW$. We set $W {}^J=\{\min (\tw W_J); \tw \in W\}$, ${}^J W=(W {}^J)\i$ and ${}^J W {}^J={}^J W \cap W {}^J$.

\subsection{}\label{setup2} Let $\s$ be the Frobenius automorphism of $G(L)$. We also denote by $\s$ the induced automorphism on the root datum $\car$. Then $\s$ acts on $Y_\QQ$ as a linear transformation of finite order. For $\tw \in \tW$, there exists $n \in \NN$ such that $(\tw\s)^n=t^{\xi}$ for some $\xi \in Y$. We define $\nu_{\tw}=\xi / n$, which dose not depend on the choice of $n$.

Let $\eta_G: G(L) \to \pi_1(G)=Y/\ZZ \Phi^\vee \cong \Omega$ be the natural homomorphism such that $\eta_G (K t^\mu K)=\mu \in Y/\ZZ \Phi^\vee$ for $\mu \in Y$. Let $\k_G: G(L) \to \pi_1(G)_\s$ be the Kottwitz map obtained by composing $\eta_G$ with the quotient map $\pi_1(G) \to \pi_1(G)_\s=\pi_1(G) / (1-\s)(\pi_1(G))$. Since $K$ lies in the kernel of $\eta_G$, we also denote by $\eta_G$ (resp. $\k_G$) the induced map on $\tW$ or on $G(L)/K$.

For $b \in G(L)$, we denote by $[b]=\{g\i b \s(g); g \in G(L)\}$ the $\s$-conjugacy class of $G(L)$ containing $b$. It is known that $[b] \cap N_T(L) \neq \emptyset$. Then the Kottwitz point $\k_{[b]}^G$ and the Newton point $\nu_{[b]}^G$ of $[b]$ are given by $\k(\t)$ and $\bar \nu_{p(\t)}$ respectively for any $\t \in [b] \cap N_T(L)$. Here $p: N_T(L) \to \tW$ is the natural projection. Note that $\s(\nu_{[b]}^G)=\nu_{[b]}^G$ since $\s([b])=[b]$. We say $b$ or $[b]$ is basic if $\nu_{[b]}^G$ is central on $\Phi$, that is, $\<\a, \nu_{[b]}^G\>=0$ for any $\a \in \Phi$. We say $b$ or $[b]$ is superbasic if there dose not exists a proper Levi subgroup $M \subseteq G$ (over $\FF_q [[t]]$) such that $[b] \cap M(L) \neq \emptyset$.

We say $\o \in \Omega$ is $\s$-superbasic in $\tW$ if each orbit of the action $s \mapsto \o \s(s) \o\i$ on $S^a$ is a union of connected components of the affine Dynkin diagram of $S^a$.
\begin{lem} \label{ss}
Let $b \in G(L)$. Then $[b]$ is a superbasic $\s$-conjugacy of $G(L)$ if and only if there exits $\t \in [b] \cap N_T(L)$ such that $p(\t) \in \Omega$ is $\s$-superbasic in $\tW$.
\end{lem}
\begin{proof}
It follows from \cite[Proposition 3.5]{HN1} and \cite[Lemma 3.1.1]{CKV}.
\end{proof}

\subsection{} \label{setup3} Let $M \subseteq G$ be a Levi subgroup containing $T$. Replacing the triple $T \subseteq B \subseteq G$ with $T \subseteq M \cap B \subseteq M$, we can define, as in \S \ref{setup1} and \S \ref{setup2}, $\Pi_M$, $\Phi_M$, $\tW_M$, $S_M^a$, $\leq_M$ and so on.

Let $J \subseteq S_0$ with $\s(J)=J$. We denote by $M_J \subseteq G$ the Levi subgroup generated by $T$ and the root subgroups $U_\a$ such that $s_\a \in W_J$. We write $\Pi_J=\Pi_{M_J}$, $\Phi_J=\Phi_{M_J}$, $\tW_J=\tW_{M_J}$ and so on for simplicity. Moreover, for $v \in Y_\RR$, we denote by $\bar v^J$ the unique $J$-dominant elements in the $W_J$-orbit of $v$.

For $\mu, \l \in Y$, we write $\mu \le_J \l$ if $\l-\mu \in \NN \Pi_J^\vee$ and write $\mu \preceq_J \l$ if $\bar \mu^J \le_J \bar \l^J$. If $J=S_0$, we abbreviate $\le_J$ (resp. $\preceq_J$) to $\le$ (resp. $\preceq$) for simplicity. The following two properties will be used frequently in the paper.

(a) $\mu \preceq \l$ if $\mu \preceq_J \l$;

(b) Let $\a \in \Phi_J$. Then $\mu-\a^\vee \preceq_J \mu$ if $\<\a, \mu\> \ge 1$. Moreover, $\mu-\a^\vee \prec_J \mu$ if $\<\a, \mu\> \ge 2$.

Let $K_J=M_J(\bold k[[t]])=K \cap M_J(L)$. Via the natural embedding $M_J(L) /K_J \hookrightarrow G(L) / K$, we can view $X_{\preceq_J \l}^{M_J}(b)$, the closed Deligne-Lusztig variety define with respect to $M_J$, as a closed subset of $X_{\preceq \l}(b)$ for $b \in M_J(L)$.

\subsection{} \label{setup4} Let $J \subseteq S_0$ with $\s(J)=J$ and $b \in G(L)$. We say the pair $(J, b)$ is admissible if the following two statements holds:

(a) $\nu_{[b]}^G$ is central on $\Phi_J$, that is, $\<\a, \nu_{[b]}^G\>=0$ for any $\a \in \Phi_J$;

(b) $[b] \cap M_J(L)$ contains a (unique) $\s$-conjugacy class $[b]_{J, \dom}$ of $M_J(L)$ such that $\nu_{[b]_{J, \dom}}^{M_J}=\nu_{[b]}^G$.

Assume $(J, b)$ is admissible. We define $$ \bar I_{\l, M_J, b}=\bar I_{\l, J, b}=\{x \in \pi_1(M_J); \k_J(x)=\k_{[b]_{J, \dom}}^{M_J}, \mu_x \preceq \l\}.$$ Here $\mu_x \in Y$ is the unique $J$-dominant and $J$-minuscule coweight such that $\eta_J(t^{\mu_x})=\eta_{M_J}(t^{\mu_x})=x$. Via the bijection $x \mapsto \mu_x$, we also view $\bar I_{\l, J, b}$ as the set of $J$-dominant $J$-minuscule cocharacters $\mu$ such that $\mu \preceq \l$ and $\k_J(t^\mu)=\k_J(b)$. Furthermore, we say $b$ is superbasic in $M_J$ if $[b]_{J,\dom}$ is a superbasic $\s$-conjugacy class of $M_J(L)$.

For $x \in \pi_1(M_J)$, we denote by $b_x=t^{\mu_x} w_x \in \Omega_J$ the unique element such that $\eta_J(b_x)=x$, where $w_x = u_x w_J$ and $u_x$ (resp. $w_J$) is the unique maximal element in parabolic subgroup generated by $\{s \in J; s(\mu_x)=\mu_x\}$ (resp. $J$).

\begin{lem} \label{dominant}
Let $J \subseteq S_0$ with $\s(J)=J$ and $b \in G(L)$ such that $(J, b)$ is admissible. If $\bar I_{\l, J,b} \neq \emptyset$, then $\nu_{b_x}=\nu_{[b]}^G$ for any $x \in \bar I_{\l, J,b}$.
\end{lem}
\begin{proof}
By definition, there exists $\t \in [b]_{J, \dom} \cap N_T(L)$ such that $\nu_{p(\t)}=\nu_{[b]}^G$ is central on $\Phi_J$. Let $x_0=\eta_J(\t) \in \pi_1(M_J)$. Applying \cite[Theorem 3.3]{HN1} to $\tW_J$, we have $\nu_{b_{x_0}}=\nu_{p(\t)}=\nu_{[b]}^G$.

Let $x \in \bar I_{\l, J, b}$. We show $\nu_{b_x}=\nu_{b_{x_0}}$. Note that $\k_J(x)=\k_{[b]_{J, \dom}}^{M_J}=\k_J(x_0)$. We have $x=x_0+(1-\s)(x') \in \pi_1(M_J)$ for some $x' \in \pi_1(M_J)$. Then $b_x=b_{x'} b_{x_0} \s(b_{x'}\i) \in \Omega_J$. Since $\nu_{b_{x_0}}=\nu_{[b]}^G$ is central on $\Phi_J$, we have $\nu_{b_x}=w_{x'}(\nu_{b_{x_0}})=\nu_{b_{x_0}}$ as desired.
\end{proof}

\subsection*{Terminology}
Let $g, h \in G(L)$ and $Z \subseteq G(L)$. We put ${}^g h= g h g\i$ (resp. ${}^g Z=g Z g\i$) and ${}^{g\s} h=g \s(h) g\i$ (resp. ${}^{g\s} Z=g \s(Z) g\i$).

For $\l \in Y^+$, $J \subseteq S_0$ and $b \in G(L)$, we write $Q \sim_{\l, J, b} P$ if $P$ and $Q$ are in the same connected component of $X_{\preceq_J \l}^{M_J}(b)$. If $J=S_0$, we abbreviate $\sim_{J, \l, b}$ to $\sim_{\l, b}$ for simplicity.

For $\tw \in \tW_J$ with $J \subseteq S_0$, we denote by $\dot \tw$ a lift of $\tw$ in $N_T(L) \cap M(L)$ under the natural projection $p: N_T(L) \to \tW$.

\section{The superbasic case}
In this section, we use the Deligne-Lusztig reduction method to study the connected components of unions of affine Deligne-Lusztig varieties in affine flag varieties. As an application, we obtain a new conceptual proof of Theorem \ref{main'} in the superbasic case.

\

Let $B^-$ be the opposite Borel subgroup of $B$ such that $B \cap B^-=T$. Let $I \subseteq G(L)$ be the pre-image of $B^-(\bold k)$ under the natural mod $t$ projection $K=G(\bold k[[t]]) \to G(\bold k)$. We call $I$ an Iwahori subgroup of $G(L)$. We have Bruhat decomposition $G(L)=\sqcup_{\tw \in \tW} I \dot \tw I$.

For $a=(\a, k) \in \tPhi$ with $\a \in \Phi$ and $k \in \ZZ$, let $U_a$ denote the corresponding affine root subgroup of the loop group $LG$. We have $U_a(\bold k)=U_\a(\bold k t^k)$. Here $U_\a \subseteq G$ is the root subgroup of $\a$. Then $I=T(\bold k[[t]]) \prod_{a \in \tPhi^+} U_a(\bold k)$.

For $\tw \in \tW$ and $b \in G(L)$, we define the corresponding affine Deligne-Lusztig variety (in affine flag variety) by $$X_{\tw}(b)=\{g I \in G(L)/I; g\i b \s(g) \in I \dot \tw I\}.$$ We say $\tw$ is reducible (with respect to $b$) if the following two statements holds:

(a): $X_{\tw}(b) \neq \emptyset$;

(b): for any $Q \in X_{\tw}(b)$, there exists a morphism $\phi: \PP^1 \to G(L)/I$ and an open subset $0 \in U \subseteq \AA^1=\PP^1-\{\infty\}$ such that $\phi(U) \subseteq X_{\tw}(b)$, $Q=\phi(0)$ and $P=\phi(\infty) \in X_{\tu}(b)$ for some $\tu \in \tW$ such that $\tu < \tw$ in the sense of Bruhat order.

\begin{lem}\label{sigma}
$\tw$ is reducible if so is $\s(\tw)$.
\end{lem}

\begin{lem}\label{red}
If $\tw \in \tW$ is non-reducible, then it is of minimal length in its $\s$-conjugacy classes in $\tW$.
\end{lem}

To prove the lemma, we need some preparation. For $w, w' \in \tW$ and $s \in S^a$, we write $w \xrightarrow{s}_\s w'$ if $w'=s w \s(s)$ and $\ell(w') \le \ell(w)$.  We write $w \to_\s w'$ if there is a sequence $w=w_0, w_1, \cdots, w_n=w'$ of elements in $\tW$ such that for each $k \in [1, n]$, $w_{k-1} \xrightarrow{s_k}_\s w_k$ for some $s_k \in S^a$. We write $w \approx_\s w'$ if $w \to_\s w'$ and $w' \to_\s w$.

\begin{thm} \cite{HN1} \label{min}
Let $\co$ be a $\s$-conjugacy class in $\tW$ and let $\co_{\min}$ be the set of minimal length elements in $\co$. Then for each $w \in \co$, there exists $w' \in \co_{\min}$ such that $w \rightarrow_\s w'$.
\end{thm}

Let $Q , P \in G(L)/I$. Then there exists a unique $\tw \in \tW$ such that $g Q=I$ and $g P=\dot\tw I$ for some $g \in G(L)$. In this case, we write $P \overset \tw \to Q$. For example, $Q \in X_{\tw}(b)$ if and only if $ Q \overset \tw \to b \s(Q)$. Define $Z_{\tw}=\{(Q, P) \in (G(L)/I)^2; Q \overset {\tw} \to P\}$

\begin{lem} \label{add}
Let $\tw, \tw_1, \tw_2 \in \tW$. We have the following properties:

(1) If $P \overset {\tw_1} \to H \overset {\tw_2} \to Q$ and $Q \overset {\tw'} \to P$, then $I \dot\tw I \subseteq I \dot\tw_1 I \dot\tw_2 I$;

(2) If $Q \overset \tw \to P$ and $\tw=\tw_1 \tw_2$ with $\ell(\tw)=\ell(\tw_1)+\ell(\tw_2)$, then there exists a unique $H \in G(L)/I$ such that $P \overset {\tw_1} \to H \overset {\tw_2} \to Q$;

(3) $\overline{Z_{\tw}}=\cup_{\tu \le \tw} Z_{\tu}$, where $\overline{Z_{\tw}}$ denotes the closure of $Z_{\tw}$ in $(G(L)/I)^2$.
\end{lem}

\begin{proof}[Proof of Lemma \ref{red}]
Assume $\tw$ is not of minimal length in its conjugacy classes. We show it is reducible. Let $b \in G(L)$ such that $X_{\tw}(b) \neq \emptyset$. By \cite[Theorem 1.1]{HN1}, there exist $\tw' \in \tW$ with $\tw' \approx_\s \tw$ and $s \in S^a$ such that $s \tw' \s(s) < \tw'$. First we show

(a) $\tw'$ is reducible.

Since $\tw' \approx_\s \tw$, $X_{\tw'}(b) \neq \emptyset$. We choose $Q' \in X_{\tw'}(b)$. Let $g' \in G(L)$ such that $Q'=g'I$ and ${g'}\i b \s(g') \in I_s U_{a_s}(x_0) \dot\tw$ for some $x_0 \in \bold{k}$. Here $a_s \in \tPhi^+$ is the positive affine root corresponding to $s$ and $I_s=T(\bold k[[t]])\prod_{a_s \neq a \in \tilde \Phi^+} U_a(\bold{k})$. Define $\phi: \AA^1 \to G(L)$ by $x \mapsto g' U_{-\a_s}(x) I$. Set $U=\AA^1$ if $x_0=0$ and $U=\AA^1 - \{-x_0\i\}$ otherwise. One checks that $\phi(U) \subseteq X_{\tw'}(b)$, $Q'=\phi(0)$ and $P':=\phi(\infty)=g' s I$. Then $P' \in X_{\tw' \s(s)}(b) \cup X_{s \tw' \s(s)}(b)$. Note that $\tw'\s(s), s \tw' \s(s) < \tw'$. So $\tw'$ is reducible and (a) is proved.

By Theorem \ref{min}, to finish the proof, it suffices to show the following statement:

\

{\it Let $s \in S^a$ such that $\tw$ and $\tw'=s \tw \s(s)$ are of the same length. Then $\tw$ is reducible if so is $\tw'$.}

\

Without loss of generality, we assume $s \tw < \tw < \tw \s(s)$. By Lemma \ref{sigma}, it suffices to show that $\s(\tw)$ is reducible. Let $Q_1 \in X_{\s(\tw)}(b)$. By Lemma \ref{add}(2), there exists unique $Q_1' \in G(L)/I$ such that $Q_1 \overset {\s(s)} \longrightarrow {\s(Q_1')} \overset {\s(s\tw)} \longrightarrow b\s(Q_1)$. Note that $b\s(Q_1) \overset {\s^2(s)} \to b\s^2(Q_1')$ and $\ell(s \tw \s(s))=\ell(s \tw)+\ell(\s(s))$, we have $\s(Q_1') \overset {\s(s \tw \s(s))} \to b\s^2(Q_1')$ by Lemma \ref{add}(1), that is, $Q_1' \in X_{\tw'}(\s\i(b))$.

By assumption, there exists $\phi': \PP^1 \to G(L)/I$ and an open subset $0 \in U \subseteq \AA^1$ such that $\phi'(U) \subseteq X_{\tw'}(\s\i(b))$, $Q_1'=\phi'(0)$ and $P_1':=\phi'(\infty) \in X_{\tu'}(\s\i(b))$ for some $\tu' < \tw'$. Using Lemma \ref{add}(2), we can define a morphism $\phi: U \to G(L)/I$ such that $\phi' \overset {\tw' \s(s)} \longrightarrow \s\i(b) \phi \overset {\s(s)} \longrightarrow \s\i(b)\s(\phi')$. One checks that $\phi \overset {\s(s)} \longrightarrow \s(\phi') \overset {\s(\tw' \s(s))} \longrightarrow b\s(\phi)$. So $\phi(U) \subseteq X_{\s(\tw)}(b)$ and $\phi(0)=Q_1$. Now we extend $\phi$ to $\PP^1 \supseteq U$, which is still denoted by $\phi$. Let $\tu \in \tW$ such that $P_1:=\phi(\infty) \in X_{\s(\tu)}(b)$. By Lemma \ref{add}(3), $\tu \le \tw$. We have to show that $\tu < \tw$. Assume otherwise, that is, $\tu=\tw$ and $P_1 \overset {\s(\tw)} \longrightarrow b\s(P_1)$. We claim

(b) $P_1 \overset {\s(s)} \longrightarrow \s(P_1')$.

Note that $\phi(x) \overset {\s(s)} \longrightarrow \s(\phi'(x))$ for $x \in U$. So either $P_1 \overset {\s(s)} \longrightarrow \s(P_1')$ or $P_1=\s(P_1')$ by Lemma \ref{add}(3). If the latter case occurs, then $P_1=\s(P_1') \in X_{\s(\tu')}(b)$ and hence $\tu'=\tw$, which is impossible because $\ell(\s(\tu')) < \ell(\tw')=\ell(\tw)$. Therefore, (b) is proved.

Now we have \begin{align*}\tag{$\ast$} P_1 \overset {\s(s)} \longrightarrow \s(P_1') \overset {\s(\tu')} \longrightarrow b\s^2(P_1') \overset {\s^2(s)} \longrightarrow b\s(P_1).\end{align*} Therefore, $\ell(\tu') \ge \ell(\tw')-2$ by Lemma \ref{add}(1). If $\ell(\tu')=\ell(\tw')-2$, then $\tw=s \tu' \s(s)$ and hence $\tw \s(s) < \tw$, which contradicts our assumption on $\tw$. So $\ell(\tu') =\ell(\tw')-1$ and $\tw=s \tu'$ or $\tw=\tu' \s(s)$. In the former case, we have $\ell(s \tu' \s(s))=\ell(\tu')+2$, which means, by ($\ast$) and Lemma \ref{add}(1), $\tw=s \tu' \s(s)$, a contradiction because $\ell(s \tu' \s(s)) \neq \ell(\tw)$. In the latter case, we have $\tu'=\tw \s(s)$ and $\ell(\tu') > \ell(\tw')$, which is also impossible. Therefore $\s(\tw)$ is reducible and the proof is finished.
\end{proof}

\

Let $v \in Y_\RR$. Denote by $M_v \subseteq G$ the Levi subgroup (over $\bold k[[t]]$) generated by $T$ and the root subgroups $U_\a$ with $\<\a, v\>=0$. Denote by $N_v \subseteq G$ the unipotent subgroup (over $\bold k[[t]]$) generated by the root subgroups $U_\b$ with $\<\b, v\> > 0$. We put $I_{M_v}=M_v(L) \cap I$ and $I_{N_v}=N_v(L) \cap I$.

Let $\tw=t^\mu w \in \tW$ with $\mu \in Y$ and $w \in W_0$. We say $\tw$ is a $(v, \s)$-alcove element if $w \s(v)=v$ and ${}^{\dot\tw\s} I_{N_v} \subseteq I_{N_v}$.  We say $\tw$ is $(v, \s)$-fundamental if $\tw$ is a $(v, \s)$-alcove element and ${}^{\dot\tw\s} I_{M_v}=I_{M_v}$.

\begin{thm} [\cite{GHKR}, \cite{GHN}] \label{HN}
Let $v \in Y_\RR$ and $\tw \in \tW$ such that $\tw$ is a $(v, \s)$-alcove elements. Then the inclusion $I_{M_v} \tw \s(I_{M_v}) \hookrightarrow I \tw I$ induces a bijection between $I_{M_v}$-$\s$-conjugacy classes in $I_{M_v} \tw \s(I_{M_v})$ and $I$-$\s$-conjugacy classes in $I \tw I$.
\end{thm}

\begin{thm} [\cite{GHKR}, \cite{He}, \cite{N}] \label{fundamental}
Let $\tw \in \tW$. The following are equivalent:

(a) $\tw$ is $(\nu_{\tw}, \s)$-fundamental;

(b) $\tw$ is $\s$-fundamental, that is, $I \dot \tw I$ lies in a single $I$-$\s$-conjugacy class;

(c) $\tw$ is $\s$-straight, that is, $(I \dot \tw \s I)^n=I (\dot \tw \s)^n I$ for $n \in \NN$, or equivalently, $\ell(\tw)=2\<\rho, \bar \nu_{\tw}\>$. Here $\rho$ is the half sum of positive roots.
\end{thm}

\begin{prop} \label{flag}
Let $b \in G(L)$ and let $C \subseteq \tW$ be closed under the Bruhat order. Then any connected component of $\cup_{\tu \in C} X_{\tu}(b)$ intersects with $X_x(b)$ for some $\s$-straight element $x \in C$.
\end{prop}
\begin{proof}
Let $\tw \in C$ and $Q \in X_{\tw}(b)$. We show the connected component of $\cup_{\tu \in C} X_{\tu}(b)$ containing $Q$ intersects with $X_x(b)$ for some $\s$-straight element $x \in C$. Since $C$ is closed under the Bruhat order, we may assume, by Lemma \ref{red}, that $\tw$ is of minimal length in its $\s$-conjugacy class. Thanks to \cite[Theorem 3.5]{He}, $I \dot\tw I$ lies in a single $\s$-conjugacy class of $G(L)$. Therefore, $\tw$ is a $(\nu_{\tw}, \s)$-alcove element by \cite[Theorem 5.2]{N}. Moreover, there exist $J \subseteq S_M^a$ with $W_J$ finite, $x \in \tW$ which is $(\nu_{\tw}, \s)$-fundamental such that $x\s(J)x\i=J$ and $\tw =u x$ for some $u \in W_J$. Here $M=M_{\nu_{\tw}}$ and $S_M^a$ is the set of simple affine refections in $\tW_M$ defined in \S \ref{setup3}.

Write $Q=g I$ for some $g \in G(L)$. Then $g\i b \s(g) \in I \dot\tw I$. By Theorem \ref{HN}, we can assume $g\i b \s(g) \in I_M \dot\tw \s(I_M)$ since $\tw$ is a $(\nu_{\tw}, \s)$-alcove element. Let $\cp_J\subseteq M(L)$ be the standard parahoric subgroup of type $J$ and let $\cu_J \subseteq \cp_J$ be its pro-unipotent radical.
Since ${}^{\dot x\s} I_M=I_M$ and $x\s(J)x\i=J$, we see that Lang's map $\psi: \cp_J \to \cp_J$ define by $f \mapsto f\i \dot x\s(f) \dot x\i$ is surjective. So there exists $h_0 \in \cp_J$ such that $h_0\i \dot x \s(h_0)=g\i b \s(g)$, which means $g=j h_0$ for some $j \in G(L)$ satisfying $j\i b \s(j)=\dot x$. For $w \in W_J$, we define $$X_w^{\cp_J}(\dot x)=\{h I_M \in \cp_J / I_M; h\i \dot x \s(h) \dot x\i \in I_M \dot w I_M\}$$ and $X_{\leq_M u}^{\cp_J}(\dot x)=\cup_{w \in W_J, w \leq_M u} X_w^{\cp_J}(\dot x)$. Here $\leq_M$ is the Bruhat order on $\tW_M$ defined in \S \ref{setup3}. Then the map $h I_M \mapsto j h I$ induces a natural morphism $\varpi_{\tw, b}: X_{\leq_M u}^{\cp_J}(\dot x) \to \cup_{\tu \in C} X_{\tu}(b)$.

Since the image of $\cup_{w \in W_J, w \leq_M u} I_M \dot w I_M$ under the natural projection $\cp_J \to \cp_J / \cu_J$ is connected and $\psi$ induces a flat and finite endomorphism on $\cp_J / \cu_J$, there exists $h_1 \in \cp_J$ such that $h_1 I_M \in X_{\Id}^{\cp_J}(\dot x)$ and $h_0 I_M \in X_{u}^{\cp_J}(\dot x)$ are in the same connected component of $X_{\leq_M u}^{\cp_J}(\dot x)$. Therefore, via the morphism $\varpi_{\tw, b}$, we deduce that $Q=g I=j h_0 I$ and $P=j h_1 I \in X_x(b)$ are connected in $\cup_{\tu \in C} X_{\tu}(b)$.
\end{proof}

\

\begin{rmk}
There is an alternative proof of Proposition \ref{flag} due to Xuhua He \cite{He1}. The main idea is to show the following two facts:

(1) $X_{\tw}(b)$ is locally quasi-affine;

(2) each irreducible component of $X_{\tw}(b)$ is of dimension $\ge 1$ if $\tw$ is not $\s$-straight.
\end{rmk}

\

We say $b \in G(L)$ is basic if $\k_{[b]}^G$ is central on $\Phi$.
\begin{cor} \label{basic}
Let $\l \in Y^+$ and $b \in G(L)$ which is basic. Then

(i) each connected component of $X_{\preceq \l}(b)$ intersects with $X_{\mu}(b)$, where $\mu \in Y^+$ is the unique dominant minuscule cocharacter such that $\mu \preceq \l$;

(ii) $J_b(F)=\{g \in G(L); g\i b \s(g)=b\}$ acts transitively on $\pi_0(X_{\preceq \l}(b))$;

(iii) if $b$ is superbasic, then the natural projection $X_{\preceq \l}(b) \to \pi_1(G)$ induces an isomorphism $\pi_0(X_{\preceq \l}(b)) \cong (\s-1)\i(\eta(t^\l)-\eta(b))$.
\end{cor}
\begin{proof}
Let $\pi: G(L)/I \to G(L)/K$ be the natural projection. Note that $\pi\i(X_{\preceq \l}(b))=\cup_{\tu \in C} X_{\tu}(b)$, where $C=\cup_{\l' \preceq \l} W_0 t^{\l'} W_0$ is closed under the Bruhat order. By Lemma \ref{flag}, each connected component of $\pi\i(X_{\preceq \l}(b))$ intersects with $X_x(b)$ for some $\s$-straight element $x \in C$. So, we have $[\dot x]=[b]$ by \cite[Theorem 3.7]{He}. In particular, $\bar \nu_x=\nu_{[b]}^G$ and hence $\ell(x)=2\<\rho, \nu_{[b]}^G\>=0$. Therefore, $x$ must be the unique element of $\Omega \cap C$.

Now (i) follows by noticing that $x \in t^\mu W_0$ and $\pi(X_x(b)) \subseteq X_\mu(b)$.

Since $x \in \Omega$, each element of $I \dot x I$ is $\s$-conjugate to $\dot x$ by $I$. In particular, the map $g \mapsto g I$ induces a surjective map: $$J_{b ,\dot x}(F)=\{g \in G(L); g\i b \s(g)=\dot x\} \twoheadrightarrow X_x(b).$$ Then (ii) follows by noticing that $J_b(F)$ acts (by left multiplication) on $J_{b ,\dot x}(F)$ transitively.

Since $[\dot x]=[b]$, we can assume $b=\dot x \in N_T(L)$. Since $x \in \Omega$ is superbasic, $J_{b}(F)$ is generated by $N_T(L)^{b\s}=J_b(F) \cap N_T(L)$ and $I^{b\s}=J_b(F) \cap I$. Moreover, by \cite[Proposition 5.3]{He}, the image of $N_T(L)^{b\s}$ under the natural projection $p: N_T(L) \to \tW$ is $$\tW^{x \s}=\{\tw \in \tW; x \s(\tw) x\i=\tw\}=\{\tw \in \Omega; \s(\tw)=\tw\} \cong \pi_1(G)^\s.$$ So $J_b(F)=N_T(L)^{b\s} I^{b\s}$ and $J_b(F) / I^{b\s} = \tW^{x \s}$. By (ii), the map $g \mapsto g K$ induces a surjection $$\pi_1(G)^\s =\tW^{x \s}=J_{b}(F) / I^{b\s} \to \pi_0(X_{\preceq \l}(b)).$$ On the other hand, the map is obviously an injection and hence a bijection as desired.
\end{proof}

\begin{cor} \label{red'}
Let $\l \in Y^+$, $b \in G(L)$ and $J \subseteq S_0$ with $\s(J)=J$ such that $(J, b)$ is admissible. Assume $b \in N_T(L) \cap [b]_{J, \dom}$ such that $p(b) \in \Omega_J$. Then each connected component of $X_{\preceq \l}(b)$ contains a point $j Q$, where $j \in J_b(F) \cap N(L)$ and $Q \in X_{\mu}^{M_J}(b)$ for some $\mu \in \bar I_{\l, J, b}$. Here $N=\prod_{\a \in \Phi^+-\Phi_J} U_\a$.
\end{cor}
\begin{proof}
By \cite[Proposition 3.4.1]{CKV}, each connected component of $X_{\preceq \l}(b)$ contains a point $j Q'$, where $j \in J_b(F) \cap N(L)$ and $Q' \in X_{\mu'}^{M_J}(b)$ for some $\mu' \preceq \l$. Applying Corollary \ref{basic} (i), there exist a $J$-minuscule and $J$-dominant cocharacter $\mu \preceq_J \mu'$ and $Q \in X_\mu^{M_J}(b)$ such that $Q' \sim_{J, \mu', b} Q$. Hence we have $j Q' \sim_{\l, b} j Q$. Since $X_\mu^{M_J}(b) \neq \emptyset$, we have $\k_J(t^\mu)=\k_J(b)$. So $\mu \in \bar I_{\l, J, b}$ and the proof is finished.
\end{proof}

\section{The choice of $J$}
Let $\l \in Y^+$ and $b \in G(L)$. In view of Corollary \ref{red'}, we are faced with the question of the existence of $J \subseteq S_0$ such that

\

(a) $\s(J)=J$, $(J, b)$ is admissible and $b$ is superbasic in $M_J$.

\

It is known such $J$ always exists. However, due to technical reasons (see Lemma \ref{lattice}), we need to add the following additional requirement.

\

(b) $\bar I_{\l, J, b}$ contains a weakly dominant cocharacter. Here we say a cocharacter $\mu \in Y$ is weakly dominant if $\<\a, \mu\> \ge -1$ for any $\a \in \Phi^+$.

\

The main purpose of this section is to show the existence of $J$ that satisfies (a) and (b), using the different characterizations of $\s$-straight elements in Theorem \ref{fundamental}. Notice that if $\l$ is minuscule, then each element $\bar I_{\l, J, b}$ is minuscule and (b) is a priori satisfied.

\

For $v \in Y_\QQ^+$, we define $J_v=\{s \in S_0; s(v)=v\}$. Let $\tw \in \tW$ be a $\s$-straight element. We set $\tw^\diamond=z \tw \s(z)$, where $z \in {}^{J_{\bar \nu_{\tw}}} W_0$ is the unique element such that $z(\nu_{\tw})=\bar \nu_{\tw}$.

\begin{lem} \label{diamond}
Let $\tw \in \tW$ be $\s$-straight. Then

(a) $\tw^\diamond \in \Omega_{J_{\bar \nu_{\tw}}}$ with $\s(J_{\bar \nu_{\tw}})=J_{\bar \nu_{\tw}}$;

(b) $\tw^\diamond \in t^\chi  W_{J_{\bar \nu_{\tw}}}$ with $\chi \in Y$ weakly dominant.
\end{lem}
\begin{proof}
(a) The equality $\s(J_{\bar \nu_{\tw}})=J_{\bar \nu_{\tw}}$ follows immediately from the fact that $\bar \nu_{\tw}$ is dominant. Note that ${}^{\dot z} I_{M_{\nu_{\tw}}}=I_{M_{\bar \nu_{\tw}}}$. By Theorem \ref{fundamental}(a), we have ${}^{\dot\tw\s} I_{M_{\nu_{\tw}}}=I_{M_{\nu_{\tw}}}$, that is, ${}^{\dot\tw^\diamond\s} I_{M_{\bar \nu_{\tw}}}=I_{M_{\bar \nu_{\tw}}}$. So $\tw^\diamond \in \Omega_{J_{\bar \nu_{\tw}}}$.

(b) We show that $\<\a, \chi\> \ge -1$ for any $\a \in \Phi^+$. If $\a \in \Phi_{J_{\bar \nu_{\tw}}}$, then $\<\a, \chi\> \in \{0, 1\}$ since $\tw^\diamond \in \Omega_{J_{\bar \nu_{\tw}}}$. Assume $\a \in \Phi^+ - \Phi_{J_{\bar \nu_{\tw}}}$. Then $\<z\i \s\i(u)\i(\a), \nu_{\tw}\>=\<\s\i(u)\i(\a), \bar \nu_{\tw}\>=\<\a, \bar \nu_{\tw}\> > 0$, where $u=t^{-\chi} \tw^\diamond \in W_{J_{\bar \nu_{\tw}}}$ and $z \in {}^{J_{\bar \nu_{\tw}}} W_0$ is the unique element such that $z(\nu_{\tw})=\bar \nu_{\tw}$. Since $\tw$ is a $(\nu_{\tw}, \s)$-alcove element, ${}^{\dot\tw\s} (U_{z\i \s\i(u\i)(\a)}(L) \cap I) \subseteq I$, which implies $\<z\i(\a), z\i(\chi)\>=\<\a, \chi\> \ge -1$ and the proof is finished.
\end{proof}

\begin{lem} \label{exist}
Let $\tw \in \tW$ be $\s$-straight. Then there exist $J \subseteq J_{\bar \nu_{\tw}}$ and $z_0 \in W_{J_{\bar \nu_{\tw}}}$ such that

(a) $\s(J)=J$ and $\tw_0= z_0 \tw^\diamond \s(z_0)\i$ is $\s$-superbasic in $\tW_J$;

(b) $\tw_0 \in t^{\chi_0} W_J$ with $\chi_0 \in Y$ weakly dominant;

(c) $\nu_{\tw_0}=\nu_{\tw^\diamond}=\bar \nu_{\tw}$.
\end{lem}
\begin{proof}
Let $V_{\tw^\diamond}=\{q \in Y_\QQ; \tw^\diamond\s(q)=q+\bar \nu_{\tw}\}$, which is a nonempty affine subspace. Following \cite[Section 3]{HN2}, we set $V_{\tw^\diamond}'=\{q-e; q \in V_{\tw^\diamond}\}$, where $e$ is an arbitrary point of $V_{\tw^\diamond}$. Then $V_{\tw^\diamond}'$ is the (linear) subspace parallel to $V_{\tw^\diamond}$. We choose a point $v_0 \in V_{\tw^\diamond}'$ lying in a sufficiently small neighborhood of $\nu_{\tw^\diamond} =\bar \nu_{\tw}$ which is also generic in $V_{\tw^\diamond}'$, that is, for any root $\a \in \Phi$, $\<\a, v_0\>=0$ implies that $\<\a, v'\>=0$ for each $v' \in V_{\tw^\diamond}'$. Let $J=J_{\bar v_0}$ and let $z_0$ be the unique element in ${}^{J_{\bar v_0}} W_0$ such that $z_0(v_0)=\bar v_0$. Since $v_0$ lies in a sufficient small neighborhood of the dominant vector $\bar \nu_{\tw}$, $z_0(\bar \nu_{\tw})$ is also dominant. So $z_0(\bar \nu_{\tw})=\bar \nu_{\tw}$ and $z_0 \in W_{J_{\bar \nu_{\tw}}}$.

(a) Note that $\tw^\diamond \in \Omega_{J_{\bar \nu_{\tw}}}$. Then (a) follows from \cite[Lemma 3.1]{HN2}.

(b) Assume $\tw^\diamond \in t^\chi W_0$ for some $\chi \in Y$. Then $\tw_0 \in t^{z_0(\chi)} W_0$. By Lemma \ref{diamond}, $\chi$ is weakly dominant and $J_{\bar \nu_{\tw}}$-minuscule. So $\chi_0=z_0(\chi)$ is weakly dominant as $z_0 \in W_{J_{\bar \nu_{\tw}}}$.

(c) It follows directly from the inclusion $z_0 \in W_{J_{\bar \nu_{\tw}}}$.
\end{proof}

\begin{cor}\label{ideal}
Let $\l \in Y^+$ and $b \in G(L)$. If $X_{\preceq \l}(b) \neq \emptyset$. Then there exists $J \subseteq S_0$ with $\s(J)=J$ such that

(i) $(J, b)$ is admissible and $b$ is superbasic in $M_J$;

(ii) $\bar I_{\l, J, b}$ contains a weakly dominant coweight.
\end{cor}
\begin{proof}
Since $X_{\preceq \l}(b) \neq \emptyset$, there exists $\tu \in W_0 t^{\l'} W_0$ with $\l' \preceq \l$ such that $X_{\tu}(b) \neq \emptyset$. By the proof of \cite[Proposition 1.5]{N}, there exists a $\s$-straight element $\tw \in \tW$ such that $\tw \le \tu$ and $[\dot\tw]=[b]$. Let $J \subseteq S_0$ and $\tw_0 \in \Omega_J$ be defined as in Lemma \ref{exist} (with respect to $\tw$). Then $\tw_0$ is $\s$-superbasic in $W_J$ with $\nu_{\tw_0}=\bar \nu_{\tw}=\nu_{[b]}^G$. So $[\dot \tw_0]_J=[b]_{\dom, J}$ and (i) follows. Let $\chi_0 \in Y$ be such that $\tw_0 \in t^\mu W_J$, which is $J$-dominant, $J$-minuscule and weakly dominant by Lemma \ref{exist}. Moreover, $\chi_0 \preceq \l$ since $W_0 \tw_0 W_0=W_0 \tw W_0$ and $\tw \le \tu$. So $\chi_0 \in \bar I_{\l, J, b}$ and (ii) follows.
\end{proof}

\section{Proof of Theorem \ref{main'}}
The aim of this section is to prove Theorem \ref{main'}. The strategy is the same as in \cite[\S 4.1]{CKV}. We rephrase it here for completeness. Throughout the section, we assume $G$ is adjoint and simple. We fix $\l \in Y^+$, $b \in G(L)$ such that $X_{\preceq \l}(b) \neq \emptyset$.

\subsection{} \label{5.1} In this subsection, we fix $J \subseteq S_0$ such that $\s(J)=J$, $(J, b)$ is admissible and $b$ is superbasic in $M_J$. By Lemma \ref{ss}, we can assume:

(A) $b \in N_T(L) \cap [b]_{J, \dom}$ such that $\nu_{p(b)}=\nu_{[b]}^G$ and $p(b) \in \Omega_J$ is $\s$-superbasic in $\tW_J$.

By lifting the natural projection $N_T^{M_J}(L) \to \tW_J$, we also view, by (A), $\Omega_J^\s=\{x \in \Omega_J; \s(x)=x\}=\pi_1(M_J)^\s$ as a subset of $N_T^{M_J}(L)^{b\s}=J_b^{M_J}(F) \cap N_T^{M_J}(L)$, where $N_T^{M_J}$ is the normalizer of $T$ in $M_J$.

\begin{prop} \rm{(cf. \cite[Theorem 4.1.12]{CKV})} \label{k1}
The image of the natural map $$\pi_0(X_{\mu_x}^{M_J}(b)) \to \pi_0(X_{\preceq \l}(b))$$ induced by the embedding $X_{\mu_x}^{M_J}(b) \hookrightarrow X_{\preceq \l}(b)$  dose not depend on the choice of $x \in \bar I_{\l, J. b}$.
\end{prop}

The proof will be given in Section \ref{sec k1}.

\

For $\a \in \Phi$, we set $y_\a=\sum_{\b \in \co_\a} \b^\vee \in Y/ \ZZ \Phi_J^\vee=\Omega_J^\s = \pi_1(M_J)^\s$. Here $\co_\a=\{\s^k(\a); k \in \ZZ\}$. Let $\ca_{\l, J, b, x}$ be the subset of $\Phi$ such that $\a \in \ca_{\l, J, b, x}$ if and only if there exist $Q, P \in X_{\mu_x}^{M_J}(b)$ satisfying $Q \sim_{\l, b} P$ and $\eta_J(Q)-\eta_J(P)=y_\a$.
\begin{lem}\label{k21}
Let $x \in \bar I_{\l, J, b}$. Then

(i) $\ca_{\l, J, b, x} = - \ca_{\l, J, b, x}$;

(ii) $\ca_{\l, J, b, x'}=\ca_{\l, J, b, x}$ for any $x' \in \bar I_{\l, J, b}$;

(iii) If $\a \in \ca_{\l, J, b, x}$, then for any $Q \in  X_{\mu_x}^{M_J}(b)$, there exists $P \in  X_{\mu_x}^{M_J}(b)$ such that $Q \sim_{\l, b} P$ and $\eta_J(Q)-\eta_J(P)=y_\a \in \pi_1(M_J)^\s$.
\end{lem}
\begin{proof}
(i) it follows directly by definition.

(ii) By symmetry, it suffices to show $\ca_{\l, J, b, x} \subseteq \ca_{\l, J, b, x'}$. Let $\a \in \ca_{\l, J, b, x}$. Then there exist $Q, P \in X_{\mu_x}^{M_J}(b)$ such that $Q \sim_{\l, b} P$ and $\eta_J(Q)-\eta_J(P)=y_\a \in \pi_1(M_J)^\s$. By Proposition \ref{k1}, there exists $Q' \in X_{\mu_{x'}}^{M_J}(b)$ such that $Q' \sim_{\l, b} Q$. Note that $y_\a\i Q \sim_{J, \mu_x, b} P$ by Corollary \ref{basic} (iii).
Thus $Q' \sim_{\l, b} Q \sim_{\l, b} P \sim_{J, \l, b} y_\a\i Q \sim_{\l, b} y_\a\i Q'$ and the proof is finished.

(iii) It follows from the fact that the action of $J_b^{M_J}(F)$ on $\pi_0(X_{\mu_x}^{M_J}(b))$ is transitive, see Corollary \ref{basic} (ii).
\end{proof}

Thanks to Lemma \ref{k21}, we can define $\ca_{\l,J,b}=\ca_{\l,J,b,x}$, which is independent of the choice $x \in \bar I_{\l,J,b}$. Let $\cl_{\l, b, J} \subseteq \pi_1(M_J)^\s$ be the sublattice spanned by $y_\a$ for $\a \in \ca_{\l,J,b}\}$.

\begin{prop} \label{k2}
Assume furthermore that $(\l, b)$ is HN-irreducible (see Introduction) and $\bar I_{\l, J, b}$ contains a weakly dominant cocharacter. Then $\cl_{\l, b, J}$ is the kernel of the natural surjective projection $\pi_1(M_J)^\s \twoheadrightarrow \pi_1(G)^\s$.
\end{prop}

The proof will be given in Section \ref{sec k2}.

\

\subsection{} Now we are ready to prove our main result.
\begin{proof} [Proof of Theorem \ref{main'}]
By Corollary \ref{ideal}, there exists $J$ with $\s(J)=J$ such that $(J, b)$ is admissible, $b$ is superbasic in $M_J(L)$ and $\bar I_{\l, J, b}$ contains a weakly dominant cocharacter. Moreover, we can assume $b=\dot b_x$ for some $x \in \bar I_{\l, J, b}$. So \S \ref{5.1} (A) holds. Since $\eta_G(b)=\eta_G(t^\l)$, it suffices to show the map $g \mapsto \eta_G(g)$ induces a bijection $\pi_0(X_{\preceq \l}(b)) \cong \pi_1(G)^\s$.

By Corollary \ref{basic} (iii), the map $g \mapsto g K_J$ induces a $\pi_1(M_J)^\s$-equivariant  surjection $$\pi_1(M_J)^\s \cong J_b^{M_J}(F) / (J_b^{M_J}(F) \cap I_{M_J}) \twoheadrightarrow \pi_0(X_{\mu_x}^{M_J}(b)).$$ Let $C$ be the image of the natural map $\pi_0(X_{\mu_x}^{M_J}(b)) \to \pi_0(X_{\preceq \l}(b))$. Then we obtain a natural $\pi_1(M_J)^\s$-equivariant map $$\varphi_x: \pi_1(M_J)^\s \twoheadrightarrow \pi_0(X_{\mu_x}^{M_J}(b))  \twoheadrightarrow C.$$ Since $\pi_1(M_J)^\s$ is an abelian group, the stabilizer $\stab(\varphi_x):=\varphi_x\i(c)$ is independent of the choice of $c \in C$. We claim that

(i) $\stab(\varphi_x)=\ker(\pi_1(M_J)^\s \twoheadrightarrow \pi_1(G)^\s)$.

(ii) $C=\pi_0(X_{\preceq \l}(b))$.

Assuming (i) and (ii), we see that $\varphi_x$ induces a bijection from $\pi_1(G)^\s =\pi_1(M_J)^\s / \stab(\varphi_x)$ to $\pi_0(X_{\preceq \l}(b))$, as desired.

We prove (i) first. Let $Q \in X_{\mu_x}^{M_J}(b)$. By Proposition \ref{k2}, it suffices to show $Q \sim_{\l, b} y_\a Q$ for each $\a \in \ca_{\l, J, b}=\ca_{\l, J, b, x}$. Applying Lemma \ref{k21} (iii), there exist $P \in X_{\mu_x}^{M_J}(b)$ such that $Q \sim_{\l, b} P$ and $\eta_J(Q)-\eta_J(P)=-y_\a \in \pi_1(M_J)^\s$. On the other hand, since $\eta_J(P)=\eta_J(y_\a Q)$, we have $y_\a Q \sim_{J, \mu_x, b} P$ by Corollary \ref{basic} (iii). So $Q \sim_{\l, b} y_\a Q$ and (i) is proved.

Now we show (ii). By Corollary \ref{red'} and Proposition \ref{k1}, $\varphi_x$ is surjective modulo the action of $J_b(F) \cap N(L)$ on $\pi_0(X_{\preceq \l}(b))$. So it suffices to show $Q \sim_{\l, b} j Q$ for any $j \in J_b(F) \cap N(L)$ and any $Q \in X_{\mu_x}^{M_J}(b)$. Choose $\d \in \ZZ \Phi^\vee$ (depending on $j$ and $Q$) such that $\d$ is central on $\Phi_J$, $\s(\d)=\d$ and $t^\d j t^{-\d} Q=Q$. Note that $\d \in \ker$ and $t^\d \in J_b^{M_J}(F)$. Then we have $t^\d Q \sim_{\l, b} Q$ by (i). Therefore, $$t^\d j Q=t^\d j t^{-\d} t^\d Q \sim_{\l, b} t^\d j t^{-\d} Q=Q \sim_{\l, J} t^\d Q,$$ which means $j Q \sim_{\l, b} Q$ as desired.
\end{proof}

\section{The set $\Theta (\mu, \mu', \l)$}
This section is devoted to the study of the set $\Theta (\mu, \mu', \l)$ define below, which plays a crucial role in the proof of Proposition \ref{k1}.

\

Let $\mu, \mu' \in Y$ and $\l \in Y^+$ such that $\mu, \mu' \preceq \l$. We Define \begin{align*} \Theta(\mu, \mu', \l)&=\{\a \in \Phi; \<\a, \mu-\mu'\> \ge 2, \mu-\a^\vee, \mu'+\a^\vee \preceq \l\}; \\ \Xi(\mu, \mu') &= \{\a \in \Phi; \<\a, \mu\>, -\<\a, \mu'\> \ge 1\}; \\ \Xi_1(\mu, \mu') &= \{\a \in \Phi; \<\a, \mu\>=-\<\a, \mu'\>=1\}. \end{align*}

\begin{lem}\label{diff} \label{act}
We have the following properties:

(a) $\Xi(\mu, \mu') \subseteq \Theta(\mu, \mu', \l)$;

(b) $\Theta(\mu, \mu', \l)=-\Theta(\mu', \mu, \l)$;

(c) $w(\Theta(\mu, \mu', \l))=\Theta(w(\mu), w(\mu'), \l)$, $w(\Xi(\mu, \mu'))=\Xi(w(\mu), w(\mu'))$ and $w(\Xi_1(\mu, \mu'))=\Xi_1(w(\mu), w(\mu'))$, where $w \in W_0$.
\end{lem}

\begin{lem} \label{pm}
If $\Xi(\mu, \mu') = \emptyset$, then there exists $w \in W_0$ such that $w(\mu), w(\mu')$ are dominant.
\end{lem}



\begin{lem} \label{contraction}
Let $\d, \d' \in \NN \Pi_0^\vee$. Then there exists an index set $\L$ of roots such that

(1) $\d'-\d=\sum_{k \in \L} \g_k^\vee$;

(2) $\sum_{k \in \Sigma} \g_k^\vee \le \d'$ and $-\sum_{k \in \Sigma} \g_k^\vee \le \d$ for any subset $\Sigma \subseteq \L$;

(3) $\<\g_i, \g_j^\vee\> \ge 0$ for any $i, j \in \L$.
\end{lem}
\begin{proof}
We argue by induction on the partial order $\le$ on $\NN \Pi_0^\vee$. By \cite[Lemma 4.2.5]{CKV}, we can write $\d=-\sum_{i \in \G} \a_i^\vee$ (resp. $\d'=\sum_{j \in \G'} \b_j^\vee$) such that $\a_i \in \Phi^-$ (resp. $\b_{i'} \in \Phi^+$) and $\<\a_i, \a_j^\vee\> \ge 0$ (resp. $\<\b_{i'}, \b_{j'}^\vee\> \ge 0$) for any $i, j \in \G$ (resp. $i', j' \in \G'$). If $\<\a_i, \b_{i'}^\vee\> \ge 0$ for any $i \in \G$ and $i' \in \G'$, we may take $\L=\G \sqcup \G'$. Otherwise, there exists $i_0 \in \G$ and $i_0' \in \G'$ such that $\<\a_{i_0}, \b_{i_0'}^\vee\> \le -1$. Then $\th^\vee:=\a_{i_0}^\vee+\b_{i_0'}^\vee \in \Phi^\vee \cup \{0\}$. Without loss of generality, we may assume $\th^\vee > 0$. Set $\d_1= -\sum_{i \in \G \setminus \{i_0\}} \a_j^\vee < \d$ and $\d_1'=\th^\vee + \sum_{j' \in \G' \setminus \{i_0'\}} \b_{j'}^\vee < \d'$. Note that $\d_1'-\d_1=\d'-\d$. The statement now follows from the induction hypothesis on the pair $(\d_1, \d_1')$.
\end{proof}

\begin{lem} \label{key}
Let $\mu \in Y$, $\l \in Y^+$ and $\a \in \Phi$. If $\mu+\a^\vee \le \l$ and $\mu+\a^\vee \npreceq \l$, then there exists $\g \in \Phi^+$ such that $\<\g, \mu+\a^\vee\> \le -2$. If, moreover, $\mu \in Y^+$, then $\<\g, \mu_x\> \le 1$, $\<\g, \a^\vee\> \le -2$ and either

(a) $\g=-\a$; or

(b) $\g$ is a long root of $\Phi$ and $\mu+\a^\vee+\g^\vee \le \l$.
\end{lem}
\begin{proof}
Set $\mu_0=\mu+\a^\vee$. We define $$\Phi^+(\mu_0)=\{\th \in \Phi^+; \<\th, \mu_0\> \le -2\}.$$ Let $w \in W_0$ be the unique minimal element satisfying $w(\mu_0)=\bar \mu_0$. Let $w=s_r \cdots s_1$ be a reduced expression of $w$ with each $s_i \in S_0$. Let $\mu_k=s_k \cdots s_1(\mu_0)$. Note that $\<\a_k, \mu_{k-1}\> \le -1$ and hence $\mu_{k-1}+\a_k^\vee \le \mu_k$ for $k \in [1, r]$, where $\a_k$ denotes the simple root corresponding to $s_k$. Since $\mu_0 \le \l$ and $\mu_r=\bar \mu_0 \nleqslant \l$, there exists $i \in [1, r]$ such that $\mu_{i-1} \le \l$ and $\mu_i=s_i(\mu_{i-1}) \nleqslant \l$. Write $\l = \mu_{i-1}+ \sum_{\a \in \Pi} n_{\a} \a^\vee$, where each $n_\b \in \NN$. We have $n_{\a_{i}} \ge 1$ since $\<\a_i, \l\> \ge 0$. So $\mu_{i-1} + \a_i^\vee \le \l$. If $\<\a_i, \mu_{i-1}\> =-1$, then $\mu_i=\mu_{i-1}+\a_i^\vee \le \l$, which contradicts the choice of $i$. Therefore, $\<\a_i, \mu_{i-1}\> \le -2$, that is, $\g:=s_1 \cdots s_{i-1}(\a_{i}) \in \Phi^+(\mu_0)$.

Assume, moreover, $\mu$ is dominant. We have $\<\g, \mu\> \le 1$ and $\<\g, \a^\vee\> \le -2$ because $\<\g, \a^\vee\> \ge -3$. If $\g \neq -\a$, then $\g$ and $\a_i$ are long roots of $\Phi$. Therefore, \begin{align*} \mu+\a^\vee+\g^\vee= \mu_0+\g^\vee &\le \mu_0 + \a_1^\vee+\cdots +\a_i^\vee \\ &\le \mu_1+ \a_2^\vee + \cdots + \a_i^\vee  \\ &\le \cdots \le \mu_{i-1} + \a_i^\vee \le \l. \end{align*} The proof is finished.
\end{proof}

\begin{lem} \cite[Proposition 2.2]{Ga} \label{Gashi}
Let $\mu \in Y$ and $\l \in Y^+$. If $\mu \le \l$ and $\mu$ is weakly dominant, then $\mu \preceq \l$. In particular, If $\mu \in Y^+$ and $\g \in \Phi^+$ is a positive long root such that $\mu+\g^\vee \le \l$, then $\mu+\g^\vee \preceq \l$.
\end{lem}

\begin{lem} \label{minus}
Let $\mu \in Y$ be a weakly dominant coweight and $\a \in \Pi_0$ such that $\<\a, \mu\> \ge 1$ (resp. $\<\a, \mu\>=-1$). Then $\mu-\a^\vee$ (resp. $\mu+\a^\vee$) is weakly dominant.
\end{lem}
\begin{proof}
Assume $\<\a, \mu\> \ge 1$. Let $\b \in \Phi^+$. If $\b=\a$, then $\<\a, \mu-\a^\vee\> \ge -1$. Otherwise, we have $\<\b, \mu\> \ge -1$ and $\<\b, s_\a(\mu)\> = \<s_\a(\b), \mu\> \ge -1$  since $s_\a(\b) \in \Phi^+$ and $\mu$ is weakly dominant. Note that $\mu-\a^\vee$ lies in the convex hull of $\mu$ and $s_\a(\mu)$, which means $\<\b, \mu-\a^\vee\>=-1$. Therefore, $\mu-\a^\vee$ is weakly dominant.

The case that $\<\a, \mu\> = -1$ can be handled similarly.
\end{proof}

\begin{lem} \label{con}
Let $\mu \neq \mu', \l \in Y^+$ with $\mu, \mu' \preceq \l$. Then $\Theta(\mu, \mu', \l) \neq \emptyset$.
\end{lem}
\begin{proof}
Let $\d, \d' \in \NN \Pi_0^\vee$ such that $\l=\mu+\d=\mu'+\d'$. We write $\mu-\mu'=\d'-\d=\sum_{k \in \L} \g_k^\vee$, where $\L$ is as in Lemma \ref{contraction}. Since $\mu \neq \mu'$, $\L$ is nonempty. Choose $k_0 \in \L$ and set $\a=\g_{k_0}$. Then $\<\a, \d'-\d\>=\<\a, \sum_{k \in \L} \g_k^\vee\> \ge 2$ and $\mu-\a^\vee, \mu'+\a^\vee \le \l$. If $\mu-\a^\vee, \mu'+\a^\vee \preceq \l$, the proof is finished. Otherwise, by symmetry, we may assume $\mu'+\a^\vee \npreceq \l$. Applying Lemma \ref{key}, there exists $\g \in \Phi^+$ such that $\<\g, \mu'+\a^\vee\> \le -2$. In particular, $\<\g, \a^\vee\> \le -2$ and $\<\g, \mu'\> \le 1$. We claim that $\g \neq -\a$. Otherwise, we have $2 \le \<\a, \mu-\mu'\>=-\<\g, \mu\>+\<\g, \mu'\> \le 0+1=1$, which is a contradiction and the claim is proved.

Therefore, by Lemma \ref{key} (b), $\g$ is a long coroot of $\Phi$ and $\mu+\a^\vee+\g^\vee \le \l$. Define $\b=s_\a(\g)$, which is again a long root. We show that $\b \in \Theta(\mu, \mu', \l)$. First note that \begin{align*}\<\b, \mu-\mu'\> &=\<\g-\<\g,\a^\vee\> \a, \mu-\mu'\> &\ge \<\g, \mu\> - \<\g, \mu'\> + 2\<\a,\mu-\mu'\> \\ & \ge 0-1+4=3. \end{align*} It remains to show $\mu-\b^\vee, \mu'+\b^\vee \preceq \l$. Note that $\mu'+\b^\vee=\mu'+\a^\vee+\g^\vee \le \l$ and $\mu-\b^\vee \le \mu-\a^\vee \le \l$. Thanks to Lemma \ref{Gashi}, it suffices to show $\mu-\b^\vee, \mu'+\b^\vee$ are weakly dominant. Let $\g' \in \Phi^+ - \{\pm \b\}$, then $\<\g',\mu'-\b^\vee\> \ge 0-1=-1$ since $\b$ is a long root. Similarly, $\<\g', \mu+\b^\vee\> \ge -1$. Without loss of generality, we may assume $-\b \in \Phi^+$, the other case can be handled similarly. Then $\<-\b, \mu-\b^\vee\> \ge 0+2=2$ and $\<-\b, \mu'+\b^\vee\>=\<-\b, \mu+\b^\vee\>+\<-\b, \mu'-\mu\> \ge -2+3 =1$. So $\mu-\b^\vee, \mu'+\b^\vee$ are weakly dominant as desired. Therefore, $\b \in \Theta(\mu, \mu', \l)$.
\end{proof}

\begin{prop} \label{con'}
Let $\mu \neq \mu' \in Y$ and $\l \in Y^+$ such that $\mu, \mu' \preceq \l$. Then $\Theta(\mu, \mu', \l) \neq \emptyset$.
\end{prop}
\begin{proof}
If  $\Xi(\mu, \mu') \neq \emptyset$, then $\Theta(\mu, \mu', \l) \supseteq \Xi(\mu, \mu')$ is nonempty. Otherwise, by Lemma \ref{pm}, there exists $w \in W_0$ such that $w(\mu), w(\mu') \in Y^+$. Then by Lemma \ref{con}, $\Theta(\mu, \mu', \l)=w\i(\Theta(w(\mu), w(\mu'), \l))$ is again nonempty.
\end{proof}

\

Let $D \subseteq \Phi$ be a subset. We say $D$ is orthogonal if $\<\a, \b^\vee\> =0$ for any $\a \neq \b \in D_\mu$.
\begin{lem} \label{weak}
Assume $\Phi$ is simply laced. Let $\mu \in Y$ be weakly dominant. Then there exists an orthogonal subset $D_\mu \subseteq \Phi^+$ such that

(a) $\<\g, \mu\> =-1$ for any $\g \in D_\mu$;

(b) $\bar \mu - \mu =\sum_{\a \in D_\mu} \a^\vee$.
\end{lem}
\begin{proof}
We argue by induction on $\bar \mu -\mu \in \NN \Pi_0^\vee$. If $\bar \mu -\mu=0$, then $D_\mu=\emptyset$ and there is nothing to prove. Otherwise, there exists a simple root $\a \in \Pi_0$ such that $\<\a, \mu\> \le -1$, and hence $\<\a, \mu\>=-1$ since $\mu$ is weakly dominant. Let $\upsilon=s_\a(\mu)=\mu+\a^\vee$, which is also weakly dominant by Lemma \ref{minus}. Noticing that $\bar \upsilon - \upsilon < \bar \mu - \mu$. So there exists, by induction hypothesis, an orthogonal subset $D_{\upsilon} \subseteq \Phi^+$ with properties (a) and (b). Let $$D_{\upsilon, -}=\{\b \in D_{\upsilon}; \<\b, \mu\>=-1\}=\{\b \in D_{\upsilon}; \<\b, \a^\vee\>=0\}$$ and $$D_{\upsilon, 0}=\{\b \in D_{\upsilon}; \<\b, \mu\>=0\}=\{\b \in D_{\upsilon}; \<\b, \a^\vee\>=-1\}.$$ Since $\mu$ is weakly dominant and $\Phi$ is simply laced, we have $D_{\upsilon}=D_{\upsilon,-} \cup D_{\upsilon,0}$. By the orthogonality of $D_{\upsilon}$, $\g:=\a+\sum_{\b \in D_{\upsilon,0}} \b \in \Phi^+$. Now we set $D_\mu=D_{\upsilon, -} \cup \{\g\}$, which satisfies (a) and (b).
\end{proof}

\begin{lem} \label{special} \label{orth}
Assume $\Phi$ is simply laced. Let $\mu, \upsilon, \l \in Y$ with $\mu, \upsilon \preceq \l$. If $\Theta(\mu, \upsilon, \l)=\Xi_1(\mu, \upsilon)$, then $\mu$ and $\upsilon$ are conjugate under $W_0$. Moreover, there exists an orthogonal subset $\D_{\mu, \upsilon} \subseteq \Xi_1(\mu, \upsilon)$ such that $\mu-\upsilon=\sum_{\b \in \D_{\mu, \upsilon}} \b^\vee$.
\end{lem}
\begin{proof}
By Lemma \ref{act}, the lemma dose not change if we replace the pair $(\mu, \upsilon)$ with $(w(\mu), w(\upsilon))$ for any $w \in W_0$. So we can assume $\mu \in Y^+$ and $\{\a \in \Phi^+; \<\a, \upsilon\> \le -1\} \subseteq \Xi(\mu, \upsilon)$. On the other hand, we have, by assumption, $\Xi(\mu, \upsilon) \subseteq \Theta(\mu, \upsilon, \l)=\Xi_1(\mu, \upsilon)$. Therefore, $$\{\a \in \Phi^+; \<\a, \upsilon\> \le -1\}=\Xi_1(\mu, \upsilon).$$ In particular, $\upsilon$ is weakly dominant.

We have to show $\mu= \bar \upsilon$. Assume otherwise. By lemma \ref{weak}, there exists an orthogonal subset $$D_{\upsilon} \subseteq \{\a \in \Phi^+; \<\a, \upsilon\> = -1\}=\Xi_1(\mu, \upsilon)$$ such that $\bar \upsilon - \upsilon=\sum_{\b \in D_{\upsilon}} \b^\vee$.

By Proposition \ref{con'}, $\Theta(\mu, \bar \upsilon, \l) \neq \emptyset$. Choose $\g \in \Theta(\mu, \bar \upsilon, \l)$. We claim

(a) $\pm \g \notin D_{\upsilon} \subseteq \Xi_1(\mu, \upsilon)$.

Indeed, if $-\g \in D_{\upsilon}$, one computes that \begin{align*}\<-\g, \mu-\upsilon\> &=\<-\g, \mu-\bar\upsilon\>+\<-\g, \bar\upsilon-\upsilon\> \\ &=\<-\g, \mu-\bar\upsilon\>+\<-\g, \sum_{\b \in D_{\upsilon}} \b^\vee\> \\ &=\<-\g, \mu-\bar\upsilon\> +2 \le -2+2=0,\end{align*} contradicting the inclusion $\g \in \Xi_1(\mu, \upsilon)$. If $\g \in D_{\upsilon}$, one computes similarly that $\<\g, \mu-\upsilon\> \ge 2+2=4$, which also contradicts the inclusion $\g \in \Xi_1(\mu, \upsilon)$. So (a) is proved.

Thanks to (a) and that $\Phi$ is simply laced, we have $$D_{\upsilon, \g}:=\{\b \in D_{\upsilon}; \<\g, \b^\vee\> \le -1\}=\{\b \in D_{\upsilon}; \<\g, \b^\vee\> = -1\}.$$ Moreover, by the orthogonality of $D_\upsilon$, we have $\a:=\g+\sum_{\b \in D_{\upsilon, \g}} \b \in \Phi$. Note that $\<\b, \a^\vee\> \ge 0$ for $\b \in D_{\upsilon}-D_{\upsilon, \g}$ and $\<\b', \a^\vee\> =1 $ for $\b' \in D_{\upsilon, \g}$. We claim that

(b) $\upsilon+\a^\vee, \mu-\a^\vee \preceq \l$.

We number the roots in $D_{\upsilon}-D_{\upsilon, \g}$ as $\b_1, \b_2, \dots, \b_r$ for some $r \in \NN$. We set $\upsilon_i=\upsilon+\a^\vee+\sum_{k=1}^i \b_k^\vee$. Note that $\upsilon_r = \bar \upsilon+\g^\vee \preceq \l$ since $\g \in \Theta(\mu, \bar \upsilon, \l)$. To show $\upsilon+\a^\vee \preceq \l$, it suffices to show $\upsilon_{i-1} \preceq \upsilon_i$ for $i \in [1, r]$, which follows from the observation: $\<\b_i, \upsilon_i\>=\<\b_i, \upsilon+\a^\vee+\b_i^\vee\> \ge -1+0+2= 1$. Now we show $\mu-\a^\vee \preceq \l$. Similarly, we number the roots in $D_{\upsilon, \g}$ as $\b_1', \b_2', \dots, \b_{r'}'$ and set $\mu_i=\mu-\g^\vee-\sum_{k=1}^i {\b_k'}^\vee$. Note that $\mu_{r'}=\mu-\a^\vee$ and $\mu_0=\mu-\g^\vee \preceq \l$. Now it suffices to show $\mu_{i-1} \succeq \mu_i$ for $i \in [1, r']$, which follows from the observation: $\<\b_i', \mu_{i-1}\>=\<\b_i', \mu-\g^\vee\> = 2$. The claim (b) is proved.

Now one computes that \begin{align*}\tag c \<\a, \mu-\upsilon\> &=\<\a, \mu-\bar \upsilon\> +\<\a, \bar \upsilon -\upsilon\> \\ &=\<\g, \mu-\bar \upsilon\> + \sharp D_{\upsilon, \g} + \sum_{\b \in D_{\upsilon}-D_{\upsilon, \g}}\<\a, \b^\vee\>,\end{align*} where the second equality follow from the fact that $\<\b, \mu\>=\<\b, \bar \upsilon\>=\<\b', \a^\vee\>=1$ for any $\b \in D_{\upsilon}$ and any $\b' \in D_{\upsilon, \g}$. So $\<\a, \mu-\upsilon\> \ge \<\g, \mu-\bar \upsilon\> \ge 2$ and $\a \in \Theta(\mu, \upsilon, \l)=\Xi_1(\mu, \upsilon)$. Thus $\<\a, \upsilon\>=-1$ and $\<\a, \mu-\upsilon\>=2$. By (c), we deduce that $\<\b, \g^\vee\>=0$ for any $\b \in D_{\upsilon}$. Therefore, $\a=\g$ and $\<\g, \mu\>=-\<\g, \upsilon\>=-\<\g, \bar \upsilon\> =1$, which contradicts that $\mu, \bar \upsilon \in Y^+$. Therefore, we must have $\mu = \bar \upsilon$. Now the ``Moreover" part follows from Lemma \ref{weak} by noticing that $\upsilon$ is weakly dominant and that $\{\a \in \Phi^+; \<\a, \upsilon\> \le -1\}=\Xi_1(\mu, \upsilon)$.
\end{proof}

\section{Proof of Proposition \ref{k1}} \label{sec k1}
The aim of this section is to prove Proposition \ref{k1}. Through out this section, we assume $G$ is adjoint and simple; the root system $\Phi$ of $G$ has $h$ connected components, on which $\s$ acts transitively.

\begin{lem} \label{equal}
Let $x, x' \in \bar I_{\l, J, b}$. Then there exist some $W_J$-conjugates $\mu, \mu' \in Y$ of $\mu_x$ and $\mu_{x'}$ respectively such that $\mu- \mu' \in (1-\s) \ZZ\Phi^\vee$.
\end{lem}
\begin{proof}
The proof is the same as that of \cite[Proposition 4.3.1]{CKV}, noticing that we only use the property that $\mu_x, \mu_{x'}$ are $J$-minuscule.
\end{proof}

Let $\a \in \Phi$, we denote by $\co_\a$ the $\s$-orbit of $\a$.
\begin{lem} \cite[Lemma 4.2.1 \& Example 4.2.2]{CKV} \label{graph}
We have the following properties:

(i) Let $\a \in \Phi$ and $\b \neq \g \in \co_\a$, then $\<\b, \g^\vee\> \le 0$. Moreover, the equality holds unless each connected component of $\Phi$ is of type $A_m$ with $m \in \NN$ even.

(ii) Assume $\a, \b$ are in the same connected component of $\Phi$ and $\<\s^h(\a), \a^\vee\>=\<\s^h(\b), \b^\vee\>=-1$. Then $\<\a, \b^\vee\> \neq 0$.
\end{lem}

\begin{lem} \rm{(cf. \cite[Proposition 4.3.2]{CKV})} \label{conv}
Let $\mu \neq \upsilon \in Y$ and $\l \in Y^+$ such that $\mu-\upsilon \in (1-\s) Y$ and $\mu, \upsilon \preceq \l$. Then there exist $\a \in \Phi$ and $r \in [1, h-1]$ if $\sharp \co_{\a}=h$ (resp. $r \in [1, h]$ if $\sharp \co_{\a}=2h$ and $r \in [1, 2h-1]$ if $\sharp \co_{\a}=3h$) such that

(a) $\s^r(\a) \neq \a$;

(b) either $\mu, \mu-\a^\vee, \mu+\s^r(\a^\vee), \mu-\a^\vee+\s^r(\a^\vee) \preceq \l$ or $\upsilon, \upsilon+\a^\vee, \upsilon-\s^r(\a^\vee), \upsilon+\a^\vee-\s^r(\a^\vee) \preceq \l$;

(c) $\mu-\upsilon-\a^\vee+\s^r(\a^\vee) \prec \mu-\upsilon$.
\end{lem}
\begin{proof}
By Proposition \ref{con'}, $\Theta(\mu, \upsilon, \l) \neq \emptyset$. Let $\g \in \Theta(\mu, \upsilon, \l)$. since $\mu-\upsilon \in (1-\s) Y$, we have $\sum_{k=1}^{\sharp \co_\g}\<\s^k(\g), \mu-\upsilon\> = 0$. Define $$E_{\g, \mu, \upsilon}=\{k \in [1, \sharp \co_\g]; \<\s^k(\g), \mu-\upsilon\> \le -1 \}.$$ Since $\<\g, \mu-\upsilon\> \ge 2$, we have $E_{\g, \mu, \upsilon} \neq \emptyset$.

Case(1): there exists $\g \in \Theta(\mu, \upsilon, \l)$, $j \in E_{\g, \mu, \upsilon}$ such that $\<\s^j(\g), \g^\vee\>=0$. Then $\s^j(\g) \neq \pm \g$. If $\<\s^j(\g), \upsilon\> \ge 1$, then $\<\s^j(\g), \upsilon\>=\<\s^j(\g), \upsilon+\g^\vee\> \ge 1$, which implies $\upsilon-\s^j(\g^\vee) \preceq \upsilon \preceq \l$ and $\upsilon+\g^\vee-\s^j(\g^\vee) \preceq \upsilon+\g^\vee \preceq \l$. Otherwise, we have $\<\s^j(\g), \mu\> \le -1$. Then $\<\s^j(\g), \mu\>=\<\s^j(\g), \mu-\g^\vee\> \le -1$ and we have $\mu+\s^j(\g^\vee) \preceq \mu \preceq \l$ and $\mu-\g^\vee+\s^j(\g^\vee) \preceq \mu-\g^\vee \preceq \l$ similarly. Moreover, note that $\<\s^j(\g), \mu-\upsilon-\g^\vee\>= \<\s^j(\g), \mu-\upsilon\> \le -1$ and $\<\g, \mu-\upsilon\> \ge 2$. One deduces that $$\mu-\upsilon-\g^\vee+\s^j(\g^\vee) \preceq \mu-\upsilon-\g^\vee \prec \mu-\upsilon.$$ If $\sharp \co_\a=h$, then $E_{\g, \mu, \upsilon} \subseteq [1, h-1]$, and we take $\a=\g$ and $r=j \in [1, h-1]$. If $\sharp \co_\g =2h$ (resp. $\sharp \co_\g=3h$). We set $\a=\g$ and $r=j$ if $j \le h$ (resp. $j \le 2h-1$) and set $\a=-\s^j(\g)$ and $r=\sharp \co_\g-j$ otherwise. Now one checks that the lemma follows in this case.

Case(2): Case(1) fails. By Lemma \ref{graph} (i), each connected component of $\Phi$ is of type $A_m$ with $m$ even. Moreover, we have $E_{\g, \mu, \upsilon}=\{h\}$, $\sharp \co_\g=2h$ and $\<\g, \s^h(\g)\>=-1$ for any $\g \in \Theta(\mu, \upsilon, \l)$. We take $r=h$ and choose $\a$ from $\Theta(\mu, \upsilon, \l)$. The precise choice of $\a$ depends on the following subcases. However, (a) always holds since $\<\s^h(\g), \mu-\upsilon\> \le -\<\g, \mu-\upsilon\> \le -2$ for any $\g \in \Theta(\mu, \upsilon, \l)$.

Case(2.1): there exists some $\g \in \Theta(\mu, \upsilon, \l)- \Xi_1(\mu, \upsilon)$. Note that $\<\s^h(\g), \mu-\upsilon\> \le -2$ since $E_{\g, \mu, \upsilon}=\{h\}$. So one of the following two cases occurs:

(2.1.1): $\<\s^h(\g), \mu\> \le -2$ or $\<\s^h(\g), \upsilon\> \ge 2$;

(2.1.2): $\<\s^h(\g), \mu\>=-\<\s^h(\g), \upsilon\>=-1$, that is, $-\s^h(\g) \in \Xi_1(\mu, \upsilon)$.

In (2.1.1), we take $\a=\g$. Without loss of generality, we may assume $\<\s^h(\a), \upsilon\> \ge 2$. Then $\<\s^h(\a), \upsilon\>, \<\s^h(\a), \upsilon+\a^\vee\> \ge 1$ and hence $\upsilon-\s^h(\a^\vee), \upsilon+\a^\vee-\s^h(\a^\vee), \upsilon+\a^\vee \preceq \l$. Therefore, (b) holds.

In (2.1.2), we take $\a=-\s^h(\g) \in \Xi_1(\mu, \upsilon) \subseteq \Theta(\mu, \upsilon, \l)$. Then either $\<\s^h(\a), \mu\> \le -2$ or $\<\s^h(\a), \upsilon\> \ge 2$ since $\<\s^h(\a), \mu-\upsilon\> \le -2$ and $-\s^h(\a)=\g \notin \Xi_1(\mu, \upsilon)$. Therefore, (b) follows similarly as in (2.1.1).

Case (2.2): $\Theta(\mu, \upsilon, \l)=\Xi_1(\mu, \upsilon)$.

(2.2.1): $\<\s^i(\g), \mu-\upsilon\> \ge 1$ for some $\g \in \Theta(\mu, \upsilon, \l)$ and some $i \in [1, 2h-1] - \{h\}$. Then $\<\s^h(\g), \mu-\upsilon\> \le -3$, which implies either $\<\s^h(\g), \mu\> \le -2$ or $\<\s^h(\g), \upsilon\> \ge 2$. Then (b) follows similarly as in (2.1.1) if we take $\a=\g$.

(2.2.2): $\<\s^k(\g), \mu-\upsilon\> =0$ for any $\g \in \Theta(\mu, \upsilon, \l)$ and any $k \in [1, 2h-1] - \{h\}$. We show this case dose not occur. By Corollary \ref{orth}, there exists an orthogonal subset $\D_{\mu, \upsilon} \subseteq \Xi_1(\mu, \upsilon)$ such that $\mu-\upsilon=\sum_{\b \in \D_{\mu, \upsilon}} \b^\vee$. Since $0 \neq \mu-\upsilon \in (1-\s)Y$, we have $\sharp \D_{\mu, \upsilon} \ge 2$. Choose $\g_1 \neq \g_2 \in \D_{\mu, \upsilon}$. If $\g_1, \g_2$ are in the same connected component of $\Phi$, it contradicts Lemma \ref{graph} (ii) in view of the assumption of Case (2). So each connected component of $\Phi$ contains at most one element of $\D_{\mu, \upsilon}$. Now let $k_0 \in [1, 2h-1]-\{h\}$ such that $\s^{k_0}(\g_1), \g_2$ are in the same connected component of $\Phi$. By the assumption of (2.2.2), $\<\s^{k_0}(\g_1), \g_2^\vee\>=\<\s^{k_0}(\g_1), \mu-\upsilon\>=0$, which again contradicts Lemma \ref{graph} (ii).

Finally, it remains to show (c) for Case (2). Since $\a \in \Theta(\mu, \upsilon, \l)$ and $E_{\a, \mu, \upsilon}=\{r\}$ (with $r=h$), we have $\<\s^r(\a), \mu-\upsilon\> \le -2$ and hence $\<\s^r(\a), \mu-\upsilon-\a^\vee\> \le -1$ by (a). Therefore,  $$\mu-\upsilon-\a^\vee+\s^r(\a^\vee) \preceq \mu-\upsilon-\a^\vee \prec \mu-\upsilon$$ and the proof is finished.
\end{proof}

\begin{cor} \label{conv'}
Let $\mu, \upsilon \in Y$ and $\l \in Y^+$ such that $\mu-\upsilon \in (1-\s) Y$ and $\mu, \upsilon \preceq \l$. Then there exist $\mu_j \in Y$, $\a_j \in \Phi$, and $r_j \in [1, h-1]$ if $\sharp \co_{\a_j}=h$ (resp. $r_j \in [1, h]$ if $\sharp \co_{\a_j}=2h$ and $r_j \in [1, 2h-1]$ if $\sharp \co_{\a_j}=3h$) such that $\mu_0=\mu$, $\mu_m=\upsilon$, $\mu_{j+1}-\mu_j=\a_j^\vee-\s^{r_j}(\a_j^\vee)$ and $$\mu_j, \ \mu_j+\a_j^\vee, \ \mu_j-\s^{r_j}(\a_j^\vee), \ \mu_j+\a_j^\vee-\s^{r_j}(\a_j^\vee)  \preceq \l$$  for $j \in [0, m-1]$.
\end{cor}
\begin{proof}
We argue by induction on $\mu-\upsilon$. If $\mu-\upsilon=0$, then $\mu=\upsilon$ and there is nothing to prove. We assume the lemma holds for any pair $(\mu', \upsilon')$ such that $\mu', \upsilon' \preceq \l$ and $\mu'-\upsilon' \prec \mu-\upsilon$. We show it also holds for the pair $(\mu, \upsilon)$. Applying Lemma \ref{conv}, there exists $\a \in \Phi$ and $r \in [1, h-1]$ if $\sharp \co_{\a}=h$ (resp. $r \in [1, h]$ if $\sharp \co_\a=2h$ and $r \in [1, 2h-1]$ if $\sharp \co_\a=3h$) such that $\mu-\upsilon+\a^\vee-\s^r(\a^\vee) \prec \mu-\upsilon$ and one of the following two statements holds:

(a) $\mu, \ \mu+\a^\vee, \ \mu-\s^r(\a^\vee), \ \mu+\a^\vee-\s^r(\a^\vee) \preceq \l$

(b) $\upsilon, \ \upsilon-\a^\vee, \ \upsilon+\s^r(\a^\vee), \ \upsilon-\a^\vee+\s^r(\a^\vee) \preceq \l$.

We set $\eta=\mu+\a^\vee-\s^r(\a^\vee)$ (resp. $\eta=\upsilon-\a^\vee+\s^r(\a^\vee)$) if (a) (resp. (b)) occurs. Then $\eta \preceq \l$ and $\eta-\upsilon \prec \mu-\upsilon$ (resp. $\mu-\eta \prec \mu-\upsilon$). The lemma then follows by induction hypothesis on the pair $(\eta, \upsilon)$ (resp. $(\mu, \eta)$).
\end{proof}

\

Inspired by \cite[Definition 4.4.8]{CKV}, we introduce the following notations, which play a crucial role throughout the paper.

Let $x, x' \in \pi(M_J)$. We write $x \overset {(\a, r)} \to x'$ for some $\a \in \Phi-\Phi_J$ and some $r \in \NN$ if $x'-x=\a^\vee-\s^r(\a^\vee)$ and $\mu_x, \mu_{x+\a^\vee}, \mu_{x-\s^r(\a^\vee)}, \mu_{x'} \preceq \l$. We write $x \overset {(\a, r)} \rightarrowtail x'$ if $x \overset {(\a, r)} \to x'$ and neither $$x \overset {(\a, i)} \to x+\a^\vee-\s^i(\a^\vee) \overset {(\s^i(\a), r-i)} \to x'$$ nor $$x \overset {(\s^i(\a), r-i)} \to x+\s^i(\a^\vee)-\s^r(\a^\vee) \overset {(\a, i)} \to x'$$ for any $i \in [1, r-1]$.

\begin{lem} \label{convv'}
Let $x, x' \in \bar I_{\l, J, b}$ such that $x \overset {(\a, r)} \to x'$ for some $\a \in \Phi-\Phi_J$ and some $r \in \NN$. Then there exist $x_j \in \bar I_{\l ,J, b}$, $\a_j \in \co_\a$ and $r_j \in [1, r]$ such that $x_0=x$, $x_m=x'$ and $x_j \overset {(\a_j, r_j)} \rightarrowtail x_{j+1}$ for $j \in [0, m-1]$. Moreover, we have $\sum_{i=0}^{r-1} \s^i(\a^\vee)=\sum_{j=0}^{m-1} \sum_{k=0}^{r_j-1} \s^k(\a_j^\vee)$.
\end{lem}
\begin{proof}
It follows directly by induction on $r$.
\end{proof}

\begin{lem} \rm{(cf. \cite[Lemma 4.4.1]{CKV})} \label{minu}
Let $J \subseteq S_0$. Then the following two statements holds.

(a) For $\b \in \Phi-\Phi_J$, there exists a unique $\b_J \in \Phi-\Phi_J$ such that $\b_J^\vee - \b^\vee \in \ZZ \Phi_J^\vee$ and $\b_J^\vee$ is $J$-anti-dominant and $J$-minuscule.

(b) Let $\l \in Y^+$ and $\mu \in Y$ be a $J$-dominant and $J$-minuscule coweight. If $\mu+\b^\vee \preceq \l$, then $\mu+\b_J^\vee \preceq \l$.
\end{lem}
\begin{proof}
Denote by $\b_J'$ the unique $J$-anti-dominant conjugate of $\b$ under $W_J$. Then $\mu+{\b_J'}^\vee \preceq \mu+\b^\vee \preceq \l$ since $\mu$ is $J$-dominant. If ${\b_J'}^\vee$ is $J$-minuscule, we set $\b_J=\b_J'$ and the proof is finished. Otherwise, $\<\g, {\b_J'}^\vee\> \le -2$ for some $\g \in \Phi_J^+$. Since $\b \notin \Phi_J$, $\b \neq -\g$ and hence $\g$ is a long root in $\Phi$. Let $\a=s_{\b_J'}(\g)$. Note that $\<\g, \mu+{\b_J'}^\vee\> \le 1-2=-1$. We have $\mu+\a^\vee=\mu+{\b_J'}^\vee+\g^\vee \preceq \l$. Set $\b_J=\a_J'$. Then $\b_J^\vee$ is $J$-minuscule since $\a$ is a long root of $\Phi$. Moreover, we have, we have $\mu+\b_J^\vee \preceq \l$ (since $\mu+\a^\vee \preceq \l$) and $\b_J^\vee-\b^\vee \in \ZZ \Phi_J^\vee$, as desired.
\end{proof}

\begin{lem} \rm{(cf. \cite[Proposition 4.4.10]{CKV})} \label{convv}
Let $x \neq x' \in \bar I_{\l, J, b}$. Then there exist $x_j \in \bar I_{\l ,J, b}$, $\a_j \in \Phi-\Phi_J$ and $r_j \in \NN$ for $j \in [0, m-1]$ such that

(1) $\a_j^\vee$ is $J$-anti-dominant and $J$-minuscule;

(2) $r_j \in [1, h]$ if $\sharp \co_{\a_j} \in \{h, 2h\}$ and $r_j \in [1, 2h-1]$ if $\sharp \co_{\a_j}=3h$;

(3) $x_0=x$, $x_m=x'$ and $x_j \overset {(\a_j, r_j)} \to x_{j+1}$.
\end{lem}
\begin{proof}
By Lemma \ref{equal}, there exist some $W_J$-conjugates $\mu, \upsilon \in Y$ of $\mu_x, \mu_{x'}$ respectively such that $\mu- \upsilon \in (1-\s) \ZZ\Phi^\vee$. Note that $\mu, \upsilon \preceq \l$. By Corollary \ref{conv'}, there exist $\mu_j \in Y$, $\b_j \in \Phi$ and $t_j \in [1, h-1]$ if $\sharp \co_{\b_j}=h$ (resp. $t_j \in [1, h]$ if $\sharp \co_{\b_j}=2h$, and $t_j \in [1, 2h-1]$ if $\sharp \co_{\b_j}=3h$) such that $\mu_0=\mu$, $\mu_{m'}=\upsilon$, $\mu_{j+1}-\mu_j=\b_j^\vee-\s^{t_j}(\b_j^\vee)$ and $\mu_j, \mu_j+\b_j^\vee, \mu_j-\s^{t_j}(\b_j^\vee) \preceq \l$ for $j \in [0, m'-1]$.

Set $x_0=\mu_0 \in \pi_1(M_J)$. Assuming $x_k$ and $\a_{k-1}$ are already constructed for $k \in [0, j]$, we now construct $x_{j+1} \in \pi_1(M_J)$ and $\a_j \in \Phi-\Phi_J$ such that $x_j \overset {(\a_j, t_j)} \to x_{j+1}$. If $x_j=\mu_{m'} \in \pi_1(M_J)$, the construction is finished. Otherwise, let $i_j=\max\{i \in [1, m']; \mu_i=x_j \in \pi_1(M_J)\}$ and set $x_{j+1}=\mu_{i_j+1}$ and $\a_j=(\b_{i_j})_J$, where $(\b_{i_j})_J$ is defined (for $\b_{i_j}$) as in Lemma \ref{minu}. In this way, we obtain $$x=x_0 \overset {(\a_0, t_0)} \to x_1 \overset {(\a_1, t_1)} \to \cdots \overset {(\a_{m-1}, t_{m-1})} \to x_m=x'$$ such that $x_0, x_1, \dots, x_m$ are distinct elements of $\pi_1(M_J)$. In particular, $\a_j \in \Phi - \Phi_J$ for $i \in [0, m]$. By Lemma \ref{minu} (a), $\s^i(\a_j)=\a_j$ if $\s^i(\b_j)=\b_j$. Thus either $\sharp \co_{\a_j}=\sharp \co_{\b_j}$ or $\sharp \co_{\a_j}=h < \sharp \co_{\b_j}$. Define $r_j=t_j$ if $t_j \le h$ or $\sharp \co_{\a_j}=\sharp \co_{\b_j}$ and define $r_j=t_j-h$ otherwise. One checks that (2) and (3) are satisfied.
\end{proof}

\begin{prop} \rm{(cf. \cite[Proposition 4.5.4]{CKV})}\label{main}
Let $x, x' \in \bar I_{\l, J, b}$, $\a \in \Phi-\Phi_J$ and $r \in \NN$ such that

(1) $\a^\vee$ is $J$-anti-dominant and $J$-minuscule;

(2) $r \in [1, h]$ if $\sharp \co_\a \in \{h, 2h\}$ and $r \in [1, 2h-1]$ if $\sharp \co_\a=3h$;

(3) $x \overset {(\a, r)} \to x'$.

Then for any $P \in X_{\mu_x}^{M_J}(b)$, there exists $P' \in X_{\mu_{x'}}^{M_J}(b)$ satisfying $P \sim_{\l, b} P'$ and $\eta_J(P)-\eta_J(P')=\sum_{i=0}^{r-1} \s^i(\a^\vee) \in \pi_1(M_J)$.
\end{prop}

\begin{proof}[Proof of Proposition \ref{k1}]
It follows by combining Lemma \ref{convv} with Proposition \ref{main}.
\end{proof}

\

The rest of this section is devoted to the proof of Proposition \ref{main}.

\begin{lem} \cite[Lemma 4.4.7]{CKV} \label{central'}
Let $x \in \pi_1(M_J)$ and $\b \in \Phi-\Phi_J$ such that $\b^\vee$ is $J$-anti-dominant.

(a) $\<w(\b), \mu_x\> \le \<w_x(\b), \mu_x\>$ for any $w \in W_J$;

(b) $w_x(\b)$ is the unique minimal element in $\{w(\b); w \in W_J, \<w(\b), \mu_x\>=\<w_x(\b), \mu_x\>\}$;

(c) If $w_x(\b)=\b$, then $\mu_x$ is central on the connected components of $\Phi_J$ on which $\b^\vee$ is not central;

(d) in particular, (c) happens if $\<\b, \mu_x\>=\<w_x(\b), \mu_x\>$.

Here $w_x \in W_J$ satisfies $t^{\mu_x} w_x \in \Omega_J$, see \S \ref{setup4}.
\end{lem}

\begin{lem} \cite[Lemma 4.4.5]{CKV} \label{minuscule}
Let $x \in \pi_1(M_J)$ and $\d, \d' \in Y$ which are $J$-anti-dominant and $J$-minuscule. Then $\mu_x+\d$, $\mu_x-w_x(\d')$ and $\mu_x+\d-w_x(\d')$ are $J$-minuscule. In particular, they are conjugate to $\mu_{x+\d}$,$\mu_{x-\d'}$ and $\mu_{x-\d+\d'}$ under $W_J$ respectively.
\end{lem}

\begin{lem} \cite[Lemma 4.5.1]{CKV} \label{central}
Let $x \in \pi_1(M_J)$ and $\b, \g \in \Phi$ such that $\mu_{x+\b^\vee}, \mu_{x-\g^\vee} \preceq \l$. If $\a \in \Phi-\Phi_J$ satisfies

(i) $\mu_{x+\b^\vee-\a^\vee}, \mu_{x+\a^\vee-\g^\vee} \npreceq \l$;

(ii) neither $\b$ nor $\g$ lies in the connected component of $\Phi$  containing $\a$.

Then $\<w(\a), \mu_x\>=0$ for any $w \in W_J$. In particular, $w_x(\a)=\a$ and $\<\a, \mu_x\>=0$.
\end{lem}

\begin{cor}\rm{(cf. \cite[Remark 4.5.2]{CKV})} \label{central''}
Let $x, x' \in \bar I_{\l, J, b}$ such that $x \overset {(\a, r)} \rightarrowtail x'$ for some $\a \in \Phi-\Phi_J$ and some $r \in \NN$. Then $w_x(\s^i(\a))=\s^i(\a)$ and $\<\s^i(\a), \mu_x\>=0$ if $i \in [1, r-1]$ and $i, i-r \notin \ZZ h$.
\end{cor}

\

\begin{lem} \label{control}
Let $\b \in \Phi$, $\mu \in Y$ and $x \in \bold{k}[[t]]^\times$. Then $U_\b(x t\i) t^\mu$ lies in $ K t^\mu K$ if $\<\b, \mu\> \le -1$ and lies in $K t^{\b^\vee+\mu} K$ otherwise. Similarly, $t^\mu U_\b(x t\i) $ lies in $ K t^\mu K$ if $\<\b, \mu\> \ge 1$ and lies in $K t^{\mu-\b^\vee} K$ otherwise.
\end{lem}
\begin{proof}
We only need to prove the first statement. If $\<\b, \mu\> \le -1$, then $U_\b(x t\i) t^\mu= t^\mu U_\b(x t^{-\<\b, \mu\>-1}) \in K t K$. Otherwise, we have $$U_\b(x t\i) t^\mu \in U_{-\b}(x\i t) s_\b \b^\vee(x\i) t^{\b^\vee} U_{-\b}(x\i t) t^\mu K \subseteq K t^{\mu+\b^\vee} K$$ as desired.
\end{proof}

\begin{lem} \label{bound}
Let $\mu \in Y$, $\l \in Y^+$ and $\a, \b \in \Phi$ such that the root system $\Phi \cap (\ZZ \a + \ZZ \b)$ is of type $A_2$ or $A_1 \times A_1$ or $A_1$, and $$\mu, \mu+\a^\vee, \mu-\b^\vee, \mu+\a^\vee-\b^\vee \preceq \l.$$ Then $U_\a(y t\i) t^\mu U_\b(z t\i) \in \cup_{\l' \preceq \l} K t^{\l'} K$, where $y, z \in \bold{k}[[t]]$.
\end{lem}
\begin{proof}
Without loss of generality, we may assume $\Phi \cap (\ZZ \a + \ZZ \b)$ is of type $A_2$. In particular, $\a \neq \pm \b$.

First, we suppose either $\<\a, \mu\> \le -2$ or $\<\a, \mu\> = -1$ and $\<\a, \b^\vee\> \ge 0$. Note that $\Phi \cap (\ZZ \a + \ZZ \b)$ is of type $A_2$. One computes that \begin{align*} & \quad\ U_\a(y t\i) t^\mu U_\b(z t\i) \\ & = t^\mu U_\a(y t^{-1-\<\a, \mu\>}) U_\b(z t\i) \\ &=\begin{cases} t^\mu U_\b(z t\i)  U_\a(y t^{-1-\<\a, \mu\>}), & \text{ if } \<\a, \b^\vee\> \ge 0; \\ t^\mu U_\b(z t\i)  U_\a(y t^{-1-\<\a, \mu\>}) U_{\a+\b}(c_{\a, \b} y z t^{-2-\<\a, \mu\>}), & \text{ otherwise.} \end{cases} \\ & \subseteq K t^\mu K \cup K t^{\mu-\b^\vee} K, \end{align*} Here $c_{\a, \b} \in \bold{k}[[t]]$ is some constant and the inclusion follows from Lemma \ref{control}.

Now we suppose otherwise, that is, either $\<\a, \mu\> \ge 0$ or $\<\a, \mu\> = -1$ and $\<\a, \b^\vee\> = -1$. One checks that \begin{align*} & \quad\ U_\a(y t\i) t^\mu U_\b(z t\i) \\ & = U_{-\a}(y\i t) s_\a \a^\vee(y\i) t^{\mu+\a^\vee} U_{-\a}(y\i t^{1+\<\a, \mu\>}) U_\b(z t\i) \\ & \in \begin{cases} K t^{\mu+\a^\vee} U_\b(z t\i)  U_{-\a}(y\i t^{1+\<\a, \mu\>}), & \text{ if } \<\a, \b^\vee\> \le 0; \\ K t^{\mu+\a^\vee} U_\b(z t\i)  U_{-\a}(y\i t^{1+\<\a, \mu\>}) U_{\b-\a}(c_{-\a, \b} y\i z t^{\<\a, \mu\>}), & \text{ otherwise. } \end{cases} \\ & \subseteq K t^{\mu+\a^\vee} K \cup  K t^{\mu+\a^\vee-\b^\vee} K,  \end{align*} where $c_{-\a, \b} \in \bold{k}[[t]]$ is some constant and the last inclusion follows from Lemma \ref{control}. The proof is finished.
\end{proof}

\begin{lem}\label{simple}
Let $x, x' \in \bar I_{\l, J, b}$, $\a \in \Phi$ and $r \in \NN$ such that $x'-x=\a^\vee-\s^r(\a^\vee) \in \pi_1(M_J)$. Then the following two statements are equivalent:

(a) For any $P \in X_{\mu_x}^{M_J}(b)$, there exists $P' \in X_{\mu_{x'}}^{M_J}(b)$ satisfying $P \sim_{\l, b} P'$ and $\eta_J(P)-\eta_J(P')=\sum_{i=0}^{r-1} \s^i(\a^\vee) \in \pi_1(M_J)$.

(b) There exist $P \in X_{\mu_x}^{M_J}(b)$ and $P'' \in X_{\mu'}^{M_J}(b)$ such that $P \sim_{\l, b} P''$ and $\eta_J(P)-\eta_J(P'')=\sum_{i=0}^{r-1} \s^i(\a^\vee)$ for some $\mu' \in Y$.
\end{lem}
\begin{proof}
We only need to show (b) $\Rightarrow$ (a). Since $J_b^{M_J}(F)$ acts transitively on $\pi_0(X_{\mu_x}^{M_J}(b))$, it suffices to show there exists $P' \in X_{\mu_{x'}}^{M_J}(b)$ such that $P \sim_{\l, b} P'$ and $\eta_J(P)-\eta_J(P')=\sum_{i=0}^{r-1} \s^i(\a^\vee) \in \pi_1(M_J)$. In view of the conditions of (b), we have $\mu'=\mu_{x'} \in \pi_1(M_J)$ and $\mu' \preceq \l$. By Corollary \ref{basic} (i), there exists $P' \in X_{\mu_{x'}}^{M_J}(b)$ with $P'' \sim_{J, \mu', b} P'$. Then we have $P \sim_{\l, b} P'$ and $\eta_J(P)-\eta_J(P')=\sum_{i=0}^{r-1} \s^i(\a^\vee)$ as desired.
\end{proof}

Now we are ready to prove Proposition \ref{main}.
\begin{proof}[Proof of Proposition \ref{main}]
Thanks to Lemma \ref{convv'}, we can and do assume $x \overset {(\a, r)} \rightarrowtail x'$. Moreover, by Lemma \ref{simple}, it suffices to show

(i) there exist $P \in X_{\mu_x}^{M_J}(b)$ and $P'' \in X_{\mu'}^{M_J}(b)$ such that $P \sim_{\l, b} P''$ and $\eta_J(P)-\eta_J(P'')=\sum_{i=0}^{r-1} \s^i(\a^\vee)$ for some $\mu' \in Y$.

Let $g_x \in M_J(L)$ such that $g_x\i b \s(g_x)=\dot b_x=t^{\mu_x} \dot w_x$, where $b_x=t^{\mu_x} w_x \in \Omega_J$ is define in \S \ref{setup4}. By Corollary \ref{central''}, we have

(a) $w_x(\s^i(\a))=\s^i(\a)$, $\<\s^i(\a), \mu_x\>=0$ and hence $\<\s^i(\a), \b\>=0$ (by Lemma \ref{central'} (c)) if $i \in [1, r-1]$ and $i, i-r \notin \ZZ h$.

Case(1): $r \in [1, h]$. Define $\textsl{g}: \PP^1=\AA^1 \cup \{\infty\} \to G(L) / K$ such that $$\textsl{g}(z)=g_x \ U_{\a}(z t\i) \ {}^{\dot b_x\s} U_{\a}(z t\i) \cdots {}^{(\dot b_x\s)^{r-1}} U_{\a}(z t\i) K.$$ One computes that \begin{align*} & \quad \ \textsl{g}(z)\i b\s \textsl{g}(z) \\ &= K {}^{(\dot b_x\s)^{r-1}} U_{\a}(-z t\i)) \cdots U_{\a}(-z t\i) g_x\i b\s g_x U_{\a}(z t\i) \cdots {}^{(\dot b_x\s)^{r-1}} U_{\a}(z t\i) K \\ &=K {}^{(\dot b_x\s)^{r-1}} U_{\a}(-z t\i) \cdots U_{\a}(-z t\i) \dot b_x\s U_{\a}(z t\i) \cdots {}^{(\dot b_x\s)^{r-1}} U_{\a}(z t\i) K \\ &= K U_{\a}(-z t\i) \ b_x\s \ {}^{(\dot b_x\s)^{r-1}} U_{\a}(z t\i) K \\ &= K U_{\a}(-z t\i) t^{\mu_x} U_{w_x(\s^r(\a))}(c \s^r(z) t\i) \dot w_x \s K \\ & \in \cup_{\l' \preceq \l} K t^{\l'}K \s \subseteq K \backslash G(L) / K,\end{align*} where $c \in \bold{k}[[t]]^\times$ is some constant; the third equality follows from that $\a, \s^i(\a)$ lie in different connected components of $\Phi$ for $i \in [1, r-1]$; the fourth equality follows from (ii); the last inclusion follows from Lemma \ref{minuscule} and Lemma \ref{bound}. Moreover, one computes that \begin{align*} P'':=\textsl{g}(\infty) &= \lim_{z \to \infty} g_x \prod_{i=0}^{r-1} U_{-\s^i(\a)}(c_i \s^i(z\i) t) t^{-\s^i(\a^\vee)} K \\ &= g_x \ t^{-\sum_{i=0}^{r-1} \s^i(\a^\vee)} K  \in M_J(L) / (M_J(L) \cap K), \end{align*} where $c_i \in \bold{k}[[t]]^\times$ is some constant. So $P'' \in X_{\mu_{x'}}^{M_J}(b)$, $P=\textsl{g}(0) \sim_{\l, b} \textsl{g}(\infty)=P''$ and $\eta_J(P)-\eta_J(P'')=\sum_{i=0}^{r-1} \s^i(\a^\vee)$ as desired.

\

Case(2): $r \in [h+1, 2h-1]$. Then $\sharp \co_\a=3h$ and each connected component of $\Phi$ is of type $D_4$. Since $\a^\vee$ is $J$-anti-dominant and $J$-minuscule, we have either $J = \emptyset$ or $\Pi_J = \co_\b$, where $\b$ is the unique simple root in $\Pi_J$ such that $\s^h(\b)=\b$ and $\b, \a$ are in the same connected component of $\Phi$.

We claim that

(2-a) If $\<\b, \a^\vee\>=-1$ and $\<\s^r(\b), \mu_x\>=1$, then $\<\s^r(\b), \mu_{x'}\>=0$.

Indeed, if $\<\s^r(\b), \mu_x\>=1$, then $w_x(\s^i(\a))=\s^i(\a)+\s^r(\b)$ for $i \in r+\ZZ h$. Using the relation $x'=x+\a^\vee-\s^r(\a^\vee)$, we have \begin{align*} \<\s^r(\b), \mu_{x'}\> & =\<\s^r(\b), s_\b(\mu_x+\a^\vee-w_x(\s^r(\a^\vee)))\> \\ & =\<\s^r(\b), \mu_x-\s^r(\a^\vee)-\s^r(\b^\vee)\>=0\end{align*} and (2-a) is verified. 

Note that $x' \overset {(-\a, r)} \rightarrowtail x$ is equivalent to $x \overset {(\a, r)} \rightarrowtail x'$. Therefore, by Lemma \ref{simple}, the proposition dose not change if we exchange the roles of $x$ and $x'$. Thus, we may and do assume

(2-b) $\<\s^r(\b), \mu_x\>=0$ if $\<\b, \a^\vee\>=-1$.

Then we show that

(2-c) $\<\s^{r-h}(\a), \mu_x\>=\<\s^h(\a), \mu_x\>=0$.

Note that $\<\b, \a^\vee\> \in \{0, -1\}$. So by (2-b) if it applies, we always have $$\mu_x-\s^r(\a)=\mu_x-w_x(\s^r(\a)) \preceq \l.$$ Then $\<\s^{r-h}(\a), \z\>=\<\s^{r-h}(\a), \mu_x\>$, where $\z \in A(x, \a, r):=\{\mu_x, \mu_x+\a^\vee, \mu_x-w_x(\s^r(\a^\vee)), \mu_x+\a^\vee-w_x(\s^r(\a^\vee))\}$. By Lemma \ref{minuscule}, we know that $A(x, \a, r) \preceq \l$. Thus one checks that \begin{align*} & x \overset {(\a, r-h)} \to x+\a^\vee-\s^{r-h}(\a^\vee) \overset {(\s^{r-h}(\a), h)} \to x' \text{ if } \<\s^{r-h}(\a), \mu_x\> \ge 1; \\ & x \overset {(\s^{r-h}(\a), h)} \to x+\s^{r-h}(\a^\vee)-\s^r(\a^\vee) \overset {(\a, r-h)} \to x' \text{ if } \<\s^{r-h}(\a), \mu_x\> \le -1. \end{align*} Hence $\<\s^{r-h}(\a), \mu_x\>=0$ since $x \overset {(\a, r)} \rightarrowtail x'$. The equality $\<\s^h(\a), \mu_x\>=0$ follows similarly and (2-c) is proved.

Case(2.1): $\<\b, \a^\vee\>=-1$ and $\<\b, \mu_x\>=1$. To show (i), we define $\textsl{g}: \PP^1 \to G(L)/K$ such that $$\textsl{g}(z)=g_x \ {}^{(\dot b_x\s)^{r-1}} U_{\a}(z t\i) \ {}^{(\dot b_x\s)^{r-2}} U_{\a}(z t\i) \cdots U_{\a}(z t\i) K.$$ Let $P''=\textsl{g}(\infty) \in M_J(L) / (K \cap M_J(L))$. By computations on on page 45-47 of \cite{CKV}, we have $\eta_J(P)-\eta_J(P')=\sum_{i=0}^{r-1} \s^i(\a^\vee) \in \pi_1(M_J)$. In view of (i), it remains to show $\textsl{g}\i b \s(\textsl{g}) \subseteq \cup_{\l' \preceq \l} K t^{\l'} K \subseteq K \backslash G(L) / K$, which holds, by (a), (2-b) and (2-c) (see page 45 of \cite{CKV}), if

(2.1.1) $U_\a(y t\i) t^{\mu_x} \in \cup_{\l' \preceq \l} K t^{\l'} K$ for $y \in \bold{k}[[t]]$;

(2.1.2) $U_{\s^{r-h}(\a)}(y t\i) U_{\s^r(\a)+\s^r(\b)}(z t^{\<\s^r(\a), \mu_x\>}) U_{\s^{r-h}(\a)}(-y t\i) \in K$ for $y, z \in \bold{k}[[t]]$.

Since $\mu_x, \mu_x+\a^\vee \preceq \l$, (2.1.1) follows from Lemma \ref{bound}. To verify (2.1.2), it suffices to show $\<\s^r(\a), \mu_x\> \ge 1$. Indeed, assume $\<\s^r(\a), \mu_x\> \le 0$. Then $\<\s^r(\a+\b), \mu_x-\s^r(\a^\vee)-\s^r(\b^\vee)-\s^{r-h}(\a^\vee)\> \le -1$, which implies $$\mu_x-\s^{r-h}(\a^\vee) \preceq \mu_x-\s^r(\a^\vee)-\s^r(\b^\vee)-\s^{r-h}(\a^\vee) = s_{\s^{r-h}(\b+\a)}(\mu_x-\s^r(\a^\vee)) \preceq \l,$$ where the second equality follows from $\<\s^r(\b)+\s^{r-h}(\a), \mu_x-\s^r(\a^\vee)\>=1$. Thus $$\mu_x+\a^\vee-\s^{r-h}(\a^\vee), \mu_x+\a^\vee-\s^r(\a^\vee)-\s^r(\b^\vee)-\s^{r-h}(\a^\vee) \preceq \l$$ since $\a, \s^r(\a)$ are not in the same connected component of $\Phi$. Then $$x \overset {(\a, r-h)} \to x+\a^\vee-\s^{r-h}(\a^\vee) \overset {(\s^{r-h}(\a), h)} \to x',$$ contradicting the assumption $x \overset {(\a, r)} \rightarrowtail x'$. So we have $\<\s^r(\a), \mu_x\> \ge 1$ as desired.

Case(2.2): either $\<\b, \a^\vee\>$ or $\<\b, \mu_x\>=0$. By (2-b) if it applies, $w_x$ fixes $\s^{r-h}(\a)$ and $\s^h(\a)$. Combining (a) and (2-c), we have $\<\s^j(\a), \mu_x\>=0$ and $w_x(\s^j(\a))=\s^j(\a)$ for $j \in [1, r-1]$. Noticing that $\<\s^{i-h}(\a), \s^i(\a^\vee)\>=0$ for $i \in \ZZ$, the proposition follows similarly as in Case(1).
\end{proof}

\section{Proof of Proposition \ref{k2}} \label{sec k2}
The main aim of this section is to prove Proposition \ref{k2}. Through out this section, we assume $G$ is adjoint and simple; the root system $\Phi$ of $G$ has $h$ connected components, on which $\s$ acts transitively.

\

For $\mu, \l \in Y$, we define $\Upsilon^+(\mu, \l)=\{\a \in \Phi^+; \mu+\a^\vee \preceq \l\}$.
\begin{lem}\rm{(cf. \cite[Proposition 4.1.9]{CKV})}\label{gen}
Let $J \subseteq S_0$ with $\s(J)=J$, $b \in G(L)$ and $\l \in Y^+$ such that $(J, b)$ is admissible and $(\l, b)$ is HN-irreducible. For $x \in \bar I_{\l, J, b}$, we define $$ C_{\l, J, b, x}=\{\b \in \Upsilon^+(\mu_x, \l); \text{$\b^\vee$ is $J$-minuscule and $J$-anti-dominant}\}.$$ Then the coroot lattice $\ZZ \Phi^\vee$ is spanned by $\Phi_J^\vee$ and $\cup_{\b \in C_{\l, J, b, x}} \co_\b^\vee$.

\

If, moreover, $\Phi$ is a union root systems of type $G_2$ and $J=\Pi_{\text{short}, G_2}$. Then $\ZZ \Phi^\vee$ is still generated by $\Pi_J^\vee$ and $\cup_{\a \in C_{\l, J, b, x} \cap \Pi_{\text{long}, G_2}} \co_\a^\vee$. Here $\Pi_{\text{short}, G_2}$ (resp. $\Pi_{\text{long}, G_2}$) denotes the single $\s$-orbit of simple short (resp. long) roots in $\Pi_0$.
\end{lem}
\begin{proof}
By Lemma \ref{dominant} and that $(\l, b)$ is HN-irreducible, we have $$\l^\diamond- \nu_{b_x}=\l^\diamond- \nu_{[b]}^G=\sum_{\a \in \Pi_0} c_\a \a^\vee,$$ where $b_x \in \Omega_J$ is defined in \S \ref{setup4} and each $c_\a \in \QQ$ is strictly positive. Therefore,

(a) for each $\a \in \Pi_0 - \Pi_J$, there exists $\th \in \co_\a$ such that $\mu_x+\th^\vee \le \l$.

We claim that

(b) $\ZZ \Phi^\vee$ is generated by $\Phi_J^\vee$ and $\cup_{\a \in \Upsilon^+(\mu_x, \l)} \co_\a^\vee$.

Assume (b). For $\a \in \Upsilon^+(\mu_x, \l) - \Phi_J$, let $\a_J$ be defined as in Lemma \ref{minu}. Then $\a_J \in C_{\l, J, b, x}$ and $\a^\vee-\a_J^\vee \in \ZZ \Phi_J^\vee$. Therefore, $\ZZ \Phi^\vee$ is also spanned by $\ZZ \Phi_J^\vee$ and $\co_{\a_J}^\vee \subseteq \cup_{\b \in C_{\l, J, b, x}} \co_\b^\vee$ for $\a \in \Upsilon^+(\mu_x, \l)$, as desired.

It remains to show (b). Let $A$ be the lattice spanned by $\Phi_J^\vee$ and and $\cup_{\a \in \Upsilon^+(\mu_x, \l)} \co_\a^\vee$. Let $J_0 \subseteq S_0$ be the maximal subset such that $\Phi_J^\vee \subseteq A$. Note that $J \subseteq J_0$ and $\s(J_0)=J_0$. We have to show that $J_0=S_0$. Assume otherwise, we choose $w \in W_{J_0}$ such that $\mu=w(\mu_x)$ is $J_0$-dominant.

Case(1): $\mu \notin Y^+$. Then there exists $\b \in \Pi - \Pi_{J_0}$ such that $\<\b, \mu\> \le -1$, which implies $w\i(\b) \in \Upsilon^+(\mu_x, \l)$. Since $w \in W_{J_0}$ and $\Phi_{J_0}^\vee \subseteq A$, we deduce that $\b^\vee \in A$, contradicting the maximality of $J_0$.

Case(2): $\mu \in Y^+$. Choose $\a \in \Pi - \Pi_{J_0}$. Since $\mu-\mu_x \in \ZZ \Phi_{J_0}^\vee$ and $\mu \le \l$, there exists, by (a), $\th \in \co_\a$ such that $\mu+\th^\vee \le \l$.

Subcase(2.1): $\mu+\th^\vee \preceq \l$. Then we have $w\i(\th) \in \Upsilon^+(\mu_x, \l)$ and $\th^\vee \in A$ as in Case(1), a contradiction.

Subcase(2.2): $\mu+\th^\vee \npreceq \l$. Notice that $\th \in \Phi^+$. Thus Lemma \ref{key} (b) applies, that is, there exists $\g \in \Phi^+$ such that $\mu+\g^\vee \le \mu+\th^\vee+\g^\vee \le \l$ and $\g^\vee, s_\th(\g^\vee)=\th^\vee+\g^\vee $ are short positive coroots. By Lemma \ref{Gashi}, $\mu+\g^\vee, \mu+\th^\vee+\g^\vee \preceq \l$ and hence $w\i(\g), w\i(s_\th(\g)) \in \Upsilon^+(\mu_x, \l)$. Therefore, we deduce that $\th^\vee \in A$, a contradiction. So (b) is proved.

\

Now we turn to the ``moreover" part. Since $\ZZ\Phi^\vee$ is spanned by $\Pi_{\text{short}, G_2}^\vee$ and $\Pi_{\text{long}, G_2}^\vee$, it suffices to show

(c) $C_{\l, J, b, x} \cap \Pi_{\text{long}, G_2} \neq \emptyset$.

If $\mu_x \notin Y^+$, (c) follows similarly as in Case(1) above. Assume $\mu_x \in Y^+$. By (a), there exists $\th \in \Pi_{\text{long}, G_2}$ such that $\mu_x+\th^\vee \le \l$. Hence $\mu_x+\th^\vee \preceq \l$ by Lemma \ref{Gashi}. So (c) still holds.
\end{proof}

\begin{cor}\label{mod} Keep the assumptions of Lemma \ref{gen}. We modify the definition of $C_{\l, J, b, x}$ slightly by resetting $$C_{\l, J, b, x}=\{\a \in \Pi_{\text{long}, G_2}; \a \text{ is $J$-antidominant and $\mu_x+\a^\vee$} \}$$ if $\Phi$ is a union of root systems of type $G_2$ and $\Pi_J=\Pi_{\text{short}, G_2}$. Then $\ZZ \Phi^\vee$ is still generated by $\Pi_J^\vee$ and $\cup_{\a \in C_{\l, J, b, x}} \co_\a^\vee$.
\end{cor}

This modified definition will be used to prove Proposition \ref{k2'} below.

\

\begin{prop} \rm{(cf. \cite[Proposition 4.1.10]{CKV})} \label{k2'}
Let $J \subseteq S_0$ with $\s(J)=J$, $b \in G(L)$ and $\l \in Y^+$ such that $(J, b)$ is admissible. For $x \in \bar I_{\l, J, b}$ we have $C_{\l, J, b, x} \subseteq \ca_{\l, J, b}$, where $\ca_{\l, J, b}$ is defined in \S \ref{5.1}.
\end{prop}

The proof will be given in Section \ref{sec k2'}.

\

Now we are ready to prove Proposition \ref{k2}. The statement (b) and its usage in the proof below are suggested by Miaofen Chen, which play a fundamental role in our arguments.
\begin{proof}[Proof of Proposition \ref{k2}]
By \cite[Corollary 2.5.12]{CKV}, the kernel of the natural map $\pi_1(M_J)^\s \to \pi_1(G)^\s$ is spanned by $y_\a=\sum_{\b \in \co_\a} \b^\vee \in \pi_1(M_J)^\s$ for $\a \in \Phi$. Let $J_0 =\{s_\a \in S_0; \a \in \Pi_0, y_\a \in \cl_{\l,J,b}\}$, where $\cl_{\l, J, b}$ is the lattice spanned by $y_\a$ for $\a \in \ca_{\l, J, b}$. It suffices to show $J_0=S_0$. Note that $\s(J_0)=J_0$ and $J \subseteq J_0$. By assumption of the proposition, there exists $\mu \in \bar I_{\l, J, b}$, which is $J$-dominant, $J$-minuscule and weakly dominant.

Since $\s$ acts transitively on the $h$ connected components of $\Phi$, $\{\sharp \co_\g; \g \in \Pi_0\}=\{n_\Phi, m_\Phi\}$ for some $n_\Phi \le m_\Phi \in \NN$ with $m_\Phi / n_\Phi \in \{1,2,3\}$. Let $\a \in \Phi_0$. Then $\sharp \co_\a=h$ if $\sharp \co_\a < m_{\Phi_\a}$. First we claim

(a) if $\sharp \co_\a=m_\Phi$, then $\a \in \Pi_{J_0}$.

By Corollary \ref{mod}, $\a^\vee$ lies in the lattice generated by $\ZZ \Phi_J^\vee$ and $\cup_{\a \in C_{\l, J, b, \mu}} \co_\a^\vee$. Moreover, since $\sharp \co_\a=m_{\Phi}$, we have $y_\a=\sum_{i=1}^r y_{\g_i} \in \pi_1(M_J)^\s$ with each $\g_i \in \pm C_{\l, J, b, \mu}$. Hence $y_\a \in \cl_{\l,J,b}$ by Proposition \ref{k2'}. Therefore, $\a \in \Pi_{J_0}$ and (a) is proved.

If $J_0 =S_0$, there is nothing to prove. Assume $J_0 \neq S_0$. We show this assumption will lead to some contradiction, and this complete our proof. By (a), $\s$ is of order $2h$ or $3h$ and each simple root of $\Pi-\Pi_{J_0}$ is fixed by $\s^h$. In particular, $\Phi$ is simply laced. We claim that

(b) there exist $\a \in \Pi_0-\Pi_{J_0}$, $\vartheta \in \Phi^+$ with $\s^h(\vartheta)=\vartheta$ such that $\vartheta^\vee-\a^\vee \in \ZZ \Phi_{J_0}^\vee$ and either\\
(b1) $\vartheta \in C_{\l, J, b}:=\cup_{x \in \bar I_{\l, J, b}} C_{\l, J, b, x}$; or \\
(b2) $x \overset {(\b, h)} \to x+\b^\vee-\s^h(\b^\vee)$ and $x \overset {(\b+\vartheta, h)} \to x+\b^\vee-\s^h(\b^\vee)$ for some $\b \in \Phi_{J_0}$ and $x \in \bar I_{\l, J, b}$.

Assuming (b), we now can finish the proof of the proposition. If (b1) occurs, we have, by Proposition \ref{k2'}, $y_\vartheta \in \cl_{\l,J,b}$. Since $\vartheta^\vee-\a^\vee \in \ZZ \Phi_{J_0}^\vee$ and $\sharp \co_\a = \sharp \co_\vartheta=h$, we have $y_\a-y_\vartheta \in (\ZZ \Phi_{J_0}^\vee)^\s \subseteq \cl_{\l,J,b}$, which means $y_\a \in \cl_{\l,J,b}$ and hence $\a \in \Pi_{J_0}$, a contradiction. If (b2) occurs, by Proposition \ref{main}, there exist $Q \in X_{\mu_x}^{M_J}(b)$ and $P_1, P_2 \in X_{\mu_{x+\b^\vee-\s^{h}(\b^\vee)}}^{M_J}(b)$ such that $Q \sim_{\l, b} P_1$ (resp. $Q \sim_{\l, b} P_2$) and $\eta_J(Q)-\eta_J(P_1)=\sum_{i=0}^{h-1} \s^i(\b^\vee)$ (resp. $\eta_J(Q)-\eta_J(P_2)=\sum_{i=0}^{h-1} \s^i(\vartheta^\vee+\b^\vee)$). Therefore, $P_1 \sim_{\l, b} P_2$ and $\eta_J(P_1)-\eta_J(P_2)=y_\vartheta$, that is, $\vartheta \in \ca_{\l, J, b}$ and hence $y_\vartheta \in \cl_{\l,J,b}$, which again is a contradiction by the same argument in the case (b1).

It remains to prove (b).

Case (1): $\s$ is of order $3h$. Then each connected component of $\Phi$ is of type $D_4$. We can write $S_0=F \sqcup H$, where $F$ and $H$ are the two distinct $\s$-orbits on $S_0$. Assume $\s^h$ acts trivially on $H$. Then $F = J_0$ by (a) and the assumption $J_0 \neq S_0$. Note that $(\l, b)$ is NH-irreducible. So there exists $\a \in H=\Pi_0-\Pi_{J_0}$ such that

(1-a) $\mu+\a^\vee \le \l$.

We remark that

(1-b) Let $\mu' \in Y$ such that $\mu' \preceq \l$ and $\mu' + \a^\vee \le \l$. Assume $\<\a, \mu'\> \le -1$ or $\mu'$ is $J_0$-dominant, then $\mu'+\a^\vee \preceq \l$.

Indeed, Since $\mu' \preceq \l$, We can assume $\Phi$ has only one connected component. If $\<\a, \mu'\> \le -1$, the statement is obvious. Now we assume $\mu$ is $J_0$-dominant and $\<\a, \mu'\> \ge 0$, that is, $\mu'$ is dominant (since $\Phi$ is of type $D_4$). Then the statement follows from Lemma \ref{Gashi}. So (1-b) is proved.

Case (1.1): $J=F=J_0$. Then $\mu$ is $J_0$-dominant and (b1) follows from (1-b) by taking $\vartheta=\a$.

Case (1.2): $J=\emptyset$. Choose $\b \in \Pi_{J_0}$ such that $\<\b, \a^\vee\>=-1$. We claim that

(1.2-a) there exists $\mu_0 \in \bar I_{\l, J, b}$ such that $\mu_0+\a^\vee \le \l$ and $\<\b+\s^h(\b)+\s^{2h}(\b), \mu_0\> \ge 0$.

Let $E=\{\mu' \in \bar I_{\l, J, b}; \mu' \text{ is weakly dominant and } \mu'+\a^\vee \le \l \}$. Then $\mu \in E \neq \emptyset$. Let $c_0= \max \{\<\b+\s^h(\b)+\s^{2h}(\b), \mu'\>; \mu' \in E\}$. We have to show $c_0 \ge 0$. Assume otherwise. Let $\mu_0 \in E$ such that $\<\b+\s^h(\b)+\s^{2h}(\b), \mu_0\>=c_0 \le -1$. By Lemma \ref{dominant}, $\nu_{t^{\mu_0}}=\frac{1}{3h}\sum_{i=0}^{3h-1} \s^i(\mu_0)$ is dominant. In particular, $\sum_{\g \in \co_\b} \<\g, \mu_0\> \ge 0$. So there exist $\g' \in \{\b, \s^h(\b), \s^{2h}(\b)\}$ and $\g'' \in \co_\b-\{\b, \s^h(\b), \s^{2h}(\b)\}$ such that $\<\g', \mu_0\> \le -1$ and $\<\g'', \mu_0\> \ge 1$. Let $\mu_0'=\mu_0+{\g'}^\vee-{\g''}^\vee \in \bar I_{\l, J, b}$. Then $\mu_0' \preceq \l$ is weakly dominant by Lemma \ref{minus}. Moreover, $\mu_0'+\a^\vee \le \l$ since $\mu_0+\a^\vee, \mu_0' \le \l$ and $\a$ is a simple root. Thus, $\mu_0' \in E$. However, $\<\b+\s^h(\b)+\s^{2h}(\b), \mu_0'\>=c_0+2$, a contradiction. So (1.2-a) is proved.

Thanks to (1-b), we can assume, by (1.2-a) and the symmetry among $\b$, $\s^h(\b)$ and $\s^{2h}(\b)$, that $\<\mu_0, \a\> \ge 0$, $\<\b, \mu_0\>=-1$ and one of the following cases occurs:

(1.2.1): $\<\s^{2h}(\b), \mu_0\>=-1$ and $\<\s^h(\b), \mu_0\> \ge 2$;

(1.2.2): $\<\s^h(\b), \mu_0\>, \<\s^{2h}(\b), \mu_0\> \ge 0$.

Assume (1.2.1) occurs. Then $\mu_0':=\mu_0+\s^{2h}(\b^\vee)-\s^h(\b^\vee) \preceq \l$. Note that $\mu_0'+\b^\vee \preceq \l$ is $J_0$-dominant. So $\mu_0'+\b^\vee+\a^\vee \le \l$ since $\mu_0+\a^\vee \le \l$. Therefore, $\mu_0'+\b^\vee+\a^\vee \preceq \l$ by (1-b). Moreover, we have $\mu_0'-\s^{2h}(\b), \mu_0'-\s^{2h}(\b)-\a^\vee \preceq \l$ since $\<\s^{2h}(\b), \mu_0'\>, \<\s^{2h}(\b)+\a, \mu_0'\> \ge 1$. Thus we deduce that $\mu_0'+\b^\vee-\s^{2h}(\b^\vee) \overset {(\s^{2h}(\b), h)} \to \mu_0'$ and $\mu_0'+\b^\vee-\s^{2h}(\b^\vee) \overset {(\a+\s^{2h}(\b), h)} \to \mu_0'$.

Assume (1.2.2) occurs. Again we have $\mu_0+\b^\vee+\a^\vee \preceq \l$ by (1-b). If $\<\s^{2h}(\b), \mu_0\> \ge 1$, then $\mu_0+\b^\vee-\s^{2h}(\b^\vee) \overset {(\s^{2h}(\b), h)} \to \mu_0$ and $\mu_0+\b^\vee-\s^{2h}(\b^\vee) \overset {(\a+\s^{2h}(\b), h)} \to \mu_0$. If $\<\s^{h}(\b), \mu_0\> \ge 1$, then $\mu_0 \overset {(\b, h)} \to \mu_0+\b^\vee-\s^h(\b^\vee)$ and $\mu_0 \overset {(\a+\b, h)} \to \mu_0+\b^\vee-\s^h(\b^\vee)$. So (b) holds when $\s$ is of order $3h$.

\

Case (2): $\s$ is of order $2h$. As in the proof of Lemma \ref{gen}, there exists $\a \in \Pi-\Pi_{J_0}$ such that $\bar \mu^{J_0}+\a^\vee \preceq \l$ (by using Lemma \ref{Gashi} and that $\Phi$ is simply laced when $\bar \mu^{J_0}$ is dominant). In particular, $\mu+\a^\vee \le \l$. We claim that

(2-a) There exists an orthogonal subset $D \subseteq \Phi_{J_0}^+$ satisfying\\
(2-a-1) $\<\g, \mu\>=\<\g, \a^\vee\>=-1$ for any $\g \in D$; \\
(2-a-2) $\mu+\a^\vee+\sum_{\g \in D} \g^\vee \preceq \l$.

Indeed, since $\mu$ is weakly dominant, there exists, by Lemma \ref{weak}, an orthogonal subset $D_{\mu, J_0} \subseteq \Phi_{J_0}^+$ such that $\bar \mu^{J_0}-\mu=\sum_{\b \in D_{\mu, J_0}} \b^\vee$ and $\<\b, \mu\>=-1$ for any $\b \in D_{\mu, J_0}$. Let $$D=\{\b \in D_{\mu, J_0}; \<\b, \a^\vee\> \le -1\}=\{\b \in D_{\mu, J_0}; \<\b, \a^\vee\> = -1\}.$$ Note that $D_{\mu, J_0}-D=\{\b \in D_{\mu, J_0}; \<\b, \a^\vee\> \ge 0\}$. So $$\mu+\a^\vee+\sum_{\g \in D} \g^\vee \preceq \mu+\a^\vee+\sum_{\g \in D} \g^\vee + \sum_{\b \in D_{\mu, J_0}-D} \b^\vee=\bar \mu^{J_0}+\a^\vee \preceq \l.$$ Thus (2-a) is proved.

Now (b) follows, in the case that $\s$ is of order $2h$, from (2-a) by Lemma \ref{lattice} below.

In a word, the proof of (b) is finished.
\end{proof}

\begin{lem} \label{lattice}
Let $\l \in Y^+$ and $J \subseteq J_0 \subseteq S_0$ with $\s(J_0)=J_0$. Assume $\Phi$ is simply laced and $\s^{2h}=\Id$ and there exist $\mu \in \bar I_{\l, J, b}$ which is weakly dominant, $\th \in \Phi^+ - \Phi_{J_0}$ with $\s^h(\th)=\th$ and an orthogonal subset $D \subseteq \Phi_{J_0}^+$ such that

(1) $\<\g, \mu\>=\<\g, \th^\vee\>=-1$ for any $\g \in D$;

(2) $\mu+\th^\vee+\sum_{\g \in D} \g^\vee \preceq \l$.

Then there exists $\vartheta \in \Phi^+$ such that $\s^h(\vartheta)=\vartheta$, $\vartheta^\vee-\theta^\vee \in \ZZ \Phi_{J_0}^\vee$ and one of the following holds:

(i) either $\vartheta \in C_{\l, J, b}:=\cup_{x \in \bar I_{\l, J, b}} C_{\l, J, b, x}$ or

(ii) $x \overset {(\a, h)} \to x+\a^\vee-\s^h(\a^\vee)$ and $x \overset {(\a+\vartheta, h)} \to x+\a^\vee-\s^h(\a^\vee)$ for some $\a \in \Phi_{J_0}$ and some $x \in \bar I_{\l, J, b}$.
\end{lem}
\begin{proof}
We employ the notation used in Lemma \ref{minu}.

Case(0): $\s^h=\Id$. Let $\b=(\prod_{\g \in D} s_\g)(\th)$ and take $\vartheta=\b_J$, where $\b_J$ is defined as in Lemma \ref{minu}. Then $\vartheta \in C_{\l,J,b,\mu}$ as desired.

Case(1): $\s^h$ is of order $2$.

Case(1.0): $D=\emptyset$. Then $\mu+\th^\vee \preceq \l$. Let $\vartheta=\th_J$, then $\s^h(\vartheta)=\vartheta$ since $\s^h(\th)=\th$. Then the lemma follows similarly as in Case(0).

Case(1.1): $\sharp D=1$. Write $D=\{\g\}$ for some $\g \in \Phi_{J_0}^+$. If $\g=\s^h(\g)$, then we return to Case(1.0) by replacing $\th$ with $\th+\g$. Now assume $\g \neq \s^h(\g)$. If $\<\s^h(\g), \mu+\g^\vee\> \le 0$, then $\<\s^h(\g), \mu+\th^\vee+\g^\vee\> \le \<\s^h(\g), \th^\vee\>=\<\g, \th^\vee\>=-1$. Thus $\th+\g+\s^h(\g) \in \Phi^+$ and $\mu+\th+\g+\s^h(\g) \preceq \l$. So we return to Case(1.0) by replacing $\th$ with $\th+\g+\s^h(\g)$. Assume $\<\s^h(\g), \mu+\g^\vee\> \ge 1$. In particular, $\<\s^h(\g), \mu\> \ge 1$. We have $\mu-\s^h(\g^\vee) \preceq \mu \preceq \l$ and $\mu+\g^\vee-\s^h(\g^\vee) \preceq \mu+\g^\vee \preceq \mu \preceq \l$. Moreover, $$\<\s^h(\th+\g), \mu\>=\<\th+\s^h(\g), \mu\> \ge \<\th+\g, \mu\>+2 \ge -1+2=1,$$ where the last inequality follows from the assumption that $\mu$ is weakly dominant. So $\mu-\s^h(\th^\vee+\g^\vee) \preceq \mu \preceq \l$. Therefore, $\mu \overset {(\th+\g, h)} \to \mu+\g-\s^h(\g)$ and $\mu \overset {(\g, h)} \to \mu+\g-\s^h(\g)$. So (ii) is satisfied.

Case(1.2): $\sharp D=2$. Write $D=\{\g_1, \g_2\}$. We have $\s^h(\g_i) \neq \g_j$ since otherwise, we would return to Case(1.0) or Case(1.1). As a consequence, $\e_{i,j}:=\<\s^h(\g_i), \g_j^\vee\>=\<\s^h(\g_j), \g_i^\vee\> \in \{-1, 0, 1\}$ for $i, j \in \{1,2\}$, and moreover, $\e_{i, i}=0$ since $-1 \le \<\s^h(\g_i), \th^\vee+\g_i^\vee\>=-1+\e_{i, i} \le -1$. If $\e_{1,2} \le 0$, then $\<\s^h(\th+\g_1+\g_2), \th^\vee+\g_1^\vee+\g_2^\vee\> \le 0+2(\e_{1,2}-1) \le -2$, which is impossible. So $\e_{1,2}=1$. By symmetry, we may assume $\d:=\g_1-\s^h(\g_2) \in \Phi^+$. Then $\th+\g_1+\g_2=\th+\g_2+\s^h(\g_2)+\d$ and $$\<\d, \mu+\th^\vee+\g_2^\vee+\s^h(\g_2^\vee)+\d^\vee\>=\<\d, \mu\>+1.$$ If $\<\d, \mu\> \ge 0$, then $\<\d, \mu+\th^\vee+\g_2^\vee+\s^h(\g_2^\vee)+\d^\vee\> \ge 1$ and hence $$\mu+\th^\vee+\g_2^\vee+\s^h(\g_2^\vee) \preceq \mu+\th^\vee+\g_2^\vee+\s^h(\g_2^\vee)+\d^\vee = \mu+\th^\vee+\g_1^\vee+\g_2^\vee \preceq \l.$$ So we return to Case(1.0). Assume $\<\d, \mu\>=-1$ (since $\mu$ is weakly dominant). We have $\<\d, \th^\vee+\g_2^\vee+\s^h(\g_2^\vee)\>=-1$ and hence return to Case(1.1).

Case(1.3): $\sharp D=3$. Write $D=\{\g_1, \g_2, \g_3\}$. As in Case(1.2), we can assume $\s^h(\g_i) \neq \g_j$ for $i, j \in \{1,2,3\}$. Set $\e_{i,j}=\<\s^h(\g_i), \g_j^\vee\>=\<\s^h(\g_j), \g_i^\vee\> \in \{-1, 0, 1\}$ for $i, j \in \{1, 2, 3\}$. Then $\e_{i, i}=0$ for $i \in \{1, 2, 3\}$ as above. One computes that $$\<\th+\s^h(\g_i), \th^\vee+\g_1^\vee+\g_2^\vee+\g_3^\vee\>=-2+\<\s^h(\g_i), \g_1^\vee+\g_2^\vee+\g_3^\vee\> \in \{-1, 0, 1\},$$ which implies $\e_{i, j} \ge 0$ if $i \neq j$.

If $\e_{i,j}=0$ for any $i \neq j \in \{1,2,3\}$, then $\<\th+\g_1+\g_2+\g_3, \th^\vee+\g_1^\vee+\g_2^\vee+\g_3^\vee\> \le -4$, which is impossible. By symmetry, we can assume $\e_{1,2}=1$ and $\d:=\g_1-\s^h(\g_2) \in \Phi^+$. Then $\th+\g_1+\g_2+\g_3=\th+\g_2+\s^h(\g_2)+\g_3+\d$. Since $\mu$ is weakly dominant, we have

(b) $\<\d, \mu\>=-1-\<\s^h(\g_2), \mu\> \le 0$ and $\<\d, \g_3^\vee\>=-\e_{2,3} \le 0$.

Case(1.3.1): $\<\d, \g_3^\vee\>=-\e_{2,3}=0$. One computes that $$\<\d, \mu+\th^\vee+\g_2^\vee+\s^h(\g_2^\vee)+\g_3^\vee+\d^\vee\>=\<\d, \mu\>+1.$$ We apply the arguments in Case(1.2). If $\<\d, \mu\> \ge 0$, then $\mu+\th^\vee+\g_2^\vee+\s^h(\g_2^\vee)+\g_3^\vee \preceq \l$ and we return to Case(1.1). Otherwise,$$\<\g_3, \th^\vee+\g_2+\s^h(\g_2)\>=\<\d, \th^\vee+\g_2+\s^h(\g_2)\>=\<\d, \mu\>=-1$$ and we return to Case(1.2).

Case(1.3.2): $\<\d, \g_3^\vee\>=-1$. Then $\e_{2,3}=1$. Since $\mu$ is weakly dominant, $\<\d+\g_3, \mu\> \ge -1$, which implies $\<\d, \mu\>=0$ by (b). One computes that $\<\g_3+\d, \th^\vee+\g_2^\vee+\s^h(\g_2^\vee)\>=\<\g_3+\d, \mu\>=-1$ and we return to Case(1.1).

Case(1.4): $\sharp D \ge 4$. This case dose not occur since $\<\th+\sum_{\g \in D} \g, \th^\vee\> \le -2$, contradicting the fact that $\Phi$ is simply laced.
\end{proof}

\section{Proof of Proposition \ref{k2'}} \label{sec k2'}
The aim of this section is to prove Proposition \ref{k2'}, which completes the proof of Proposition \ref{k2}. We follow Chen-Kisin-Viehmann \cite[\S 4.7]{CKV} closely (but not exactly) and use their ideas and constructions freely. We will see that most of the results therein can be generalized in our setting. However, the arguments become more subtle. Throughout this section, we assume $G$ is adjoint and simple; the root system $\Phi$ of $G$ has $h$ connected components, on which $\s$ acts transitively. We also fix $\l \in Y^+$, $b \in G(L)$ and a $\s$-stable subset $J \subseteq S_0$ such that $(J, b)$ is admissible and $b$ is superbasic in $M_J(L)$.

\

For $\a \in \Phi-\Phi_J$, we set $\Phi_{J, \a}=\Phi \cap \ZZ(\Phi_J \cup \co_\a)$. Thanks to \cite[Proposition 4.2.11]{CKV}, the set of simple root of $\Phi_{J, \a}$ is $\Pi_J \cup \co_\a$ if $\Phi$ is simply laced.

Following \cite[\S 4.7]{CKV}, we say $\co_\a$ is of type I, II, or III (with respect to $J$) if $\sharp \co_\a$ is equal to $d$, $2d$ or $3d$ respectively. Here $d$ is the minimal positive integer such that $\a, \s^d(\a)$ are in the same connected component of $\Phi_{J, \a}$. We define \begin{align*} \tta=\begin{cases} \a, & \text{ if $\co_\a$ is of type I; } \\ \a + \b +\s^d(\a), &\text{ if $\co_\a$ is of type II;} \\ \a+\s^d(\a)+\s^{2d}(\a)+\b, &\text{ if $\co_\a$ is of type III, } \end{cases}\end{align*} where $\b$ is the unique common neighbor (might be zero in type II case) of $\co_\a$ in he Dynkin diagram of $\Phi_{J, \a}$. We emphasize that $\tilde \a$ is root. Moreover, $\tilde \a$ is $J$-anti-dominant and $J$-minuscule if so is $\a$.

\

Now we fix $x \in \bar I_{\l, J, b}$ and $\a \in C_{\l,J,b,x}$. Set $n=\sharp \co_\a$. To prove Proposition \ref{k2'}, we show $\a \in \ca_{\l, J, b}$ separately according to the type of $\co_\a$.

\begin{lem} \label{z1}
Let $x \in \bar I_{\l, J, b}$ and $\a \in C_{\l, J, b, x}$. If there do not exist $x' \in \bar I_{\l,J,b}$, $\th \in \co_\a$ and $j \in [1, n-1]$ such that $x \overset{(\th, j)} \to x'$. Then $w_x(\s^k(\a))=\s^k(\a)$ and $\<\s^k(\a), \mu_x\>=0$ for $k \in [1, n]-\NN d$.
\end{lem}
\begin{proof}
Let $k \in [1, n]-\NN d$. Then $\a, \s^k(\a)$ are in different connected components of $\Phi_{J,\a}$. First we claim that

(a) $\<w_x(\s^k(\a)), \mu_x\> \le 0$ for $k \in [1, n]-\NN d$.

Indeed, otherwise, $\<w_x(\s^k(\a)), \mu_x+\a^\vee\>=\<w_x(\s^k(\a)), \mu_x\> \ge 1$, which means $x \overset{(\a, k)} \to x+\a^\vee-\s^k(\a^\vee)$, contradicting the assumption of the lemma. So (a) is proved.

Assume $\<\a, \mu_x\> \ge 0$. Then $\mu_x-\a^\vee \preceq \l$ since $\mu_x, \mu_x+\a^\vee \preceq \l$. Thus, we have $\<\s^k(\a), \mu_x\> \ge 0$. Indeed, otherwise we would have $x  \overset{(\s^k(\a), n-k)} \to x+\s^k(\a^\vee)-\a^\vee$, a contradiction. By (a) and Lemma \ref{central'}, we have $\<w_x(\s^k(\a)), \mu_x\>=\<\s^k(\a), \mu_x\>=0$ and $w_x(\s^k(\a))=\s^k(\a)$ as desired.

Now assume $\<\a, \mu_x\> \le -1$. We adopt the following argument in \cite[Lemma 4.7.10]{CKV}. Since $\nu_{b_x} \in Y_\QQ$ is dominant (see Lemma \ref{dominant}), for any positive $m \in \NN d$ with $(w_x \s)^m=1$, we have $$0 \le m \<\a, \nu_{b_x}\> =\sum_{i=0}^{m-1} \<(w_x \s)^{-i}(\a), \mu_x\>=\sum_{i=0}^{m-1} \<w_i(\s^i(\a)), \mu_x\>$$ for some $w_i \in W_J$ with $w_0=1$. However, by (a), we have $\<w_i(\s^i(\a))\> \le \<w_x(\s^i(\a)), \mu_x\>=0$ for $i \in [0, m-1]-\NN d$. Using Lemma \ref{central'} (a) again, we obtain $$0 \le \sum_{i \in [0, m-1] \cap \NN d} \<w_i(\a), \mu_x\> \le \<\a, \mu_x\>+(\frac{m}{d}-1)\<w_x(\a), \mu_x\> \le \frac{m}{d} \<w_x(\a), \mu_x\>,$$ which implies

(b) $\<w_x(\s^i(\a)), \mu_x\> \ge 1$ and hence $\mu_x-w_x(\s^i(\a^\vee)) \preceq \l$.

By (b), we have $x \overset{(\s^k(\a), i-k)} \to x+\s^k(\a^\vee)-\s^i(\a^\vee)$ if $k < i$ and $x \overset{(\s^k(\a), n+i-k)} \to x+\s^k(\a^\vee)-\s^i(\a^\vee)$ if $i < k$. Now the lemma follows similarly as in the above paragraph.
\end{proof}

\begin{proof}[Proof of Proposition \ref{k2'} when $\co_\a$ is of type I] \

Case(1): $x \overset {(\th, r)} \to x'$ for some $x' \in \bar I_{\l, J, b}$, $\th \in \co_\a$ and $r \in [1, d-1]$. We claim that

(1-a) there exist $Q \in X_{\mu_x}^{M_J}(b)$ and $Q' \in X_{\mu_{x'}}^{M_J}(b)$ such that $P \sim_{\l, b} Q$ and $\eta_J(P)-\eta_J(Q)=\sum_{i=0}^{r-1} \s^i(\th^\vee)$.

First we show the proposition follows from (1-a). Indeed, by symmetry, we have $x' \overset {(\s^r(\a), d-r)} \to x$. Thus by (1-a) and that $J_b^{M_J}(F)$ acts on $\pi_0(X_{\mu_{x'}}^{M_J}(b))$ transitively, there exist $P' \in X_{\mu_{x'}}^{M_J}(b)$ and $P \in X_{\mu_x}^{M_J}(b)$ such that $Q' \sim_{J, \l, b} P' \sim_{\l, b} P$ and $\eta_J(P')-\eta_J(P)=\sum_{i=0}^{d-r-1} \s^{i+r}(\th^\vee)$. Therefore, $Q \sim_{\l,b} P$ and $\eta_J(Q)-\eta_J(P)=y_\th=y_\a \in \pi_1(M_J)^\s$. So $\a \in \ca_{\l,J,b}$ as desired.

Now we prove (1-a) by induction on $r$. If $r=1$, then $x \overset {(\th, r)} \rightarrowtail x'$ and the lemma follows from Proposition \ref{main}. Assume the lemma holds for $r < r_0 \in [1, d]$. We show it holds for $r=r_0$.

Case(1.1): there exist $u \in W_J$ and $j \in [1, r-1]$ such that $\<u(\s^j(\th)), \mu_x\> \neq 0$. Without loss of generality, we may assume that $\<u(\s^j(\th)), \mu_x\> \le -1$, the other case can be handled similarly. Let $x''=x+\s^j(\th^\vee)-\s^r(\th^\vee) \in \pi_1(M_J)$. Since $\co_\a$ is of type I, we have $\<\s^j(\th), w_x(\s^r(\th^\vee))\>=\<\s^j(\th), \th^\vee\>=0$ and hence \begin{align*} \<\s^j(\th), \mu_x+\th^\vee-w_x(\s^r(\th^\vee))\>  =\<\s^j(\th), \mu_x\> \le \<u(\s^j(\th)), \mu_x\> \le -1, \end{align*} which means $x \overset {(-\s^j(\th), r-j)} \to x'' \overset {(-\th, j)} \to x'$. By induction hypothesis, there exist $P'' \in X_{\mu_{x''}}^{M_J}$ and $P' \in X_{\mu_{x'}}^{M_J}$ such that $P \sim_{\l, b} P'' \sim_{\l, b} P'$ and $\eta_J(P)-\eta_J(P'')=\sum_{i=0}^{r-j-1} \s^{i+j}(\th^\vee)$ and $\eta_J(P'')-\eta_J(P')=\sum_{i=0}^{j-1} \s^i(\th^\vee)$. Therefore, we deduce that $P \sim_{\l, b} P'$ and $\eta_J(P)-\eta_J(P')=\sum_{i=0}^{r-1} \s^i(\th^\vee)$ as desired.

Case(1.2): $\<u(\s^i(\th)), \mu_x\> = 0$ for any $u \in W_J$ and any $i \in [1, r-1]$. By Lemma \ref{central'}, $w_x(\s^i(\th))=\s^i(\th)$ and $\<\s^i(\th), \mu_x\>=0$ for any $i \in [1, r-1]$. Define $\textsl{g}: \PP^1 \to G(L)/K$ by $$\textsl{g}(z)=g_x \ U_{\th}(z t\i) \ {}^{b_x\s} U_{\th}(z t\i) \cdots {}^{(b_x\s)^{r-1}} U_{\th}(z t\i) K,$$ where $g_x \in M_J(L)$ such that $g_x\i b \s(g_x)=\dot b_x=t^{\mu_x} \dot w_x$ and $b_x=t^{\mu_x} w_x \in \Omega_J$ is defined in \S \ref{setup4}. Note that $\th, \s^i(\th)$ are not in the same connected component of $\Phi_{J, \a}$ for $i \in [1, r-1]$ since $\co_\a$ is of type I. One checks that $\eta_J(\textsl{g}(0))-\eta_J(\textsl{g}(\infty))=\sum_{k=0}^{r-1} \s^k(\th^\vee)$ and $$\textsl{g}(z)\i b \s(\textsl{g}(z))=K U_\th(z t\i) t^{\mu_x} U_{w_x(\s^r(\th))}(c \s^r(z) t\i) \dot w_x K,$$ where $c \in \bold{k}[[t]]^\times$ is some constant. In view of Lemma \ref{simple}, to show (1-a), it remains to verify $\textsl{g}\i b \s(\textsl{g}) \subseteq \cup_{\l' \preceq \l} K t^{\l'} K$. By Lemma \ref{bound}, it suffices to show

(1.2-a) $\Phi \cap (\ZZ \th + \ZZ w_x(\s^r(\th)))$ is of type $A_1$ or $A_1 \times A_1$ or $A_2$.

If $\Phi$ is simply laced, (1.2-a) follows since all the roots of $\Phi$ are of the same length. Otherwise, (1.2-a) follows since $\th, w_x(\s^r(\th)))$ are in different connected components of $\Phi$.

Therefore, (1-a) is proved.

\

Case(2): There do not exist $x' \in \bar I_{\l,J,b}$, $\th \in \co_\a$ and $r \in [1, d-1]$ satisfying $x \overset{(\th, r)} \to x'$. By Lemma \ref{z1}, we have

(2-a) $w_x(\s^k(\a))=\s^k(\a)$ and $\<w_x(\s^k(\a)), \mu_x\>=0$ for $k \in [1, d-1]$.

We define $\textsl{g}: \PP^1 \to G(L)/K$ by $$\textsl{g}(z)=g_x \ U_{\a}(z t\i) \ {}^{b_x\s} U_{\a}(z t\i) \cdots {}^{(b_x\s)^{d-1}} U_{\a}(z t\i) K,$$ where $g_x\i b \s(g_x)=\dot b_x=t^{\mu_x} \dot w_x$ as usual. Note that $\a, \s^i(\a)$ are not in the same connected component of $\Phi_{J, \a}$ for $i \in [1, r-1]$ since $\co_\a$ is of type I. As in Case(1.2), we have $\eta_J(\textsl{g}(0))-\eta_J(\textsl{g}(\infty))=\sum_{k=0}^{r-1} \s^k(\a^\vee)=y_\a \in \pi_1(M_J)$ and $$\textsl{g}(z)\i b \s(\textsl{g}(z))=K U_\a(z t\i) t^{\mu_x} U_{w_x(\a)}(c \s^d(z) t\i) \dot w_x K,$$ where $c \in \bold{k}[[t]]^\times$ is some constant. Again, it remains to verify $\textsl{g}\i b \s(\textsl{g}) \subseteq \cup_{\l' \preceq \l} K t^{\l'} K \subseteq K \backslash G(L) / K$. So it suffices to show

(2-b) $U_\a(y t\i) t^{\mu_x} U_{w-x(\a)}(z t\i) \in \cup_{\l' \preceq \l} K t^{\l'} K$ for $y, z \in \bold{k}[[t]]$.

Similar to (a) in the proof of Lemma \ref{z1} (using (2-a)), we have

(2-c) $\<w_x(\a), \mu_x\> \ge 0$. Moreover, $\<w_x(\a), \mu_x\> \ge 1$ if $\<\a, \mu_x\> \le -1$.

Case(2.1): $\Phi \cap (\ZZ \a + \ZZ w_x(\a))$ is of type $A_2$ or $A_1 \times A_1$ or $A_1$. Thanks to Lemma \ref{bound}, to prove (2-b), it suffices to show that $\mu_x+\a^\vee, \mu_x-w_x(\a^\vee), \mu_x+\a^\vee-w_x(\a^\vee) \preceq \l$. Indeed, we already know $\mu_x+\a^\vee \preceq \l$ (since $\a \in C_{\l, J, b, x}$) and $\mu_x+\a^\vee-w_x(\a^\vee) \preceq \mu_x \preceq \l$ (by Lemma \ref{minuscule}). If $\<w_x(\a), \mu_x\> \ge 1$, we have $\mu_x - w_x(\a^\vee) \preceq \l$ as desired. Otherwise, $\<w_x(\a), \mu_x\>=\<\a, \mu_x\>=0$ by (2-c). Thus $w_x(\a)=\a$ by Lemma \ref{central'} and $\mu_x-w_x(\a^\vee) = \mu_x-\a^\vee = s_\a(\mu_x+\a^\vee) \preceq \l$ as desired.

Case(2.2): $\Phi \cap (\ZZ \a + \ZZ w_x(\a))$ is of type $B_2$. We show this case dose not occur. First note that $\a \neq \pm w_x(\a)$ and $\<w_x(\a), \a^\vee\>=0$ since $\a, w_x(\a)$ are of the same length. By the assumption of Case(2.2), we have $\a - w_x(\a) \in \Phi_J$. However, $\<w_x(\a)-\a, \a^\vee\>=-2$, which contradicts that $\a^\vee$ is $J$-minuscule.

Case(2.3): $\Phi \cap (\ZZ \a + \ZZ w_x(\a))$ is of type $G_2$. We show this case also dose not occur. First note that $\Phi \cap (\ZZ \a + \ZZ w_x(\a))$ is a connected component of $\Phi$. So $\Phi_J \cap (\ZZ \a + \ZZ w_x(\a))=\{\pm \xi\}$ for some $\xi \in \Phi_J^+$ such that $\<\xi, \a^\vee\> \neq 0$ since $w_x(\a) \neq \a$. Thus $\<\xi, \a^\vee\>=-1$ since $\a^\vee$ is $J$-anti-dominant and $J$-minuscule. If $\xi$ is a long root of $\Phi$, so is $\a$ since $\<\xi, \a^\vee\>=-1$. Thus $\Phi \cap (\ZZ \a \oplus \ZZ w_x(\a))$ is of type $A_2$, a contradiction. If $\xi$ is a short root of $\Phi$, then $\a$ is again a long root by the modified construction of $C_{\l,J,b,x}$ (since $\Phi$ is a union of root systems of type $G_2$, see Corollary \ref{mod}), leading the same contradiction.
\end{proof}

\

\begin{lem} \rm{(cf. \cite[Lemma 4.7.3]{CKV})} \label{z3}
Assume $\co_\a$ is of type II and $\mu_y+\s^k(\tta^\vee) \npreceq \l$ for any $y \in \bar I_{\l, J, b}$ and any $k \in \NN$. If there exist $\th \in \co_\a$, $r \in [d+1, 2d-1]$ and $x' \in \bar I_{\l, J, b}$ such that $x \overset {(\th, r)} \rightarrowtail x'$, then we have

(a) $w_x(\s^{r-d}(\th))=\s^{r-d}(\th)$, $\<\s^{r-d}(\th), \mu_x\>=\<\s^r(\g), \mu_x\>=0$ and $w_x(\s^r(\g))=\s^r(\g)$;

(b) $w_x(\s^d(\th))=\s^d(\th)+\g$ and $\<w_x(\s^d(\th)), \mu_x\>=1$;

(c) $\g \neq 0$ if and only if $\<\g, \mu_x\>=1$;

(d) $w_x(\s^i(\th))=\s^i(\th)$, $\<\s^i(\g), \mu_x\>=0$ and $\<\s^i(\th), \mu_x\>=0$ for $i \in [1, r-1]-\{r-d, d\}$;

(e) $\<w_x(\s^{r-d}(w_x(\s^d(\th)))), \mu_x\>=\<w_x(\s^r(\th))+\s^r(\g), \mu_x\> \ge 1$.

Here $\g$ is the unique common neighbor of $\th$ and $\s^d(\th)$ in the Dynkin diagram of $\Pi_J \cup \co_\a$.
\end{lem}
\begin{proof}
(a) Since $\mu_y+\s^k(\tta^\vee) \npreceq \l$ for any $y \in \bar I_{\l, J, b}$ and any $k \in \NN$, we have

(i) $\<\s^k(\tta), \mu_y\> \ge 0$ for any $y \in \bar I_{\l, J, b}$ and any $k \in \NN$.

We claim that

(ii) $\<w_x(\s^{r-d}(\th)), \mu_x\> \le 0$.

If $\<w_x(\s^{r-d}(\th)), \mu_x\> \ge 1$, then $\<w_x(\s^{r-d}(\th)), \mu_x+\th^\vee-w_x(\s^r(\th^\vee))\> \ge 1$ and $$x \overset {(\th, r-d)} \to {x+\th^\vee-\s^{r-d}(\th^\vee)} \overset {(\s^{r-d}(\th), d)} \to x',$$ contradicting the assumption of the lemma. So (ii) is proved.

Case(a1): $\<\s^{r-d}(\th), \mu_x\> \le -1$.

Then $\mu_x+\s^{r-d}(\th^\vee), \mu_x+\th^\vee+\s^{r-d}(\th^\vee) \preceq \l$. Moreover, $\<\s^r(\tilde \th), \mu_x\> \ge 0$ by (i), where $\tilde \th=\th+\g+\s^d(\th)$. So we obtain

(a1.0): $\<\s^r(\th), \mu_x\> \ge 0$.

Firstly, we show

(a1.1): $\<w_x(\s^r(\th)), \mu_x+\s^{r-d}(\th^\vee))\> \le 0$. Hence $\<w_x(\s^r(\th)), \mu_x\> \le 1$.

Indeed, otherwise, we have $$\mu_x+\s^{r-d}(\th^\vee)-w_x(\s^r(\th^\vee)) \preceq \mu_x+\s^{r-d}(\th^\vee) \preceq \l,$$ and similarly, $\mu_x+\th^\vee+\s^{r-d}(\th^\vee)-w_x(\s^r(\th^\vee))\preceq \l$, which implies $$\mu_x \overset {(\s^{r-d}(\th), d)} \to {\mu_x+\s^{r-d}(\th^\vee)-\s^r(\th^\vee)} \overset {(\th, r-d)} \to x',$$ contradicting the assumption of the lemma. Thus (a1.1) is proved.

Secondly, we show

(a1.2): $\<w_x(\s^r(\th)), \mu_x\> = 1$.

Assume otherwise, then $0 \le \<\s^r(\th), \mu_x\> \le \<w_x(\s^r(\th)), \mu_x\> \le 0$ by (a1.0) and (a1.1), that is, $\<\s^r(\th), \mu_x\> =\<w_x(\s^r(\th)), \mu_x\>=0$. Applying Lemma \ref{central'} (d) and (c), $w_x(\s^r(\th))=\s^r(\th)$ and $\<\s^r(\g), \mu_x\>=0$ since $\s^r(\g) \in \Pi_J$ and $\<\s^r(\g), \s^r(\th^\vee)\>=-1$. Thus $\<\s^r(\tilde \th), \mu_x\> \le -1$, which contradicts (i). So (a1.2) is proved.

Thirdly, we show

(a1.3): $\<w_x(\s^r(\th)), \mu_x\>=\<\s^r(\th)+\s^r(\g), \mu_x\>=1$ and $w_x(\s^r(\th)) \in \{\s^r(\th), \s^r(\th)+\s^r(\g)\}$.

If $\<\s^r(\th), \mu_x\>=0$, then $\<\s^r(\g), \mu_x\> =1$ since $\<\s^r(\tilde \th), \mu_x\> \ge 0$ and $\mu_x$ is $J$-minuscule. Otherwise, we have $\<\s^r(\th), \mu_x\>=1=\<w_x(\s^r(\th)), \mu_x\>$ by (a1.0) and (a1.2), then $\<\s^r(\g), \mu_x\> =0$ by Lemma \ref{central'} (c). Moreover, If $w_x(\s^r(\th)) \neq \s^r(\th)+\s^r(\g)$, then by Lemma \ref{central'} (b), $\s^r(\th) \le w_x(\s^r(\th)) < s_{\s^r(\g)}(\s^r(\th))=\s^r(\th)+\s^r(\g)$. Hence $w_x(\s^r(\th))=\s^r(\th)$ since $\s^r(\g)$ is a simple root. So (a1.3) is proved.

Now we compute by (a1.3) that $$ \<\s^r(\tilde \th), \mu_{x'}\> = \<\s^r(\tilde \th), \mu_x-w_x(\s^r(\th^\vee))\> \le -1,$$ which is a contradiction. Therefore, Case(a1) dose not occur.

So, by (ii), we have $\<\s^{r-d}(\th), \mu_x\> = \<w_x(\s^{r-d}(\th)), \mu_x\> = 0$. Now (a) follows from Lemma \ref{central'} by noticing that $\<\s^{r-d}(\th), \s^r(\g)\>=-1$ if $\g \neq 0$.

(b) If $\<w_x(\s^d(\th)), \mu_x\> \ge 2$, then $x \overset {(\th, d)} \to {x+\th^\vee-\s^d(\th^\vee)} \overset {(\s^d(\th), r-d)} \to x'$, contradicting the assumption of the lemma. So we have

(b1) $\<w_x(\s^d(\th)), \mu_x\> \le 1$.

We claim that

(b2) $\<\s^d(\th)+\g, \mu_x+\th^\vee\> \ge 0$.

Otherwise, we have $\mu_x+{\tilde \th}^\vee \preceq \mu_x+\th^\vee \preceq \l$, which contradicts the assumption of the lemma. So (b1) is proved.

As a consequence, we have

(b3) $\<\s^d(\th), \mu_x\> \ge 0$.

If $\g=0$, (b3) follows directly from (b2). Otherwise, $\<\g, \th^\vee\>=-1$ and $\<\g, \mu_x\> \le 1$ since $\mu_x$ is $J$-minuscule. So (b3) still follows from (b2).

If $\<w_x(\s^d(\th)), \mu_x\> \le 0$, then $\<\s^d(\th), \mu_x\> = \<w_x(\s^d{\th}), \mu_x\>=0$ by (b3). Hence $\<\g, \mu_x\>=0$ by Lemma \ref{central'} (c). One computes $\<\s^d(\th)+\g, \mu_x+\th^\vee\> \le -1$, which contradicts (b2). Therefore, $\<w_x(\s^d(\th)), \mu_x\> = 1$ by (b1). Combining with (b3), we have $0 \le \<\s^d(\th), \mu_x\> \le 1$. If $\<\s^d(\th), \mu_x\>=0$, then $\<\g, \mu_x\>=1$ by (b2). If $\<\s^d(\th), \mu_x\>=1=\<w_x(\s^d(\th)), \mu_x\>$, then $\<\g, \mu_x\>=0$ by Lemma \ref{central'}. In a word, we have

(b4) $\<w_x(\s^d(\th)), \mu_x\>=\<\s^d(\th)+\g, \mu_x\>=1$.

It remains to show $w_x(\s^d(\th)) = \s^d(\th)+\g$. Assume otherwise. Similarly as (a1.3), we deduce that $w_x(\s^d(\th))=\s^d(\th)$ and $\<\g, \mu_x\>=0$. So $\g \neq 0$ and $\<\s^d(\th), \th^\vee\>=0$. Then $\<w_x(\s^d(\th)), \mu_x+\th^\vee-w_x(\s^r(\th^\vee))\>=\<w_x(\s^d(\th)), \mu_x+\th^\vee\>=1$ by (b4), which implies $$x \overset {(\th, d)} \to {x+\th^\vee-\s^d(\th^\vee)} \overset {(\s^d(\th), r-d)} \to x',$$ which is a contradiction and (b) is proved.

(c) Assume $\<\g, \mu_x\>=0$ and $\g \neq 0$. Then $\<\s^d(\th), \th^\vee\>=0$. By (b), we have $\<\s^d(\th), \mu_x\>=1$, which implies $x \overset {(\th, d)} \to {x+\th^\vee-\s^d(\th^\vee)} \overset {(\s^d(\th), r-d)} \to x'$, a contradiction.

(d) It follows from Corollary \ref{central''} and Lemma \ref{central'}.

(e) First note that $w_x(\s^{r-d}(w_x(\s^d(\th))))=w_x(\s^r(\th))+\s^r(\g)$ by (b) and (a). Assume $\<w_x(\s^{r-d}(w_x(\s^d(\th)))), \mu_x\>=\<w_x(\s^r(\th)), \mu_x\> \le 0$. We claim that $\mu_x+\s^r(\th^\vee) \preceq \l$. If $\<\s^r(\th), \mu_x\> =0$, then $\<w_x(\s^r(\th)), \mu_x\> = \<\s^r(\th), \mu_x\> =0$, which means, by Lemma \ref{central'}, $w_x(\s^r(\th))= \s^r(\th)$ and $\mu_x+\s^r(\th^\vee)=s_{\s^r(\th)}(\mu_x-w_x(\s^r(\th^\vee))) \preceq \l$. If $\<\s^r(\th), \mu_x\> \le -1$, we also have $\mu_x+\s^r(\th^\vee) \preceq \mu_x \preceq \l$ and the claim is proved. By (a), $\<\s^{r-d}(\g)+\s^{r-d}(\th), \mu_x+\s^r(\th^\vee)\>=-1$. Thus $$\mu_x+\s^{r-d}({\tilde \th}^\vee) \preceq \mu_x+\s^r(\th^\vee) \preceq \l,$$ which is a contradiction. So we have $\<w_x(\s^r(\th)), \mu_x\> \ge 1$ as desired.
\end{proof}

\begin{lem} \rm{(cf. \cite[Lemma 4.7.12]{CKV})} \label{z4}
Assume $\co_\a$ is of type II and $\mu_y+\s^k(\tta^\vee) \npreceq \l$ for any $y \in \bar I_{\l, J, b}$ and any $k \in \NN$. If there do not exist $\th \in \co_\a$, $r \in [1, 2d-1]$ and $x' \in \bar I_{\l, J, b}$ satisfying $x \overset {(\th, r)} \to x'$, then

(a) $\<w_x(\s^d(\a)), \mu_x\>=1$;

(b) $w_x(\s^d(\a))=\s^d(\a)+\b$;

(c) $\b=0$ if and only if $\<\b, \mu_x\>=1$;

(d) $\<w_x(\s^d(\a)), \a^\vee\>=-1$.

(e) $\<\a, \mu_x\> \ge -1$.

Here $\b$ is the unique common neighborhood of $\a$ and $\s^d(\a)$ in the Dynkin diagram of $\Pi_J \cup \co_\a$.
\end{lem}
\begin{proof}
First we have $\<w_x(\s^d(\a)), \mu_x\> \le 1$ and (i): $\<\s^d(\a)+\b, \mu_x+\a^\vee\> \ge 0$ as in the proof of Lemma \ref{z3} (b). In particular, $\<\s^d(\a), \mu_x\> \ge 0$. If $\<w_x(\s^d(\a)), \mu_x\> \le 0$, then $\<\s^d(\a), \mu_x\>=\<w_x(\s^d(\a)), \mu_x\>=\<\b, \mu_x\>=0$ by Lemma \ref{central'}, which contradicts (i). So we have $\<w_x(\s^d(\a)), \mu_x\> =1$ and $0 \le \<\s^d(\a), \mu_x\> \le 1$. Now (a), (b) and (c) follow similarly as in the proof Lemma \ref{z3} (b). (d) follows from (b). Note that $\<\tilde \a, \mu_x\>=\<\a, \mu_x\>+1 \ge 0$. So (e) follows.
\end{proof}

\begin{lem} \rm{(cf. \cite[Lemma 4.7.13]{CKV})} \label{z5}
Keep the assumptions of Lemma \ref{z4}. Then $\<w_x(\tilde \a), \a^\vee\> \ge -1$ and $\<w_x(\tilde \a), \mu_x+\a^\vee\> \ge 1$.
\end{lem}
\begin{proof}
First we have $\<w_x(\tilde \a), \a^\vee\> \ge -1$ since $w_x(\tilde \a) \neq -\a$ and $\Phi$ is simply laced.  Note that $\tilde \a$ is $J$-anti-dominant and $J$-minuscule. By \cite[Lemma 4.6.1]{CKV}, there exists an orthogonal subset $D \subset \Phi_J^+$ such that $w_x(\tilde \a) - \tilde \a = \sum_{\g \in D} \g^\vee$. Moreover, $\<\g, \mu_x\>=1$ for each $\g \in D$ by Lemma \ref{central'} (a) (see \cite[page 55 of 66]{CKV}). Note that $\<\tilde \a, \mu_x\>=1+\<\a, \mu_x\> \ge 0$ (by Lemma \ref{z4} (a), (b), (e)) and $\<\tilde \a, \a^\vee\>=1$. We have $$\<w_x(\tilde \a), \mu_x+\a^\vee\>=\<\tilde \a + \sum_{\g \in D} \g^\vee, \mu_x+\a^\vee\> \ge \<\tilde \a, \mu_x + \a^\vee\> \ge 1$$ as desired.
\end{proof}

\begin{proof}[Proof of Proposition \ref{k2'} when $\co_\a$ is of type II] \
Note that $\tilde \a$ is $J$-anti-dominant and $J$-minuscule since $\Phi$ is simply laced. If $\mu_y + \s^{k}(\tta^\vee) \preceq \l$ for some $y \in \bar I_{\l, J. b}$ and $k \in \ZZ$, then $\s^k(\tilde \a) \in C_{\l, J, b, y}$ and $\co_{\tilde \a}$ is of type I by noticing that $y_{\s^k(\tilde \a)}=y_\a \in \pi_1(M_J)$. So the proposition follows from the type I case. We assume from now on that $\mu_y + \s^{k}({\tilde \a}^\vee) \npreceq \l$ for any $y \in \bar I_{\l, J. b}$ and any $k \in \ZZ$.

Case(1): There do not exist $\th \in \co_\a$, $r \in [1, 2d-1]$ and $x' \in \bar I_{\l, J, b}$ satisfying $x \overset {(\th, r)} \to x'$.

Case(1.1): $w_x(\tta) \neq \tta$. Let $\textsl{g}: \PP^1 \to G(L)/K$ be defined as in \cite[\S 4.7.14]{CKV}. Following the computation in loc. cit, we obtain, by Lemma \ref{z1}, Lemma \ref{z4} and Lemma \ref{z5}, that $\eta_J(\textsl{g}(0))-\eta_J(\textsl{g}(\infty))=y_\a$ and $$\textsl{g}(z)\i b \s(\textsl{g}(z)) \in K U_\a(-z t\i) U_{w_x(\tta)}(-c_d' \s^d(y) \s^{2d}(y) t^{\<w_x(\tta), \mu_x\>-1}) t^{\mu_x} K,$$ where $c_d' \in \bold{k}[[t]]^\times$ is some constant. To verify $\textsl{g}\i b \s(\textsl{g}) \subseteq \cup_{\l' \preceq \l} K t^{\l'} K \subseteq K \backslash G(L) / K$, it suffices to show

(a) $U_\a(y t\i) U_{w_x(\tta)}(z t^{\<w_x(\tta), \mu_x\>-1}) t^{\mu_x} \in \cup_{\l' \preceq \l} K t^{\l'} K$ for $y, z \in \bold{k}[[t]]$.

If $\<w_x(\tta), \a^\vee\> \ge 0$, then $U_{w_x(\tta)}$ and $U_{\a}$ commutes and \begin{align*} U_\a(y t\i) U_{w_x(\tta)}(z t^{\<w_x(\tta), \mu_x\>-1}) t^{\mu_x} &= U_{w_x(\tta)}(z t^{\<w_x(\tta), \mu_x\>-1}) U_\a(y t\i) t^{\mu_x} \\ & \in K t^{\mu_x} K \cup K t^{\mu_x+\a^\vee} K, \end{align*} where the inclusion follows from $\<w_x(\tta), \mu_x\> \ge 1$ by Lemma \ref{z5}. Otherwise, we have $\<w_x(\tta), \a^\vee\> =-1$ and $\<w_x(\tta), \mu_x\> \ge 2$ by Lemma \ref{z5}. One checks that \begin{align*} U_\a(y t\i) U_{w_x(\tta)}(z t^{\<w_x(\tta), \mu_x\>-1}) t^{\mu_x} & \in K U_{w_x(\tta)+\a}(c y z t^{\<w_x(\tta), \mu_x\>-2}) U_\a(z t\i) t^{\mu_x} \\ & \subseteq K U_\a(z t\i) t^{\mu_x} \\ & \subseteq K t^{\mu_x} K \cup K t^{\mu_x+\a^\vee} K,\end{align*} where $c \in \bold{k}[[t]]$ is some constant. Therefore, (a) is proved.

Case(1.2): $w_x(\tta) = \tta$. If $\<\a, \mu_x\>=-1$, the proposition follows from the construction in \cite[\S 4.7.16]{CKV}. Otherwise, $\<\tta, \mu_x\> \ge 1$ by Lemma \ref{z4} (b) (e), and let $\textsl{g}: \PP^1 \to G(L)/K$ be defined as in \cite[\S 4.7.14]{CKV}. Following the computation in loc. cit, we obtain, by Lemma \ref{z1} and Lemma \ref{z4}, that $\eta_J(\textsl{g}(0))-\eta_J(\textsl{g}(\infty))=y_\a$ and \begin{align*} \textsl{g}(z)\i b \s(\textsl{g}(z)) &= K U_\a(-z t\i) U_{\tta}(-c_d' \s^d(z) \s^{2d}(z) t^{\<\tta, \mu_x\>-1}) t^{\mu_x} K \\ &= K U_\a(-z t\i) t^{\mu_x} K \in K t^{\mu_x} K \cup K t^{\mu_x+\a^\vee} K,\end{align*} where $c_d' \in \bold{k}[[t]]^\times$ is some constant, and the second equality follows from that $U_\a$ commutes with $U_{\tta}$ (since $\<\tta, \a^\vee\>=1$).

Case(2): $x \overset {(\th, r)} \to x'$ for some $\th \in \co_\a$, $r \in [1, 2d-1]$ and $x' \in \bar I_{\l, J, b}$. Similar to the Case(1) in the proof of the type I case, it suffices to show that

(b) there exist $P \in X_{\mu_x}^{M_J}(b)$ and $Q \in X_{\mu_{x'}}^{M_J}(b)$ such that $P \sim_{\l, b} Q$ and $\eta_J(P)-\eta_J(Q)=\sum_{i=0}^{r-1} \s^i(\th^\vee)$.

By Lemma \ref{convv'}, we may assume $x \overset {(\th, r)} \rightarrowtail x'$. If $r \in [1, d]$, (b) follows from Proposition \ref{main}. Assume $r \in [d+1, 2d-1]$. Let $\textsl{g}: \PP^1 \to G(L)/K$ be as in the proof of \cite[Lemma 4.7.2]{CKV}. Following the computation in \cite[\S 4.7.5]{CKV}, we obtain by Lemma \ref{z3} that $\eta_J(\textsl{g}(0))-\eta_J(\textsl{g}(\infty))=\sum_{i=0}^{r-1} \s^i(\th^\vee)$ and \begin{align*}& \quad \ \textsl{g}(z)\i b \s(\textsl{g}(z)) \\ & =K U_\a(-z t\i) U_{\s^{r-d}(\a)}(-c_{r-d} \s^{r-d}(z) t\i) U_{w_x(\s^r(\a))+\s^r(\g)}(c_r t^{\<w_x(\s^r(\a))+\s^r(\g), \mu_x\>}) \\ &\quad \times U_{\s^{r-d}(\a)}(c_{r-d} \s^{r-d}(z) t\i) t^{\mu_x} \dot w_x K \\ & \in \cup_{\l' \preceq \l} K t^{\l'} K,\end{align*} where $c_{r-d}, c_r \in \bold{k}[[t]]^\times$ are some constants and the last inclusion follows from Lemma \ref{z3} (e).
\end{proof}

\begin{lem} \rm{(cf. \cite[Lemma 4.7.6]{CKV})} \label{z6}
Assume $\co_\a$ is of type III and $\mu_y+\s^k(\tta^\vee) \npreceq \l$ for any $y \in \bar I_{\l, J, b}$ and any $k \in \NN$. If there exist $\th \in \co_\a$, $r \in [2d, 3d-1]$ and $x' \in \bar I_{\l, J, b}$ such that $x \overset {(\th, r)} \rightarrowtail x'$, then

(a) If $r=2d$, then $\<\s^d(\th), \mu_x\>=0$ and $\<\s^{2d}(\th), \mu_x\> \ge 1$;

(b) If $r \in [2d+1, 3d-1]$, then $\<\g, \mu_x\>=1$, $\<\s^r(\g), \mu_x\>=0$, $\<\s^r(\th), \mu_x\> \ge 1$ and $\<\s^i(\th), \mu_x\>=0$ for $i \in \{d, 2d, r-d, r-2d\}$;

(c) $w_x(\s^i(\th))=\s^i(\th)$ and $\<\s^i(\g), \mu_x\>=\<\s^i(\th), \mu_x\>=0$ for $i \in [1, r-1]-\{d, 2d, r-d, r-2d\}$.

Here $\g \in \Pi_J$ is the common neighbor of $\th$, $\s^d(\th)$ and $\s^{2d}(\th)$.
\end{lem}
\begin{proof}
Note that $\Phi=\Phi_{J, \a}$, whose connected components are of type $D_4$. Thus $\{\th, \s^d(\th), \s^{2d}(\th)\}$ is an orthogonal set. We set $\tilde \th=\th+\s^d(\th)+\s^{2d}(\th)+\g$.

(a) First we claim that

(i) $\mu_x+\th^\vee-\s^{2d}(\th^\vee) \preceq \l$.

Indeed, one checks that $\mu_x+\th^\vee-\s^{2d}(\th^\vee)= \mu_x+\th^\vee-w_x(\s^{2d}(\th^\vee)) \preceq \l$ if $\<\g, \mu_x\>=0$ and $\mu_x+\th^\vee-\s^{2d}(\th^\vee)=s_\g(\mu_x+\th^\vee-w_x(\s^{2d}(\th^\vee))) \preceq \l$ if $\<\g, \mu_x\>=1$. So (i) is proved.

If $\<\s^d(\th), \mu_x\> \le -2$, then $x \overset {(\s^d(\th), d)} \to {x+\s^d(\th^\vee)-\s^{2d}(\th^\vee)} \overset {(\th, d)} \to x'$, which contradicts the assumption of the lemma. So $\<\s^d(\th), \mu_x\> \ge -1$. Similarly, using (i) we have $\<\s^d(\th), \mu_x\> \le 0$. Assume $\<\s^d(\th), \mu_x\>=-1$. If $\<\s^{2d}(\th), \mu_x\> \le 0$, then $\<\s^d(\th)+\s^{2d}(\th)+\g, \mu_x+\th^\vee\> \le -1$ and $\mu_x+{\tilde \th}^\vee \preceq \mu_x+\th^\vee \preceq \l$, a contradiction. So we have $\<\s^{2d}(\th), \mu_x\> \ge 1$ and $\mu_x-\s^{2d}(\th^\vee)+\s^d(\th^\vee) \preceq \mu_x-\s^{2d}(\th^\vee) \preceq \l$. Then $x \overset {(\s^d(\th), d)} \to {x+\s^d(\th^\vee)-\s^{2d}(\th^\vee)} \overset {(\th, d)} \to x'$, a contradiction. Thus $\<\s^d(\th), \mu_x\>=0$.

Assume $\<\s^{2d}(\th), \mu_x\> \le 0$. Then $\<\s^{2d}(\th), \mu_x+\th^\vee+\s^d(\th^\vee)+\g^\vee\> \le -1$ and $\<\s^{d}(\th)+\s^{2d}(\th)+\g, \mu_x+\th^\vee-\s^{2h}(\th^\vee)\> \le -1$, which implies $$\mu_x+{\tilde \th}^\vee \preceq \mu_x+\th^\vee+\s^d(\th^\vee)+\g^\vee \preceq \mu_x+\th^\vee-\s^{2h}(\th^\vee) \preceq \l,$$ which is a contradiction. So $\<\s^{2d}(\th), \mu_x\> \ge 1$ as desired.

(b) Similar to the proof of (a), we have $\<\s^d(\th), \mu_x\>=\<\s^{2d}(\th), \mu_x\> = 0$. If $\<\g, \mu_x\>=0$, then $\<\s^d(\th)+\s^{2d}(\th)+\g, \mu_x + \th^\vee\> =-1$, which means $\mu_x+{\tilde \th}^\vee \preceq \mu_x+\th^\vee \preceq \l$, a contradiction. So $\<\g, \mu_x\>=1$.

Assume $\<\s^r(\g), \mu_x\>=1$. If $\<\s^{r-d}(\th), \mu_x\> \ge 0$, then $\<w_x(\s^{r-d}(\th)), \mu_x\>=\<\s^{r-d}(\th)+\s^r(\g), \mu_x\> \ge 1$ and $$\mu_x+\th^\vee-w_x(\s^r(\th^\vee))-w_x(\s^{r-d}(\th^\vee)) \preceq \mu_x+\th^\vee-w_x(\s^r(\th^\vee)) \preceq \l,$$ which implies $x \overset {(\th, r-d)} \to {x+\th^\vee-\s^{r-d}(\th^\vee)} \overset {(\s^{r-d}(\th), d)} \to x'$, a contradiction. So $\<\s^{r-d}(\th), \mu_x\> \le -1$. Similarly, $\<\s^{r-2d}(\th), \mu_x\> \le -1$. If $\<\s^r(\th), \mu_x\> \ge 1$, then one checks that $x \overset {(\s^{r-d}(\th), d)} \to {x+\s^{r-d}(\th^\vee)-\s^r(\th^\vee)} \overset {(\th, r-d)} \to x'$, a contradiction. So $\<\s^r(\th), \mu_x\> \le 0$ and $\<\mu_x, \s^r(\tilde \th)\>=-1$, which is a contradiction. Therefore, the case $\<\s^r(\g), \mu_x\>=1$ dose not occur.

Now we have $\<\s^r(\g), \mu_x\>=0$ and $w_x(\s^i(\th))=\s^i(\th)$ for $i \in r+\ZZ d$. Then one deduces that $\<\s^{r-d}(\th), \mu_x\>=\<\s^{r-2d}(\th), \mu_x\>=0$ as in the proof of (a). If $\<\s^r(\th), \mu_x\> \le 0$, then $$\<\s^r(\tilde \th), \mu_{x'}\>=\<\s^r(\tilde \th), \mu_x-\s^r(\th^\vee)\> \le -1,$$ a contradiction. So we have $\<\s^r(\th), \mu_x\> \ge 1$ as desired.

(c) It follows from Corollary \ref{convv'} and Lemma \ref{central'}.
\end{proof}

\begin{lem} \rm{(cf. Case 3 in the proof of \cite[Lemma 4.7.10]{CKV})}\label{z7}
Assume $\co_\a$ is of type III and $\mu_y+\s^k(\tta^\vee) \npreceq \l$ for any $y \in \bar I_{\l, J, b}$ and any $k \in \NN$. If there do not exist $\th \in \co_\a$, $r \in [1, 3d-1]$ and $x' \in \bar I_{\l, J, b}$ satisfying $x \overset {(\th, r)} \to x'$, then $\<\b, \mu_x\>=1$, $\<\s^d(\a), \mu_x\>=\<\s^{2d}(\a), \mu_x\>=0$ and $\<\a, \mu_x\> \ge -1$. Here $\b$ is the common neighborhood of $\a$, $\s^d(\a)$ and $\s^{2d}(\a)$ in the Dynkin diagram of $\Pi_J \cup \co_\a$.
\end{lem}
\begin{proof}
As in the proof of Lemma \ref{z6}, we have $\<\s^d(\a), \mu_x\>, \<\s^{2d}(\a), \mu_x\> \le 0$. If $\<\s^d(\a)+\s^{2d}(\a), \mu_x\> \le -1$, then $\<\s^d(\a)+\s^{2d}(\a)+\b, \mu_x+\a^\vee\> \le -1$ and $\mu_x+\tta^\vee \preceq \mu_x+\a^\vee \preceq \l$, a contradiction. So $\<\s^d(\a), \mu_x\>=\<\s^{2d}(\a), \mu_x\>=0$. Again, by the inequality $\<\s^d(\a)+\s^{2d}(\a)+\b, \mu_x+\a^\vee\> \ge 0$, we have $\<\b, \mu_x\>=1$. Finally, we have $\<\a, \mu_x\> \ge -1$ since $\<\tta, \mu_x\> \ge 0$.
\end{proof}

\begin{proof}[Proof of Proposition \ref{k2'} when $\co_\a$ is of type III] \

If $\mu_y + \s^{k}(\tta^\vee) \preceq \l$ for some $y \in \bar I_{\l, J, b}$ and $k \in \ZZ$, the proposition follows similarly as in the type II case. We assume from now on that $\mu_y + \s^{k}({\tilde \a}^\vee) \npreceq \l$ for any $y \in \bar I_{\l, J, b}$ and any $k \in \ZZ$.

Case(1): There do not exist $\th \in \co_\a$, $r \in [1, 3d-1]$ and $x' \in \bar I_{\l, J, b}$ satisfying $x \overset {(\th, r)} \to x'$. Let $\textsl{g}: \PP^1 \to G(L)/K$ be defined as \cite[page 60]{CKV}. Following the computation in loc. cit., we obtain, by Lemma \ref{z1} and Lemma \ref{z7}, that $\eta_J(\textsl{g}(0))-\eta_J(\textsl{g}(\infty))=y_\a$ and \begin{align*} & \quad \ \textsl{g}\i(z) b \s(\textsl{g}(z)) \\ &= K U_{\a+\b}(c_{3d} \s^{3d}(z) t^{\<\a+\b, \mu_x\>}) \\ & \quad \ \times U_{\a+\s^d(\a)+\s^{2d}(\a)+2\b}(-c''\s^d(z)\s^{2d}(z)\s^{3d}(z)t^{\<\a+\b, \mu_x\>}) \\ & \quad \ \times U_{\a+\s^{2d}(\a)+\b}(c'\s^{2d}(z)\s^{3d}(z)t^{\<\a+\b, \mu_x\>}) U_\a(-z t\i) t^{\mu_x} K \\ & = K U_\a(-z t\i) t^{\mu_x} K \in \cup_{\l' \preceq \l} K t^{\l'} K \subseteq K \backslash G(L) / K, \end{align*} where $c_{3d}, c', c'' \in \bold{k}[[t]]$ are some constants, and the second inclusion follows from $\<\a+\b, \mu_x\> \ge 0$ by Lemma \ref{z7}.

Case(2): $x \overset {(\th, r)} \to x'$ for some $\th \in \co_\a$, $r \in [1, 3d-1]$ and $x' \in \bar I_{\l, J, b}$. As in the type I case, it suffices to show that

(a) there exist $P \in X_{\mu_x}^{M_J}(b)$ and $Q \in X_{\mu_{x'}}^{M_J}(b)$ such that $\eta_J(P)-\eta_J(Q)=\sum_{i=0}^{r-1} \s^i(\th^\vee)$.

By Lemma \ref{convv'}, we may assume $x \overset {(\th, r)} \rightarrowtail x'$. If $r \in [1, 2d-1]$, (a) follows from Proposition \ref{main}. Now we assume $r \in [2d,  3d-1]$. Let $\textsl{g}: \PP^1 \to G(L)/K$ be defined by $$\textsl{g}(z)=g_x {}^{(\dot b_x \s)^{r-1}} U_\th(z t\i) {}^{(\dot b_x \s)^{r-2}} U_\th(z t\i) \cdots U_\th(z t\i) K,$$ where $g_x \in M_J(L)$ satisfying $g_x\i b \s(g_x)=\dot b_x=t^{\mu_x} \dot w_x$ as usual. Then one computes that $\eta_J(\textsl{g}(0))-\eta_J(\textsl{g}(\infty))=\sum_{k=0}^{r-1} \s^k(\th^\vee)$. It remains to show $\textsl{g}\i b \s(\textsl{g}) \subseteq \cup_{\l' \preceq \l} K t^{\l'} K \subseteq K \backslash G(L) / K$.

Case(2.1): $r=2d$. If $\<\g, \mu_x\>=1$, by Lemma \ref{z1} and Lemma \ref{z6} (a), we have \begin{align*} & \quad \ \textsl{g}\i(z) b \s(\textsl{g}(z)) \\ & =K U_\th(-z t\i) U_{\s^{2d}(\th)+\s^d(\th)+\g}(c \s^d(z) \s^{2d}(z) t^{\<\s^{2d}(\th), \mu_x\>}) U_{\s^{2d}(\th)}(c_{2d} \s^{2d}(z) t^{\<\s^{2d}(\th), \mu_x\>}) t^{\mu_x} \dot w_x K\\ & = K U_\th(-z t\i) t^{\mu_x} K \in \cup_{\l' \preceq \l} K t^{\l'} K .\end{align*} If $\<\g, \mu_x\>=0$, then \begin{align*}\textsl{g}\i(z) b \s(\textsl{g}(z)) & =K U_\th(-z t\i) U_{\s^{2d}(\th)}(c_{2d} \s^{2d}(z) t^{\<\s^{2d}(\th), \mu_x\>-1}) t^{\mu_x} \dot w_x K \\ & = K U_\th(-z t\i) t^{\mu_x} K \in \cup_{\l' \preceq \l} K t^{\l'} K .\end{align*} Here $c, c_{2d} \in \bold{k}[[t]]$ are some constants; the second equality in both cases follows from $\<\s^{2d}(\th), \mu_x\> \ge 1$ by Lemma \ref{z6} (a).

Case(2.2): $r \in [2d+1, 3d-1]$. Then we have (see \cite[page 53]{CKV}) \begin{align*} & \quad \ \textsl{g}\i(z) b \s(\textsl{g}(z)) \\ & = K U_\th(-z t\i) U_{\s^r(\th)+\s^{r-d}(\th)+\s^{r-2d}(\th)+\s^r(\g)}(c'' \s^r(z)\s^{r-d}(z)\s^{r-2d}(z) t^{\<\s^r(\th), \mu_x\>-1}) \\ & \quad \ \times U_{\s^r(\th)+\s^{r-d}(\th)+\s^r(\g)}(c' \s^r(z)\s^{r-d}(z) t^{\<\s^r(\th), \mu_x\>}) U_{\s^r(\th)}(c_r \s^r(z) t^{\<\s^r(\th), \mu_x\>} ) t^{\mu_x} \dot w_x K \\ & = K U_\th(-z t\i) t^{\mu_x} K \in \cup_{\l' \preceq \l} K t^{\l'} K,\end{align*} where $c', c'' ,c_r \in \bold{k}[[t]]$ and the second equality follows from $\<\s^r(\th), \mu_x\> \ge 1$ by Lemma \ref{z6} (b).
\end{proof}

\end{document}